\documentclass[12pt]{amsart}
\usepackage{amsmath}
\usepackage{amsxtra}
\usepackage{amstext}
\usepackage{amssymb}
\usepackage{amsthm}
\usepackage{latexsym}
\usepackage{dsfont} % for \mathbf{N}
\usepackage{verbatim}
\usepackage{tabls}
\usepackage{graphicx}
\usepackage{rotating}
\usepackage{caption}
\theoremstyle{plain}
\newtheorem{thm}{Theorem}[section]
\newtheorem{cor}[thm]{Corollary}
\newtheorem{lem}[thm]{Lemma}
\newtheorem{prop}[thm]{Proposition}

\newtheorem*{thmb}{Conjecture}

\theoremstyle{definition}
\newtheorem{defn}[thm]{Definition}

\numberwithin{equation}{section}
\def\XXint#1#2#3{{\setbox0=\hbox{$#1{#2#3}{\int}$}
     \vcenter{\hbox{$#2#3$}}\kern-.5\wd0}}

\begin{document}

\title[Riesz transforms and exponential potentials]
{The fractional Riesz transform and an exponential potential}
\author{Benjamin Jaye}
\address{Benjamin Jaye, Department of Mathematics,
Kent State University,
Kent, OH 44240}
\email{bjaye@kent.edu}
\author{Fedor Nazarov}
\address{Fedor Nazarov, Department of Mathematics,
Kent State University,
Kent, OH 44240}
\email{nazarov@math.kent.edu}
\author{Alexander Volberg}
\address{Alexander Volberg, Department of Mathematics,
Michigan State University,
East Lansing, MI 48824}
\email{volberg@math.msu.edu}

\date{\today}
\begin{abstract}In this paper we study the $s$-dimensional Riesz transform of a finite measure $\mu$ in $\mathbf{R}^d$, with $s\in (d-1,d)$.  We show that the boundedness of the Riesz transform of $\mu$ yields a weak type estimate for the Wolff potential $\mathcal{W}_{\Phi,s}(\mu)(x) = \int_0^{\infty}\Phi\bigl(\frac{\mu(B(x,r))}{r^s}\bigl) \frac{dr}{r},$ where $\Phi(t) = e^{-1/t^{\beta}}$ with $\beta>0$ depending on $s$ and $d$.  In particular, this weak type estimate implies that $\mathcal{W}_{\Phi,s}(\mu)$ is finite $\mu$-almost everywhere.

As an application, we obtain an upper bound for the Calder\'{o}n-Zygmund capacity $\gamma_s$ in terms of the non-linear capacity associated to the gauge $\Phi$.  It appears to be the first result of this type for $s>1$. \end{abstract}


\maketitle

\section{Introduction}For an integer $d\geq 2$, let $s\in(d-1,d)$.  Define the $s$-dimensional Riesz transform of a finite nonnegative Borel measure $\mu$ by
$$R(\mu)(x) = \int_{\mathbf{R}^d}\frac{y-x}{|y-x|^{1+s}}d\mu(y).
$$
For any finite measure $\mu$, the integral defining $R(\mu)$ converges almost everywhere with respect to the Lebesgue measure in $\mathbf{R}^d$.  The aim of this paper is to show that the boundedness of the Riesz transform of a measure $\mu$ implies the $\mu$-almost everywhere finiteness of the Wolff potential associated to an exponential gauge.  More precisely, we obtain a (very) weak type estimate for such a potential.

Define the measure $\mathcal{L}$ on $(0,\infty)$ by $\mathcal{L}(E) = \int_{E} \frac{dr}{r}$ for $E\subset (0,\infty)$. For each $x\in \mathbf{R}^d$ and $\Delta\in (0,\infty)$, we denote $$E(x,\Delta) = \Bigl\{r\in(0,\infty)\!:\! \frac{\mu(B(x,r))}{r^s}\! >\! \Delta\!\Bigl\}.$$
Here $B(x,r)$ is the open ball of radius $r$, centred at $x$.
Let $||\cdot||_{L^{\infty}}$ be the essential supremum norm with respect to the Lebesgue measure in $\mathbf{R}^d$.  Our main result is the following theorem.

\begin{thm}\label{thm1} Suppose that
$||R(\mu)||_{L^{\infty}}\leq 1$.  There exist positive constants $C$ and $\alpha$, depending on $s$ and $d$, such that
\begin{equation}\label{weaktype}\begin{split}
\mu\bigl(\bigl\{x\in \mathbf{R}^d:\mathcal{L}(E(x, \Delta))>T\bigl\}\bigl)\leq \frac{C\mu(\mathbf{R}^d)}{\Delta \log^{\alpha} T},
\end{split}\end{equation}
%\begin{equation}\label{weaktype}\begin{split}
%\mu\Bigl(\!\Bigl\{x\in \mathbf{R}^d\!:\!\frac{dr}{r}\Bigl(&\!\Bigl\{r\in(0,\infty)\!:\! \frac{\mu(B(x,r))}{r^s}\! >\! \Delta\!\Bigl\}\!\Bigl)>\!T\!\Bigl\}\!\Bigl)\leq \!\frac{C\!\mu(\mathbf{R}^d)}{\Delta \log^{\alpha} \!T},
%\end{split}\end{equation}
for all $0<\Delta<\infty$ and $e<T<\infty$.
\end{thm}

A consequence of Theorem \ref{thm1} is that the condition $||R(\mu)||_{L^{\infty}}\leq 1$ implies that
\begin{equation}\label{exppot}\int_0^{\infty}\!\!\Phi\Big(\frac{\mu(B(x,r))}{r^s}\Bigl)\frac{dr}{r}<\infty \text{ for }\mu\text{-almost every }x\in \mathbf{R}^d,
\end{equation}
where $\Phi(t) = e^{-1/t^{\beta}}$ with any $\beta>1/\alpha$.
The estimate (\ref{exppot}) is strong enough to deduce that the Calder\'{o}n-Zygmund capacity associated to the $s$-dimensional Riesz transform is dominated by the non-linear capacity associated to $\Phi$, as we will see in Section \ref{exppotsec} below.

The almost everywhere finiteness of an exponential potential is substantially weaker than the well-known conjecture (see \cite{ENV10, Tol11}), which states that for any finite measure $\mu$,
\begin{equation}\label{strongtype}||R(\mu)||_{L^{\infty}}\leq 1 \implies \int_{\mathbf{R}^d}\int_0^{\infty}\Big(\frac{\mu(B(x,r))}{r^s}\Bigl)^2\frac{dr}{r}d\mu(x)<\infty.
\end{equation}
In \cite{MPV05}, Mateu, Prat, and Verdera proved (\ref{strongtype}) in the range $0<s<1$ by using curvature methods, but it is not known whether this result should continue to hold if $s>1$ and $s\not\in \mathbf{N}$.  Any such bound (even Theorem \ref{weaktype} above) is false in the case of integer $s$.   In the special case of measures supported on Cantor sets with certain additional geometric properties, the conjecture (\ref{strongtype}) has been proven for all $s$, see \cite{Tol11, EV11}.

For a general measure $\mu$, the result here appears to be the first to show that a positive potential of any type can be controlled by the Riesz transform with $s$ outside the curvature range. It can be viewed as a quantitative version of the recent theorem of Eiderman, Nazarov and Volberg \cite{ENV11}.  Recall that a measure $\mu$ is called \textit{totally lower irregular} if
\begin{equation}\label{lowirreg}\liminf_{r\rightarrow 0}\frac{\mu(B(x,r))}{r^s} = 0, \text{ for }\mu\text{-almost every } x\in \mathbf{R}^d.
\end{equation}
In \cite{ENV11}, the nonexistence is proved of a finite totally lower irregular measure $\mu$, supported on a set of finite $s$-dimensional Hausdorff measure, such that $\mu$ has a bounded Riesz transform.

A careful inspection of the proof in \cite{ENV11} reveals the primary qualitative step in the argument to be precisely the use of the condition (\ref{lowirreg}), which is used in a Cantor construction in order to obtain `almost orthogonality' of partial Riesz transforms associated to different Cantor levels.

In order to find a quantitative substitute for (\ref{lowirreg}), we revisit a very nice theorem of Vihtil\"{a}.  In \cite{V96}, the nonexistence is proved of a nontrivial measure $\mu$ with bounded Riesz transform, which has \textit{positive lower density}, that is
\begin{equation}\label{lowdensity}
\liminf_{r\rightarrow 0}\frac{\mu(B(x,r))}{r^s} > 0, \text{ for }\mu\text{-almost every } x\in \mathbf{R}^d.
\end{equation}
The result in \cite{V96} is proved for all $s\in (0,d)$, $s\not\in \mathbf{N}$.  In this paper we restrict our attention to $s\in (d-1,d)$.  This restriction is perhaps not so important for getting a quantitative version of Vihtil\"{a}'s theorem.  However, in another part of the argument we will make use of a certain maximum principle for the Riesz transform with $s\geq d-1$ (see Proposition \ref{maxp}), and we do not know if some analogue of this result is available for $s<d-1$. % This is very closely related to the validity of a weak maximum principle for the gradient of solutions to higher order partial differential equations. There is a large literature on such results (often called the Agmon-Miranda maximum principle), but these do not appear applicable in our set-up.

The general idea of our paper is to use multi-scale analysis to show that the Riesz transform of a measure $\mu$ is large provided $\mu$ possesses many scales of significant  density. We will then marry this with the fractal construction in \cite{ENV11}.   It was somewhat surprising that this process should estimate a positive potential, even one as weak as in (\ref{exppot}).

The result of \cite{V96} leans heavily on the theory of tangent measures.   By their definition as weak limits, tangent measures carry little quantitative information.  Therefore our first task is to derive a quantitative version of Vihtil\"{a}'s argument.  Since tangent measures have found several applications in the field of geometric measure theory (see for example \cite{Mat95b, MP95}), this may be of interest to specialists.

We remark that multi-scale methods are somewhat notorious for giving exponential (or logarithmic) dependence as in (\ref{weaktype}), even in those cases when the true dependence should be a power one; cf. \cite{Tao08} and \cite{NPV10}.  The bound here is therefore no indication the conjectured estimate (\ref{strongtype}) is false.  On the contrary, it may be viewed as further evidence to support the validity of (\ref{strongtype}).  We also do not rule out that the methods here could be improved to yield a power bound in the scale counting parameter $T$ in (\ref{weaktype}).  In order to obtain such an improvement, the bounds of Proposition \ref{qualv} below would have to be significantly strengthened.

\subsection{The plan of the paper} After a discussion of the preliminaries in Section \ref{prelim}, the paper splits into two almost independent parts.  In the first part (Section \ref{vihsec}), we develop the quantitative version of Vihtil\"{a}'s theorem.  That is formulated in Proposition \ref{qualv} below.  This proposition is the only thing used in the second part, which is devoted to proving Theorem \ref{thm1}.

Assuming $\mathcal{L}(E(x,\Delta))$ is large on a noticeable set, we construct a certain Cantor type set.  This is carried out in Section \ref{Cantorsec}.  Section \ref{l2ests} begins with three $L^2$ estimates, from which we derive a contradiction and hence prove Theorem \ref{thm1}.  The remainder of Section \ref{l2ests}, together with Section \ref{prop3sec} are devoted to proving these three estimates.  In Section \ref{exppotsec}, we conclude the paper with a brief discussion of the relationship between the Cald\'{e}ron-Zygmund capacity and the non-linear Wolff capacity associated to an exponential gauge.

\subsection{Acknowledgement}  We are very thankful to  Vladimir Eiderman for initiating this project, and for the many insightful remarks he has made which have helped shape the proof.   Should the reader have a positive opinion of this paper, they should certainly view Professor Eiderman as a co-author (despite his insistence to the contrary).

\section{Preliminaries}\label{prelim}

\subsection{Notation}  In what follows $C$, $c$, or $C_j$, $c_j$ (for $j\in \mathbf{N}$) are respectively large and small positive constants depending on $s$ and $d$.  We enumerate them so that the constant with index $j$ can be chosen in terms of constants with lower indices  (for example $C_{96}$ can depend on $c_{95}$ and $C_{4}$).  Within a specific argument, if a constant $C$ or $c$ does not have an index, then it may depend on all numbered constants chosen up to that moment, and can change from line to line.  At the very least, every large constant is greater than $1$, and every small constant is less than $1$.

%We will also use the notation $A\ll B$ to mean $A<c_0 B$ where $c_0=c_0(s,d)>0$ is a sufficiently small positive constant (its choice does not depend on any other constants in the paper and can be made at the very beginning).  Every time this notation is used, it should be read as ``the following argument is true, provided that $c_0$ was chosen small enough''.  The statement $A\gg B$ is equivalent to $B\ll A$.

Throughout the paper, $m_d$ will denote the $d$-dimensional Lebesgue measure.  Given a function $f$, either scalar or vector valued, $||f||_{L^{\infty}}$ will always stand for the essential supremum norm of $f$ with respect to $m_d$.  The quantity $\text{osc}_{E}(f) = \sup\{|f(x) -f(y)|: x,y\in E\}$ will be called the oscillation of $f$ over the set $E\subset \mathbf{R}^d$.

We adopt the standard notation that $B(x,r)$ is an \textit{open} ball of radius $r$, centred at $x$.  The $\varepsilon$ neighbourhood of a set $E$ shall refer to the \textit{open} neighbourhood $\{y\in \mathbf{R}^d : |y-x|<\varepsilon \text{ for some }x\in E\}$.   We denote the closure of $E$ by $\overline{E}$.

We denote by $\mathbf{N}$ the set of natural numbers $\{1,2,3,4,\dots\}$, and by $\mathbf{Z}_+$ the set of nonnegative integers $\mathbf{N}\cup\{0\}$.

\subsection{Growth conditions and $L^2(\mu)$ boundedness} In this section we will mention the key facts concerning the s-dimensional Riesz transform that will be used in what follows.  First of all, we will make regular use of the following necessary condition for the boundedness of the Riesz transform.
\begin{lem}\label{neccond}  Suppose $||R(\mu)||_{L^{\infty}}\leq 1$.  There is a constant $C_1$ such that, for any ball $B(x,r)\subset \mathbf{R}^d$, one has
\begin{equation}\label{growth}\mu(B(x,r))\leq C_1 r^s.
\end{equation}
\end{lem}
Lemma \ref{neccond} can be proved by elementary Fourier analysis, see \cite{MPV05, ENV11}.   The next result we will require is a suitable version of Cotlar's lemma.  Define the maximal Riesz transform $R^{\#}(\mu)$ by
$$R^{\#}(\mu)(x) = \sup_{B:\, x\in B}\Bigl|\int_{\mathbf{R}^d\backslash 2B} \frac{y-x}{|y-x|^{1+s}} d\mu(y)\Bigl|,
$$
where the supremum is taken over all balls $B$ such that $x\in B$.  Here (and elsewhere) $2B$ is the concentric double of $B$.  The following lemma can be proved by mimicking the simple argument of Lemma 3 in \cite{V96}.

\begin{lem}\label{cotlar}  Suppose that $||R(\mu)||_{L^{\infty}}\leq 1$.  There is a constant $C_2$ such that
$$|R^{\#}(\mu)(x)|\leq C_2, \text{ for all }x\in \mathbf{R}^d.
$$
\end{lem}

Lemmas \ref{neccond} and \ref{cotlar} ensure that the $s$-dimensional maximal Riesz transform together with the measure $\mu$ satisfy the hypotheses of the $T(1)$-theorem of \cite{NTV03}, which is the next result we will state.
\begin{thm}[$T(1)$-Theorem]\label{T1thm}  Suppose $||R(\mu)||_{L^{\infty}}\leq 1$.  There is a constant $C_3$ such that
\begin{equation}\label{T1}
\int_{\mathbf{R}^d}|R^{\#}(f\mu)|^2 d\mu \leq C_3 \int_{\mathbf{R}^d} |f|^2 d\mu, \text{ for all }f\in L^2(\mu).
\end{equation}
\end{thm}
Theorem \ref{T1thm} is a special case of the $T(b)$-theorem in \cite{NTV03}.   For our purposes Lemma \ref{neccond} and Theorem \ref{T1thm} are especially useful since they are \textit{hereditary} in the measure $\mu$ - if we restrict the measure to any subset, then the conditions continue to hold with the same constants.  This will allow us great flexibility when constructing the Cantor set.  This hereditary property is not true in general for the  $L^{\infty}$ bound.

We will also need an analogue of (\ref{T1}) for the adjoint Riesz transform, which is defined for a vector valued Borel measure $\nu$ by
$$R^{*}(\nu)(x) = -\int_{\mathbf{R}^d}\frac{y-x}{|y-x|^{1+s}}\cdot d\nu(y).
$$
The maximal adjoint Riesz transform is then given by
$$(R^*)^{\#}(\nu)(x) =  \sup_{B:\, x\in B}\Bigl|\int_{\mathbf{R}^d\backslash 2B} \frac{y-x}{|y-x|^{1+s}} \cdot d\nu(y)\Bigl|.
$$
Let $f=(f_1,\dots, f_d)$ be a vector field in $L^2(\mu)$.  For any ball $B$ and $x\in B$, note that
$$\Bigl|\int_{\mathbf{R}^d\backslash 2B} \frac{y-x}{|y-x|^{1+s}} \cdot f(y) d\mu(y)\Bigl|\leq \sum_{j=1}^d\Bigl|\int_{\mathbf{R}^d\backslash 2B} \frac{y-x}{|y-x|^{1+s}} f_j(y) d\mu(y)\Bigl|.
$$
Therefore, we have $[(R^*)^{\#}(f\mu)]^2\leq d\sum_{j=1}^{d}[R^{\#}(f_j\mu)]^2$, and Theorem \ref{T1thm} yields
\begin{equation}\label{T1adj}\int_{\mathbf{R}^d}|(R^*)^{\#}(f\mu)|^2 d\mu \leq C_3d \int_{\mathbf{R}^d} |f|^2 d\mu.
\end{equation}

%However, if one restricts the measure to a very regular set (for instance a ball), then control of the $L^{\infty}$ norm can be maintained.  This is the content of the following localisation principle, which is Lemma 3.1 of \cite{MPV05}.

%\begin{lem}\label{local}\cite{MPV05}  Suppose  $||R(\mu)||_{L^{\infty}}\leq 1$, and let $P$ be a ball.  Consider a function $\varphi\in C^{\infty}_0(\mathbf{R}^d)$ such that $\varphi\equiv 1$ on $P$, $\varphi\in C^{\infty}_0(2P)$ and $||\nabla %\varphi||_{L^{\infty}} \leq 1/\ell(P)$.  Then:
 %$$||R(\varphi d\mu)||_{L^{\infty}}\leq C_2.
 %$$
%\end{lem}

%This nice lemma allows us to zoom in on particular scales in our multiscale argument.

\subsection{The action on the Fourier side}\label{fourier}  We conclude the preliminaries by recapping how the $s$-dimensional Riesz transform acts on the Fourier side.  All these properties can be easily derived using Fourier analysis, see for example \cite{SW71}.  First note that there exists a constant $b=b(s,d)\in \mathbf{R}\backslash \{0\}$ such that, for any $f$ in the Schwartz class and $\xi\in \mathbf{R}^d$,
\begin{equation}\label{rieszfourier}
\widehat{R(f m_d)}(\xi) = ib\frac{\xi}{|\xi|^{d+1-s}}\widehat{f}(\xi).
\end{equation}
Let $\varphi\in C^{\infty}_0(B(0,2))$ be a nonnegative radial bump function, such that $\varphi\geq 1$ on $B(0,1)$, $\varphi \leq 2^d$, $|\nabla \varphi |\leq 2\cdot 2^d$ on $B(0,2)$, and $\int_{B(0,2)}\varphi dm_d = m_d(B(0,2))$.

Define the vector field $\psi= \frac{-1}{ib}\mathcal{F}^{-1}(\xi |\xi|^{d-1-s}\widehat{\varphi}(\xi)).$ Then $\psi$ satisfies the decay estimate
\begin{equation}\label{psidecay}
|\psi(x)|\leq \frac{C_5}{(1+|x|)^{2d-s}},
\end{equation}
see for example \cite{SW71}, Chapter 4.  Combining the definition of $\psi$ with (\ref{rieszfourier}), we formally obtain
\begin{equation}\label{fourequal}
R^{*}(\psi  \,m_d) = \varphi,
\end{equation}
and this is justified since $\psi\in L^1(\mathbf{R}^d)$, see (\ref{psidecay}).

\section{A quantitative variant of Vihtil\"{a}'s theorem}\label{vihsec}This section is devoted to a suitable version of Vihtil\"{a}'s theorem.  It is at this point where the logarithmic dependence on $T$ arises in Theorem \ref{thm1}.
%$$\delta = O(\Delta^{\alpha_{\delta}}),\,\,{\mathcal{M}}= O( \exp(c_{\mathcal{M}}/\Delta^{\alpha_{\mathcal{M}}} )),\,\mathcal{N}=O( \Delta^{-\alpha_{\mathcal{N}}}).$$
%Here $\alpha_{\delta},\, \alpha_{\mathcal{M}}$ and $\alpha_{\mathcal{N}}$ are positive constants depending only on $s$, that will be determined over the course of proving Theorem \ref{thm1}.

First of all, we need to introduce a device to measure the number of scales at which the density of a measure $\mu$ exceeds a given threshold. For this purpose, introduce a density parameter $\delta\in (0,1)$.  Then for a ball $B_0 = B(x_0,r_0)$ and $q\in \mathbf{N}$, define the set $E_{\delta}^{q}(B_0)$ by
\begin{equation}
\begin{split}\nonumber E^{q}_{\delta}(B_0)= \Big\{x\in \frac{1}{2}B_0\, :\, & \frac{\mu(B(x,r))}{r^s}> \delta  \text{ for all }r\in\Bigl[\frac{r_0}{2^{q}}, \frac{r_0}{4}\Bigl]\Big\}.
\end{split}
\end{equation}

 Note that the set $E_{\delta}^q(B_0)$ is open. To see this, let $(x_j)_j$ be a sequence in $\mathbf{R}^d\backslash E_{\delta}^q(B_0)$ that converges to some $x\in \mathbf{R}^d$.  For each $j$, we have $\mu(B(x_j,r_j))\leq \delta r_j^s$ for some $r_j \in [\frac{r_0}{2^{q}}, \frac{r_0}{4}]$. By passing to a subsequence if necessary, we may assume that $r_j \rightarrow r$, with  $r\in [\frac{r_0}{2^{q}}, \frac{r_0}{4}]$.  As a result, $\displaystyle\liminf_{j\rightarrow \infty}B(x_j, r_j)\supset B(x,r)$, and therefore $\mu(B(x,r))\leq \displaystyle\liminf_{j\rightarrow \infty}\mu(B(x_j, r_j)) \leq \delta r^s$.  Hence $x\not\in E_{\delta}^q(B_0)$.

The quantitative version of Vihtil\"{a}'s theorem should read that, provided the Riesz transform is bounded, the measure of the exceptional set $E_{\delta}^{q}(B_0)$ should decrease with $q$ at a specific rate. %For a fixed $\varepsilon\in(0,1)$, the following proposition holds.
\begin{prop}\label{qualv}Suppose $||R(\mu)||_{L^{\infty}}\leq 1$.  Then there exist positive constants $C_{16}$ and $\beta$, depending on $s$ and $d$, such that
%for every
%\begin{equation}\label{qlarge}q\geq \frac{\,1}{\varepsilon}\cdot\exp\Bigl(\frac{C_{16}}{\delta^{\beta}}\Bigl),
%\end{equation}
%one has
%$\alpha_{\delta}, \,\beta>0$, one can choose $\alpha_{\mathcal{M}},\alpha_{\mathcal{N}}$ and $c_{\mathcal{M}}$, depending on $s$, $\beta$ and $\alpha_{\delta}$ such that, for any ball $Q$:
\begin{equation}\label{EB0small}\mu(E_{\delta}^{q}(B_0))\leq \frac{1}{q}\exp\Bigl(\frac{C_{16}}{\delta^{\beta}}\Bigl) \mu(B_0).
\end{equation}
\end{prop}

%We will apply this lemma with $Q_0$ chosen so:
%$$ \frac{\mu(Q_0)}{\ell(Q_0)^s}>\Delta, $$
%and $\delta$ chosen as a large power of $\Delta\in (0,1)$.
 The proof below yields the value $\beta = \frac{s-d+2}{s-d+1}$. The rest of this section is devoted to the proof of Proposition \ref{qualv}, and hence we will suppose that the condition $||R(\mu)||_{L^{\infty}}\leq 1$ is in force.  Assume that $E_{\delta}^q(B_0)\neq\varnothing$, since otherwise there is nothing to prove.  We will often suppress the dependence on $q$ and $\delta$ in $E_{\delta}^q(B_0)$ and write $E(B_0)$.

%In the argument that follows, there will be several lower bounds on the size of $q$.  We shall always assume that $q$ is large enough to satisfy any of these conditions.  At the end of the argument we will confirm that $q$ can indeed be chosen as in (\ref{qlarge}), for a suitable $\beta = \beta(s,d)>0$.

\subsection{An alternative}\label{altsec}   Fix a small positive number $\lambda = \lambda(\delta, d, s)\leq \delta$ to be chosen later. We begin with a simple auxiliary lemma.

\begin{lem}\label{nanres}  There exists a constant $c_6$, such that for any ball $B(x,t)$ with $\mu(B(x,t))\geq \lambda t^s$, we have
\begin{equation}\label{ballstillbig}
\mu(B(x,t(1-c_6\lambda^{\frac{1}{s+1-d}})))\geq \frac{\lambda}{2}t^s.
\end{equation}
\end{lem}
\begin{proof}For $0<\theta<1/2$, the annulus $B(x,t)\backslash B(x,(1-\theta)t)$ can be covered with $C\theta^{-(d-1)}$ balls of radius $\theta t$.  It follows from the growth condition (\ref{growth})  that
$$\mu(B(x,t)\backslash B(x,(1-\theta)t)) \leq C\theta^{1-d}\cdot C_1(\theta t)^s = CC_1 \theta^{s+1-d} t^s.
$$
Consequently, $\mu(B(x,(1-\theta)t))\geq \lambda t^s/2$, provided $CC_1\theta^{s+1-d}\leq \lambda/2$.  This is satisfied with $\theta = c_6\lambda^{\frac{1}{s+1-d}}$ as long as $c_6\leq(2CC_1)^{-\frac{1}{s+1-d}}$.
\end{proof}

Before the alternative is stated, let us identify our enemy: \textit{mediocre balls}.  These are stray balls of significant measure which are located away from $E_{\delta}^q(B_0)$.

\begin{defn} A ball $B(x,r)\subset B_0$ is called mediocre if $$B(x,r)\cap E_{\delta}^q(B_0) =\varnothing, \text{ and }\mu(B(x,r))>\lambda r^s.$$\end{defn}

Fix an integer $n$, $n\geq 2$, satisfying %with $1\ll n\ll q$.  We impose the following condition on $n$ and $\lambda$:
 \begin{equation}\label{nlambdacond}2^{-n}\leq c_6\lambda^{\frac{1}{s+1-d}}.\end{equation}

The alternative below states that in order for the set $E_{\delta}^q(B_0)$ not to have small measure, there must exist a ball $B\subset B_0$ of radius $r_1\geq 2^{n-q}r_0$ that does not contain a mediocre ball of radius $2^{-n}r_1$.

\begin{lem}\label{bottomalt}There exists a constant $C_8>1$ such that one of the following two statements holds.  Either

\indent (i) The measure of $E(B_0)$ is small, i.e.,
\begin{equation}\label{EB0small}\mu(E(B_0)) \leq \displaystyle\frac{C_82^{sn}n}{q\lambda}\mu(B_0),\end{equation}
or

\indent (ii) There exists a ball $B\subset B_0$ of radius $r_1\geq 2^{n-q}r_0$, centred at a point of $E(B_0)$, such that $B$ does not contain any mediocre balls of radius $2^{-n}r_1$.
\end{lem}

\begin{proof}Statement (i) in the alternative of the lemma is trivially true unless $q> 2^{sn}n/\lambda$, so we will assume this condition on $q$ is in force.

Suppose (ii) does not hold.  We will iteratively find many disjoint portions of $B_0$ whose measures are comparable to $\mu(E(B_0))$.  To present the main step, fix $r\in [r_0 2^{n-q},r_0/4]$.  Using an $r$-net in $E(B_0)$\footnote{Pick $z_1\in E(B_0)$, and let $B_1=B(z_1,r)$.  Given $B_1,\dots, B_k$, choose $z_{k+1}\in E(B_0)\backslash \cup_{j=1}^kB_j$ and let $B_{k+1} = B(z_{k+1}, r)$.  Repeat this process until the bounded set $E(B_0)$ is covered by the balls $B_j$.}, we find a finite collection $B_j = B(z_j, r)$ of balls with a covering number of at most $C_7$, such that $E(B_0)\subset \cup_jB_j$ and $z_j\in E(B_0)$ for all $j$.

By the assumption, within each ball $B_j$ there is a mediocre ball $D_j\subset B_j$ of radius $2^{-n}r$.  From condition (\ref{nlambdacond}) and Lemma \ref{nanres}, it follows that the contracted ball $\widetilde{D}_j = (1-2^{-n})D_j$ satisfies
$$\mu(\widetilde{D}_j)\geq \frac{\lambda}{2}(2^{-n}r)^s.
$$
The virtue of the collection of balls $\widetilde{D}_j$ is that they are well separated from $E(B_0)$.  Indeed, since $D_j\cap E(B_0)=\varnothing$, we have that $\operatorname{dist}(\widetilde{D}_j, E(B_0))\geq 2^{-2n}r$ for all $j$.
Now, note that
$$\mu\Bigl(\bigcup_j \widetilde{D}_j\Bigl) \geq\frac{1}{C_7}\sum_j \mu(\widetilde{D}_j) \geq \frac{\lambda}{C_72^{ns+1}}\sum_j r^s\geq \frac{\lambda}{C_1C_72^{ns+1}}\sum_j \mu(B_j),
$$
where in the last inequality the growth estimate for $\mu$ has been used.  Since $E(B_0)\subset \cup_j B_j$, we achieve the estimate
\begin{equation}\label{tildeDmeas}
\mu\Bigl(\bigcup_j \widetilde{D}_j\Bigl)\geq \frac{\lambda}{C_1C_72^{ns+1}}\mu(E(B_0)).
\end{equation}

For the iteration, employ the above argument with $r=2^{-2kn-2}r_0$ for $0\leq k \leq \lfloor(q-2n)/2n\rfloor$.  This yields collections of balls $\widetilde{D}_j^{k}$ disjoint from $E(B_0)$ and satisfying (\ref{tildeDmeas}) for each $k$.  Furthermore, the collections $\{\widetilde{D}_j^k\}_j$ do not overlap:
$$\Bigl[\bigcup_i\widetilde{D}_i^k\Bigl] \cap \Bigl[\bigcup_j \widetilde{D}_j^{\ell}\Bigl]=\varnothing \text{ for }k< \ell.$$
To see this, note that for any $j$, the ball $\widetilde{D}_j^{\ell}$ is contained in a ball of radius $2^{-2\ell n-2}r_0$ centred at a point of $E(B_0)$; and for each $i$, we have $\operatorname{dist}(\widetilde{D}_i^k, E(B_0))\geq 2^{-2(k+1)n-2}r_0 $ by the separation property.  There are $\lfloor q/2n \rfloor$ non-overlapping collections $\{\widetilde{D}_j^k\}_j$, each contained in $B_0$ and disjoint from $E(B_0)$. Hence
$$\frac{q\lambda}{2nC_1C_72^{sn+1}}\mu(E(B_0))\leq \mu(B_0).
$$
We conclude that part (i) of the alternative holds.
\end{proof}

The aim is now to show that the second part of the alternative is incompatible with the condition $||R(\mu)||_{L^{\infty}}\leq 1$.  Once this is established, Proposition \ref{qualv} will follow without difficulty.  Let us henceforth assume that part (ii) of the alternative in Lemma \ref{bottomalt} holds. To assert that this assumption results in the blow up of the Riesz transform, we start with finding a large ball of small measure whose boundary intersects $\overline{E(B_0)}$.

It will be convenient to denote $r = 2^{-n}r_1$, where $r_1$ is the radius of the ball $B$ from part (ii) of the alternative in Lemma \ref{bottomalt}.

\begin{lem}\label{larlighball}  There exists a positive constant $c_9$ such that if $n$ satisfies $c_9(2^{n(d-s)}\delta)^{1/d}>1$, then there exists a ball $D\subset \frac{1}{2}B$ with the following properties:

(i)  $D$ has radius $R = c_{9}(2^{n(d-s)}\delta)^{1/d}r$,

(ii) $D\cap E(B_0)=\varnothing$, and

(iii) there exists $z \in \overline{E(B_0)}\cap \partial D$.

\end{lem}

\begin{proof}  The existence of the ball follows from the pigeonhole principle.  Indeed, for a constant $a\in(0,\frac{1}{2}]$ to be chosen momentarily, consider a disjoint packing of balls $D_j$ with radius $a(2^{n(d-s)}\delta)^{1/d} r$ into the ball $\frac{1}{2} B$, such that $\operatorname{dist}(D_i, D_j)> 2a(2^{n(d-s)}\delta)^{1/d} r$ for all $i\neq j$.  One can pack at least $\frac{c2^{ns}}{a^d\delta}$ such balls into $\frac{1}{2} B$. (Note that $a(2^{n(d-s)}\delta)^{1/d}r< \tfrac{1}{2}2^n r=r_1/2$.)

Let $\widetilde{D}_j$ be the $r$  neighbourhood of $D_j$.  Assume that $a(2^{n(d-s)}\delta)^{1/d}\!>\!1$, and each ball $D_j$ intersects $E(B_0)$.  Then for every $j$, we have $\widetilde{D}_j\subset B$ and $\mu(\widetilde{D}_j)>\delta r^s$.  Furthermore, the condition $a(2^{n(d-s)}\delta)^{1/d}\!>\!1$ ensures that the open balls $\widetilde{D}_j$ are pairwise disjoint.

We are now in a position to derive a contradiction.  Indeed, the observations above yield the following chain of inequalities:
$$\mu(B)\geq \sum_j \mu(\widetilde{D}_j)\geq \delta r^s\cdot \frac{c2^{ns}}{a^d \delta}= \frac{c r_1^s}{a^d}.
$$
%\begin{figure}[t]\label{heavyballpic}
 % \centering
% \includegraphics{ballDpic}
 %\caption[The ball $D$.]%
 % {The figure shows the ball $D$ with the enlargement $D'$ inside of $B_1$.  The hatched balls represent various heavy balls scattered within $B_1$.  The nearest heavy ball to $D$ is $P_0$.}
%\end{figure}
Now choose $a = \min\bigl(\frac{1}{2}, \bigl(\frac{c}{C_1+1}\bigl)^{1/d}\bigl)$.  With this choice of $a$, the right hand side of the expression above is greater than $C_1 r_1^s$, which is in contradiction with the growth estimate (\ref{growth}).  As a result, one of the balls $D_j$ does not intersect $E(B_0)$.  We can now put $c_9=a$, and arrive at a ball $D$ satisfying (i) and (ii), provided $c_9(2^{n(d-s)}\delta)^{1/d}>1$.

It remains to translate $D$ so that (iii) holds.  To this end, recall that the centre of $B$ lies in $E(B_0)$. Therefore, one may move the ball $D$ towards the centre of $B$, until its boundary touches $\overline{E(B_0)}$ at some point $z$.%
%NEEDS TO BE SORTED OUT - WHAT IF THE POINT IS CLOSER!  If the ball $D$ just found does not satisfy (iii), It follows that the maximal ball $D^{\star}$ has radius bounded above by $r_1/4$, and in particular $D^{\star}\neq \frac{1}{2} B_1$.  The closest point $z$ of $E(B_0)$ to the ball $D^{\star}$ satisfies $0\leq \operatorname{dist}(z, D^{\star})\leq 2r$.  Indeed, otherwise it is possible to shift and enlarge $D^{\star}$, so that the resulting ball remains in $\frac{1}{2} B_1$ while condition (ii) continues to hold: a direct contradiction with maximality.
%Thus, if we define $D = B(x^{\star}, R^{\star}-2r)$, then $D$ satisfies (i)--(iv).
\end{proof}

Note that each ball $B(y,r)\subset D$ has measure $\mu(B(y,r))\leq \lambda r^s$.  This is because the ball $D$ contains no mediocre balls, and does not intersect $E(B_0)$.

%It will be convenient to assume that $R$ is substantially bigger than $r$. Therefore, let us henceforth assume that the condition $c_9(2^{n(d-s)}\delta)^{1/d}>64$ holds.

 %The lemma is proved by restricting $\lambda$, so that a negligible portion of the measure of $B(x, \rho r)$ lies in its intersection with $D$. From the relation (\ref{nlambdacond}), any restriction on $\lambda$ results in a corresponding lower bound upon the choice of $n$.

\subsection{A measure estimate}\label{auxmeassec}  We shall now state and prove an elementary lemma, which will enable us to exhibit the blow up of the maximal Riesz transform.

\begin{figure}[t]\label{ballpic}
\centering
 \includegraphics[trim=60 62 60 1, clip, width = 120mm]{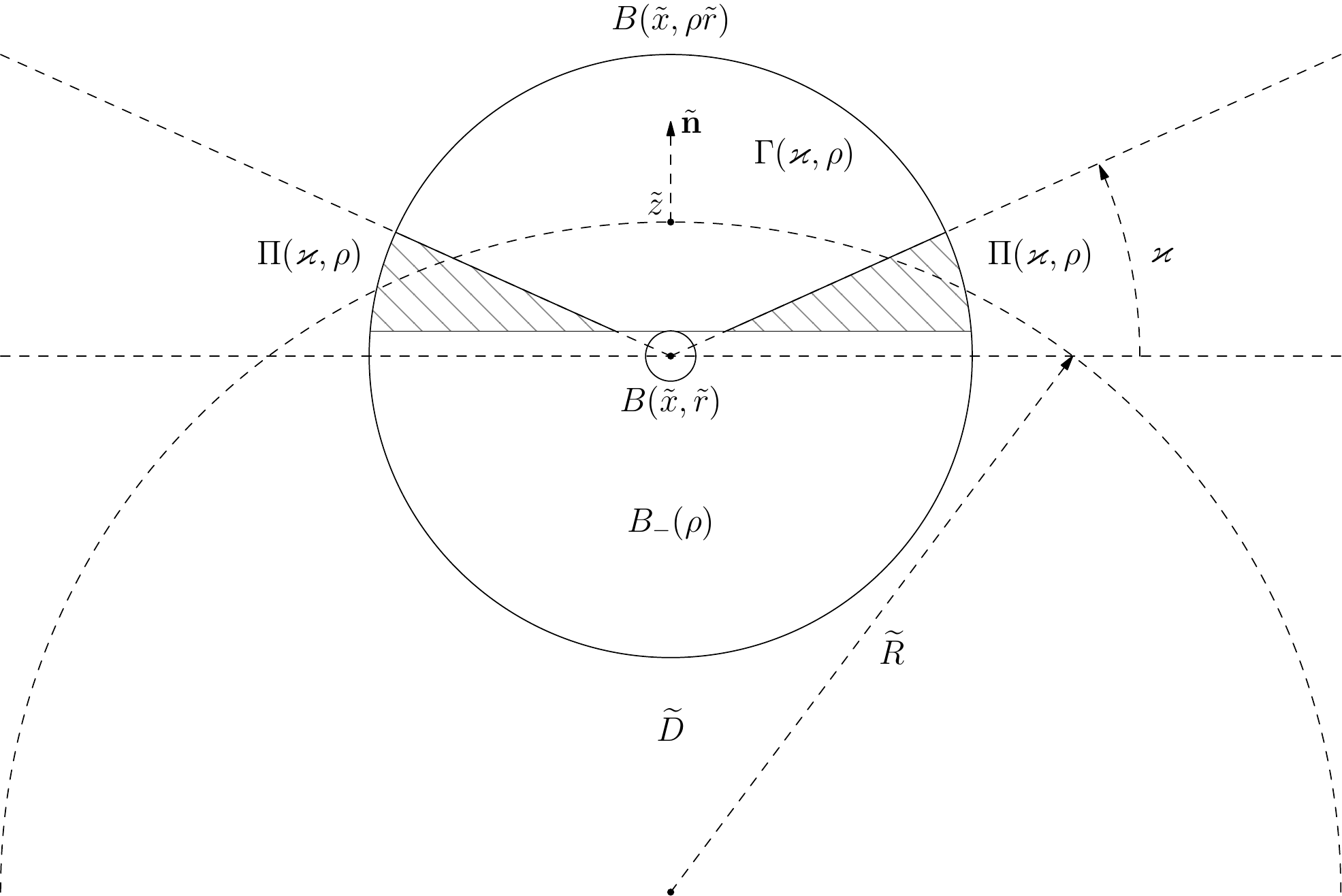}
\caption[Lemma \ref{topconebig} set-up.]
 {The set-up for Lemma \ref{topconebig}.  The angle $\varkappa$ will be chosen equal to $c_{10}\delta^{\frac{1}{s-d+1}}$.}
 \end{figure}

Let $\tilde{r}>0$ and let $\widetilde{R}>64\tilde{r}$.  Suppose that $\widetilde{D}$ is a ball of radius $\widetilde{R}$, with the property that every ball $B(y,\tilde{r})\subset \widetilde{D}$ has measure $\mu(B(y, \tilde{r}))\leq \lambda \tilde{r}^s$.  Let $\tilde{z}\in \partial \widetilde{D}$ be such that
\begin{equation}\label{tildedens}\mu(B(\tilde{z},t))\geq \delta t^s, \text{ for all }t\in (\tilde{r}, (\widetilde{R}\tilde{r})^{1/2}).\end{equation}  Finally, define $\tilde{x} = \tilde{z} - 4\tilde{r}\tilde{\mathbf{n}}$, where $\tilde{\mathbf{n}}$ is the outward unit normal to $\partial\widetilde{D}$ at $\tilde{z}$, see Figure 1 below.

\begin{lem}\label{topconebig}There exist positive constants $c_{10}$ and $c_{11}$ such that if $\rho$ and $\lambda$ satisfy
$\rho\in \bigl(8, (\widetilde{R}/\tilde{r})^{1/2}\bigl)$ and $\lambda \leq c_{11}\delta \rho^{s-d},$
%$n$ satisfies (\ref{nlambdacond}), then
%$\rho r\in [8r,2\sqrt{r(R-r)}],$ and $\lambda \leq c_{11}\delta^{\beta}\rho^{-(d-s)}$, then if $n$ satisfies (\ref{nlambdacond}),
then
\begin{equation}\label{conelarge}\nonumber
\mu\bigl(\Gamma(c_{10}\delta^{\frac{1}{s-d+1}}, \rho)\bigl)\geq \frac{\delta}{2^{s+1}}(\rho \tilde{r})^s,
%\mu\Bigl(\Bigl\{y\in B(\tilde{x}, \rho \tilde{r})\backslash B(\tilde{x},\tilde{r}) : \frac{\tilde{\mathbf{n}}\cdot(y-\tilde{x})}{|y-\tilde{x}|}\geq c_{10}\delta^{\frac{1}{s-d+1}}\Bigl\}\Bigl)\geq \frac{\delta}{2^{s+2}}(\rho \tilde{r})^s.
\end{equation}
where, for $\varkappa>0$, $$\Gamma(\varkappa,\rho) = \{y\in  B(\tilde{x}, \rho \tilde{r}):\tilde{\textbf{n}}\cdot (y- \tilde{x}) \geq \max(\varkappa|y-\tilde{x}|, \tilde{r}) \}.$$
\end{lem}

\begin{proof} %Let us first note that $x\in 1/2 \cdot P_j$ for all $j\geq 1$.
First note that  $B(\tilde{x}, \rho \tilde{r})\supset B(\tilde{z}, \rho \tilde{r}/2)$.  By (\ref{tildedens}), we have the estimate $\mu(B(\tilde{x}, \rho \tilde{r}))\geq \delta(\rho \tilde{r})^s/2^s.$  Our goal is to show that the majority of the mass of $B(\tilde{x}, \rho \tilde{r})\backslash B(\tilde{x},\tilde{r})$ lies within the set  $\Gamma(\varkappa,\rho)$, for a suitably chosen $\varkappa>0$.

As indicated in Figure 1 above, consider the lower part
$B_{-}(\rho)= \{ y\in B(\tilde{x}, \rho \tilde{r}) : \tilde{\mathbf{n}}\cdot (y-\tilde{x})\leq \tilde{r}\}$ of the ball $B(x, \rho \tilde{r})$, and the shaded region $\Pi(\varkappa, \rho) = \{y\in B(\tilde{x}, \rho \tilde{r}): \tilde{r}< \tilde{\textbf{n}}\cdot (y-\tilde{x}) < \varkappa|y-\tilde{x}|\}$.  Note that, by definition, $B(\tilde{x}, \tilde{r})\subset B_{-}(\rho)$.

We claim that
%\begin{equation}\label{kappacond}
$\mu(\Pi(\varkappa, \rho))\leq \delta(\rho \tilde{r})^s/2^{s+2}$
for $\varkappa = c_{10}\delta^{\frac{1}{s-d+1}}$, if $c_{10}>0$ is chosen small enough.  Indeed, for $\varkappa \in (0,1)$, consider a cover of $\Pi(\varkappa, \rho)$ by $C/\varkappa^{d-1}$ balls of radius $\varkappa\rho \tilde{r}$.  Applying the growth condition (\ref{growth}) to each covering ball yields $\mu(\Pi(\varkappa, \rho))\leq CC_1\frac{(\varkappa\rho \tilde{r})^s}{\varkappa^{d-1}}.$
The claim follows once we choose $c_{10} = (2^{s+2}CC_1)^{-\frac{1}{s-d+1}}.$

%Turning to the estimate for for the lower part of the ball, %note that $(\widetilde{R}-3\tilde{r})^2+ \widetilde{R}\tilde{r}<(\widetilde{R}-2\tilde{r})^2$.  Let $\rho \tilde{r}\leq (\widetilde{R}\tilde{r})^{1/2}$, then it follows from Pythagoras' theorem that $B_{-}(\rho)\subset B(\tilde{x}, \rho \tilde{r})\cap \widetilde{D}$ and $\operatorname{dist}(B_{-}(\rho), \partial \widetilde{D})\geq 2\tilde{r}$.  %We claim that there exists a constant $c_{11}>0$, such that if $\rho$ and $\lambda$ satisfy $\rho \tilde{r}\leq \sqrt{\widetilde{R}\tilde{r}}$ and $\lambda\leq c_{11}\delta/\rho^{d-s}$, then we have
%\begin{equation}\label{lambdacond}
%\mu(B_{-})\leq \frac{\delta}{4 \cdot 2^{s+1}}\cdot(\rho \tilde{r})^s.
%\end{equation}
%\begin{equation}\label{lambdacond}
%\mu(B_{-})\leq \mu(B(x,\rho r)\cap D)\leq \frac{\varkappa\delta}{4 \cdot 2^{s+1}}\cdot(\rho r)^s.
%\end{equation}
Now, cover the set $B_{-}(\rho)$ by $C\rho^d$ balls of radius $\tilde{r}$, such that each covering ball has its centre in $B_{-}(\rho)$.  Since $\sqrt{(\widetilde{R}-3\tilde{r})^2+ \widetilde{R}\tilde{r}}<\widetilde{R}-\tilde{r}$ (recall that $\widetilde{R}>64\tilde{r}$), each of these covering balls lies inside $\widetilde{D}$, and therefore has measure at most $\lambda \tilde{r}^s$.  Consequently, we deduce that
\begin{equation}\label{bottomhalfmeas}\mu(B_{-}(\rho))\leq C \rho^d \lambda \tilde{r}^s,\end{equation}  which is less than $\delta (\rho \tilde{r})^s/2^{s+2}$, provided $\lambda \leq c_{11}\delta \rho^{s-d}$ with $c_{11} \leq 1/(2^{s+2}C)$.

%We are now in a position to estimate the Riesz transform.  Suppose $\rho = 2^k\geq C_{10}(1/\delta)^{1/s}$, is such that $\rho$ and $\lambda$ satisfy conditions (\ref{rhocond}) and (\ref{lambdacond}).  Note that (\ref{rhocond}) and (\ref{lambdacond}) are also satisfied with $\rho$ replaced by any $t$ with $t<\rho$.

%Assume now that $\rho\in(8,(\widetilde{R}/\tilde{r})^{1/2})$ and $\lambda\leq c_{11}\delta \rho^{s-d}$.
Combining these measure estimates, we obtain $$\mu(\Gamma(\varkappa,\rho))\geq \mu(B(\tilde{x}, \rho \tilde{r}))-\mu(\Pi(\varkappa, \rho)) - \mu(B_{-}(\rho)) \geq \delta (\rho \tilde{r})^s/2^{s+1},$$
for $\varkappa = c_{10}\delta^{\frac{1}{s-d+1}}$.\end{proof}

 Let us now convert the measure estimate of Lemma \ref{topconebig} into an integral estimate. Denote $\beta  = \frac{s-d+2}{s-d+1} =1+\frac{1}{s-d+1}$. We will keep the notation of the proof of Lemma \ref{topconebig}.
 For $A>1$, write
 \begin{equation}\begin{split}\label{splitup}\nonumber&\int_{B(\tilde{x}, A \tilde{r}) \backslash B(\tilde{x},\tilde{r})}\frac{\tilde{\textbf{n}}\cdot(y-\tilde{x})}{|y-\tilde{x}|^{1+s}} d\mu(y)\\
& =  \int_{\Gamma(\varkappa,A)} \dots\, d\mu(y) + \int_{\Pi(\varkappa, A)} \dots\, d\mu(y) + \int_{B_{-}(A)\backslash  B(\tilde{x}, \tilde{r})} \dots \,d\mu(y),
\end{split}\end{equation}
and denote the three integrals on the right hand side by $I$, $II$, and $III$ respectively.  First, note that by the definition of $\Gamma(\varkappa, A)$,
$$I\geq \varkappa\int_{\Gamma(\varkappa,A)}\frac{d\mu(y)}{|y-\tilde{x}|^s}\geq s\varkappa\int_{8}^{A}\frac{\mu(\Gamma(\varkappa,\rho))}{(\tilde{r}\rho)^s}\frac{d\rho}{\rho},
$$
where Fubini's theorem has been applied in the final inequality.
% \begin{equation}\begin{split}\label{fubinistep}\tilde{\textbf{n}}\cdot&\int_{B(\tilde{x}, A \tilde{r}) \backslash B(\tilde{x},\tilde{r})}\frac{y-\tilde{x}}{|y-\tilde{x}|^{1+s}} d\mu(y)\\
%& \geq s\int_1^A \frac{1}{(\rho \tilde{r})^s}\int_{B(x, \rho \tilde{r})\backslash B(x,\tilde{r})}\frac{\tilde{\mathbf{n}}\cdot(y-\tilde{x})}{|y-\tilde{x}|}d\mu(y) %\frac{d\rho}{\rho} - C_1,
%\end{split}\end{equation}
%where (\ref{growth}) has been used to estimate the tail in the integration with respect to $\rho$.  %Applying the growth estimate (\ref{growth}) once again, we obtain
%$$\Bigl|\int_1^{8}\frac{1}{(\rho \tilde{r})^s}\int_{B(\tilde{x}, \rho \tilde{r})\backslash %B(\tilde{x},\tilde{r})}\frac{\tilde{\mathbf{n}}\cdot(y-\tilde{x})}{|y-\tilde{x}|}d\mu(y) \frac{d\rho}{\rho}\Bigl|\leq C_1\log(8).
%$$
Now suppose that $A$ and $\lambda$ satisfy $A \in (8, (\widetilde{R}/\tilde{r})^{1/2})$ and $\lambda\leq c_{11}\delta A^{s-d}$.  Then, with $\varkappa = c_{10}\delta^{1/(s+1-d)}$, we apply Lemma \ref{topconebig} to estimate  $I\geq sc_{10}\delta^{\beta}2^{-s-1}\log(A/8)$.  The integral $II$ is nonnegative, and therefore can be ignored in deducing a lower bound.  Concerning $III$, we apply Fubini's theorem once again to estimate
\begin{equation}\begin{split}\nonumber|III| & \leq s\int_1^{\infty}\frac{\mu(B_{-}(\rho)\cap B(\tilde{x}, A\tilde{r})\backslash B(\tilde{x}, \tilde{r}))}{(\rho \tilde{r})^s}\frac{d\rho}{\rho}\\
& \leq s\int_1^A\frac{\mu(B_{-}(\rho))}{(\rho \tilde{r})^s}\frac{d\rho}{\rho} + \frac{\mu(B_{-}(A))}{(A\tilde{r})^s}.
\end{split}\end{equation}
Since $A<(\widetilde{R}/\tilde{r})^{1/2}$ and $\lambda\leq c_{11}\delta A^{s-d}$, the bound in (\ref{bottomhalfmeas}) yields $|III|\leq C\lambda A^{d-s}\leq C_{12}$.  Thus, we arrive at the following corollary.
 \begin{cor}\label{auxrieszbig}  Under the conditions of Lemma \ref{topconebig}, we have
\begin{equation}\label{Aintest}
\int_{B(\tilde{x}, A \tilde{r}) \backslash B(\tilde{x},\tilde{r})}\frac{\tilde{\textbf{n}}\cdot(y-\tilde{x})}{|y-\tilde{x}|^{1+s}} d\mu(y)\geq \frac{sc_{10}}{2^{s+1}}\delta^{\beta}\log\bigl(\frac{A}{8}) - C_{12},
\end{equation}
provided $A \in (8, (\widetilde{R}/\tilde{r})^{1/2})$ and $\lambda\leq c_{11}\delta A^{s-d}$.
\end{cor}

\subsection{The conclusion of the proof of Proposition \ref{qualv}}  We are now in a position to bring everything together.

\begin{proof}[Proof of Proposition \ref{qualv}]  Assume that part (ii) of the alternative in Lemma \ref{bottomalt} holds.  As long as $c_9(2^{n(d-s)}\delta)^{1/d}>64$, we may apply Lemma \ref{larlighball} to find a ball $D$ of radius $R=c_9(2^{n(d-s)}\delta)^{1/d}r>64r$, and a point $z\in \partial D \cap \overline{E(B_0)}$.  Define $x= z-4r\mathbf{n}$, where $\mathbf{n}$ is the outward unit normal to $\partial D$ at the point $z$.  Since $D\cap E(B_0)=\varnothing$, every ball of radius $r$ contained in $D$ has measure at most $\lambda r^s$. Consequently, the conditions introduced in the beginning of Section \ref{auxmeassec} are satisfied %\footnote{Since $z\in \overline{E(B_0)}$, the condition (\ref{growth}) implies that $\mu(B(z, t))\geq \delta t^s$ for all $t\in (r, 2^n r)$, and $2^n r>\sqrt{Rr}$.}
 with $\widetilde{D}=D$, $\widetilde{R}=R$, $\tilde{x}=x$, $\tilde{z}=z$ and $\tilde{r}=r$.

Now suppose $A<\sqrt{R/r} = \sqrt{c_9(2^{n(d-s)}\delta)^{1/d}}$ and $\lambda< c_{11}\delta A^{s-d}$.  Then Corollary \ref{auxrieszbig} yields
 \begin{equation}\label{Aupbd}\int_{B(x, A r) \backslash B(x, r)}\frac{\textbf{n}\cdot(y-x)}{|y-x|^{1+s}}d\mu(y)\geq \frac{sc_{10}}{2^{s+1}}\delta^{\beta} \log (A/8) - C_{12},
 \end{equation}
 with $\beta = \frac{s-d+2}{s-d+1}$.
We would arrive at a contradiction if the right hand side of (\ref{Aupbd}) exceeds $3C_2$.  Indeed, it would follow that $$|R(\chi_{\mathbf{R}^d\backslash B(x, Ar)}\mu)(x) - R(\chi_{\mathbf{R}^d\backslash B(x, r)}\mu)(x)|\geq 3C_2,$$ which contradicts the Cotlar lemma (Lemma \ref{cotlar}).  This will be achieved if $A = \exp(C_{13}/\delta^{\beta})$.

It remains to choose $\lambda$ and $n$ so that all the above lemmas are applicable and this choice of $A$ is admissible in the end.  We will pick $\lambda$ first.  There are two assumptions on $\lambda$ independent of $n$:  $\lambda\leq\delta$, and $\lambda<c_{11}\delta A^{s-d}$. A reasonable choice of $\lambda$ is therefore $\lambda = \exp(-C_{14}/\delta^{\beta})$.
When choosing $n$, we have to satisfy the following three conditions:
$$2^{-n(s+1-d)}\leq c_6\lambda, \;c_9(2^{n(d-s)}\delta)^{1/d}>64,\, \text{ and } A< \sqrt{c_9 (2^{n(d-s)}\delta)^{1/d}}.$$
(The first condition is a restatement of (\ref{nlambdacond}), which guarantees that the alternative in Lemma \ref{bottomalt} holds with our choice of $\lambda$.) All three conditions are lower bounds on $n$.  In terms of the order of magnitude of $n$ as $\delta$ tends to zero, the first and third conditions are the most restrictive.  We are thus forced to choose $ n = \lfloor C_{15}/\delta^{\beta}\rfloor$.

With such choices of $\lambda$ and $n$, part (ii) of the alternative is in contradiction with the boundedness of the Riesz transform.   Substituting these values into (\ref{EB0small}), we get the desired estimate for the measure of $E(B_0)$.
\end{proof}

\section{The Cantor construction}\label{Cantorsec}
In this section we will use Proposition \ref{qualv} to quantify the Cantor construction of Eiderman, Nazarov and Volberg \cite{ENV11}.
%The goal is here to produce collections of sets $(\Omega_j^{(k)})_j$ for $1\leq k\leq N$, with the following properties:   For a fixed level $k$, the sets $\Omega_j^{(k)}$ should be well separated.  The sets should be nested, i.e. for each $k\geq 2$ and each cell $\Omega_j^{(k)}$, there exists $\Omega_{\ell}^{k-1}$ so that $\Omega_j^{(k)} \subset \Omega_{\ell}^{k-1}$.  In addition, the measure $\mu$ should have low density in each cell, and a significant amount of the measure of $\mu$ should be contained in  $\cup_j\Omega_j^{(k)}$ for any $k$.
\subsection{The general outline of the construction}Let $\Delta>0$, and let $\gamma \in(0,1]$.
%Fix positive parameters $\gamma$, $N$, $\varepsilon$, $M$ and $\delta$ to be chosen in that order.  The parameters $\varepsilon$ and $\delta$ will coincide with those in Proposition \ref{qualv}.  It is further assumed that $N,M\gg1$, with the other parameters $\gamma, \varepsilon, \delta \ll \Delta$.  All parameters will ultimately be chosen to depend on $\Delta$, $T$, $s$ and $d$.
Suppose that  $\mu$ is a finite non-negative measure with $||R(\mu)||_{L^{\infty}}\leq 1$.  Assume that
\begin{equation}\label{levelsetfalse}
\mu\Bigl(\!\Bigl\{x\in \mathbf{R}^d\!:\!\mathcal{L}\Bigl(\!\Bigl\{\!r\in (0,\infty) \!: \!\frac{\mu(B(x,r))}{r^s}>\Delta\!\Bigl\}\!\Bigl)>\!T\!\Bigl\}\!\Bigl)\!>\!2\gamma\mu(\mathbf{R}^d).
\end{equation}
We will show that this inequality contradicts the boundedness of the Riesz transform in $L^2(\mu)$ if $T$ is large enough.  Theorem \ref{thm1} will follow once we quantify this statement by obtaining a contradiction for every $T\geq\exp[(C\Delta^{-1}\gamma^{-1})^{1/\alpha}]$, with $C$ and $\alpha$ depending on $s$ and $d$ only.

Due to the growth condition (\ref{growth}), we may restrict our attention to $0<\Delta\leq C_1$.

The finiteness of $\mu$ guarantees that for any choice of $\Delta>0$, we have $\mu(B(x,r))\leq\Delta r^s$ for all $x\in \mathbf{R}^d$ and $r\geq R = \bigl(\tfrac{\mu(\mathbf{R}^d)}{\Delta}\bigl)^{1/s}$.  Hence there exists a compact set $E$ with
$$E\subset\Bigl\{x\in \mathbf{R}^d:\mathcal{L}\Bigl(\!\Bigl\{r\in (0,R) : \frac{\mu(B(x,r))}{r^s}>\Delta\Bigl\}\!\Bigl)>T\Bigl\},$$ such that $\mu(E)\geq \gamma \mu(\mathbf{R}^d)$.  Since both the condition $||R(\mu)||_{L^{\infty}}\leq 1$ and the assumption (\ref{levelsetfalse}) are invariant under replacing $\mu$ by $\mu(R\,\cdot\,)/R^s$, we may assume that $R=1$ without loss of generality.

%and will show that for suitably chosen $S$ $T$, display (\ref{levelsetfalse}) results in a contradiction with $||R(\mu)||_{L^{\infty}}\leq 1$.   By showing that $T$ can be chosen with the dependence on $\Delta$ from (\ref{weaktype}), this will prove Theorem \ref{thm1}.
The expression in (\ref{levelsetfalse}) becomes more palatable if we discretize the $\mathcal{L}$ measure.   %Hence, each point $x\in E$ possesses at least $T$ good scales,
%\begin{equation}\begin{split}\nonumber\Bigl\{x\in \mathbf{R}^d\!:\!\frac{dr}{r}&\Bigl(\Bigl\{r\in \!(0,1)\! : \!\frac{\mu(B(x,r))}{r^s}>\Delta\Bigl\}\Bigl)>T\Bigl\} \\
%&\subset \Bigl\{x:\, \text{ there are  } T \text{ good scales at }x\Bigl\},
%\end{split}\end{equation}
%where
To this end, we define a \emph{good scale at} $x$ to be a dyadic fraction $2^{-k}$, $k\in \mathbf{Z}_+$, for which the ball $B(x, 2^{-k})$ satisfies
\begin{equation}\label{goodscale}\frac{\mu(B(x,2^{-k}))}{2^{-sk}}>\frac{\Delta}{2^s}.
\end{equation}
Now suppose $\mu(B(x,r))>\Delta r^s$ for some $r\in (2^{-k-1}, 2^{-k}]$,  $k\in \mathbf{Z}_+$. Then we have $\mu(B(x, 2^{-k}))>\Delta 2^{-(k+1)s}$. It follows that
\begin{equation}\begin{split}\mathcal{L}\Bigl(&\Bigl\{r\in (0,1): \frac{\mu(B(x,r))}{r^s}>\Delta\Bigl\} \Bigl)\\
&\leq (\log 2)\cdot \text{card}\{k\in \mathbf{Z}_+: 2^{-k} \text{ is a good scale for $x$}\}.
\end{split}\end{equation}
We conclude that \emph{each point }$x\in E$ \emph{possesses }$T$ \emph{distinct good scales.} The construction of Cantor levels relies upon the existence of a noticeable set where all points have plenty of good scales.
% Now let $E$ be a compact set contained within the set of points with $T$ good scales, so that $$\mu(E)\geq \mu(\{x\in \mathbf{R}^d: \text{ there are  }T \text{ good scales at }x\})/2.$$
%\begin{equation}\label{Edefn}E :=  \Bigl\{x:\, \text{ there are } T \text{ good scales at }x\Bigl\}.
%\end{equation}
%From the assumption (\ref{levelsetfalse}), it follows that
%\begin{equation}\label{levelsetrecast}\mu(E)>\gamma \mu(\mathbf{R}^d).
%\end{equation}

We will need to introduce four auxiliary parameters, $N$, $\varepsilon$, $M$, and $\delta$, which will be chosen in this order to depend on $\gamma$, $\Delta$, $s$, and $d$.  The parameters $N$ and $M$ can be thought of as large, while $\varepsilon$ and $\delta$ can be thought of as small.  Their primary roles in the construction are described in the table below.
\begin{center}
  \begin{tabular}{ | l || l | }
    \hline
    Parameter  & Primary purpose of parameter \\ \hline\hline
    $N$ & The number of levels in the Cantor construction.\\ \hline
    $\varepsilon$  & The parameter controlling the measure of points lying \\
    & in various exceptional sets that we will need to remove. \\
    %& of the Cantor cells at a fixed level.\\
    \hline
    $M$ & The parameter controlling the size of a low density\\
    & region around each cell.\\ \hline
    $\delta$  & The parameter controlling the overall density\\
    &of the measure in  each Cantor cell.\\ \hline
    %$\gamma$ & The proportion (in terms of $\mu$) of points which have $T$ good scales.\\\hline
  \end{tabular}
\end{center}

During the construction, there will be several size requirements on $T$ - in terms of $N$, $\varepsilon$, $M$, and $\delta$ - to ensure there are sufficiently many good scales at any point of $E$ in order to construct a Cantor set deep enough to apply the arguments of \cite{ENV11}.   %Once the auxiliary parameters are chosen, we shall determine $T$  in terms of $\gamma$ and $\Delta$ so that a contradiction is reached with (\ref{levelsetfalse}).  %Inverting this relation on $T$ will ultimately designate $\gamma$ in terms of $T$ and $\Delta$, and Theorem \ref{thm1} will follow.  As a result, the reader should not be alarmed that we impose restrictions on $T$, despite the fact that $T$ is fixed in the statement of Theorem \ref{thm1}.

\begin{figure}[t]\label{cellpic}
 \centering
 \includegraphics[trim = 20mm 0mm 12mm 0mm, clip, width = 115mm]{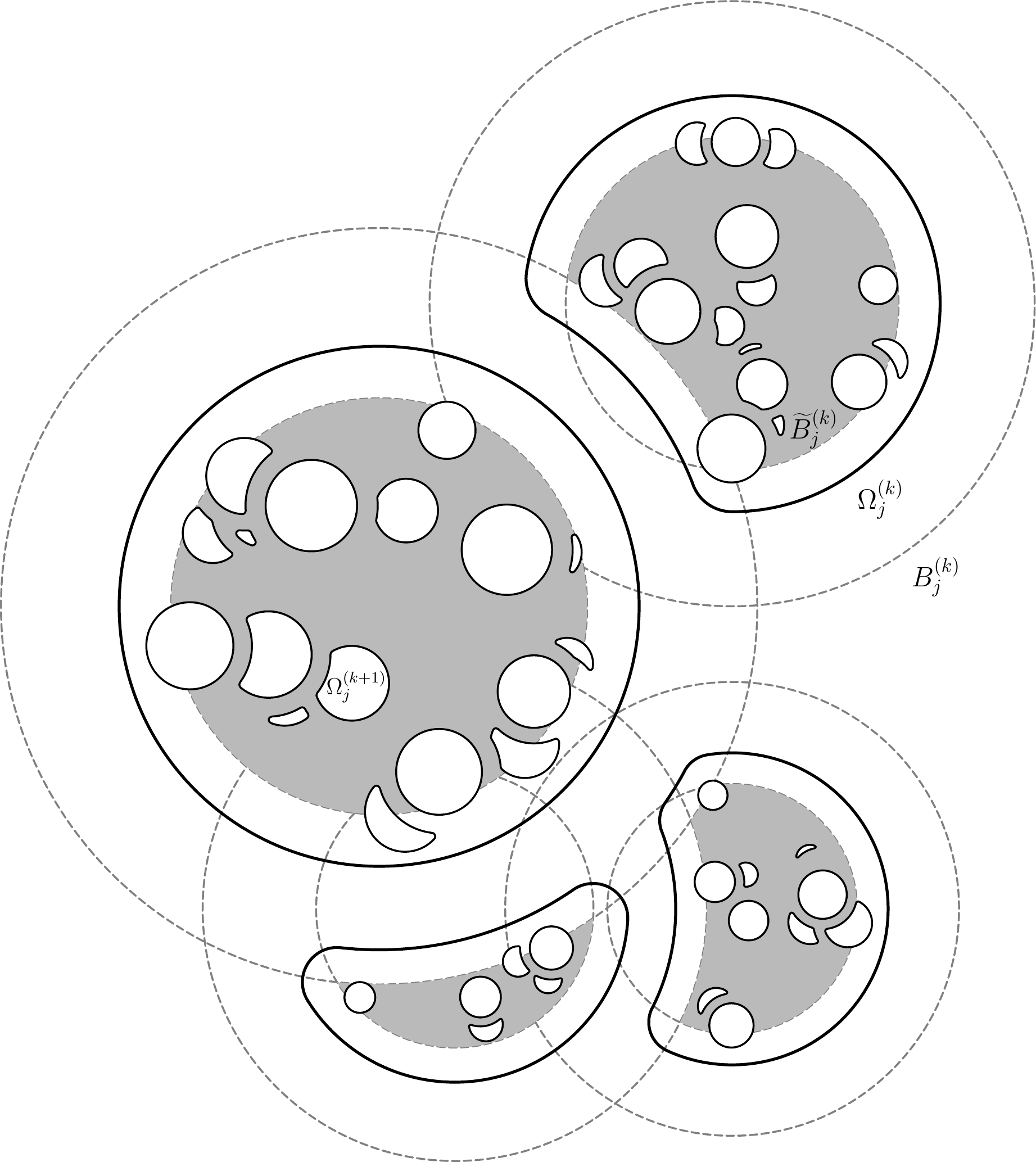}
 \caption[The Cantor construction.]
 {The figure depicts two levels of the Cantor construction.  Shown are four level $k$ cells, where the construction is displayed in full.  The large dashed balls $B_j^{(k)}$ are the bottom cover balls at the level $k$.  Each thick black path is the boundary of a Cantor cell $\Omega_j^{(k)}$.  The regions $\widetilde{B}_j^{(k)}$, partially shaded in grey, contain the inner sets $\widetilde{E}_j^{(k)}$.
 These are covered by the level $k+1$ cells $\Omega_j^{(k+1)}$, which are filled with white for contrast.}
\end{figure}

Each layer of the Cantor construction begins with choosing a \textit{top cover}.  The top cover will consist of high density balls corresponding to certain good scales.  We then apply Proposition \ref{qualv} to find the \textit{bottom cover}; namely a collection of low density balls, whose union contains all but a small portion of $E$.  Finally, we will modify these low density balls in order to obtain the Cantor cells of a given level. %These steps are covered over the next three subsections.

\subsection{The construction of one level}\label{topcover}

Consider two compact sets $\widetilde{E}$ and $\Omega$, both contained in an open ball $B$ of radius $\rho$.  Suppose $\widetilde{E}\subset \Omega$, and $\operatorname{dist}(\widetilde{E}, \partial\Omega)\geq\varepsilon \rho$.  A triple $(\Omega, \widetilde{E}, B)$ satisfying these properties is called an \emph{admissible triple}.  Assume that each $x\in \widetilde{E}$ possesses $\widetilde{T}$ good scales $2^{-k}$ with $2^{-k}\leq \varepsilon \rho/4$.  %It will be important to keep track of the number of good scales used in the construction.

%Let $m=\mu(\mathbf{R}^d)>0$.

(i) \textbf{\textit{The top cover.}}   At each point $x\in \widetilde{E}$, consider the set of all good scales $2^{-k}$ satisfying $2^{-k}\leq \varepsilon \rho/4$, and denote by $r_x$ the largest of those good scales.  We will apply the Vitali construction to the balls $\{B(x,r_x)\}_{x\in\widetilde{E}}$.

First choose $B(z_1,r_1)$ to be a ball of largest radius from the collection $\{B(x, r_x)\}_{x\in \widetilde{E}}$ (recall that each $r_x$ is nonpositive integer power of $2$ so the largest ball always exists).  Given balls $B(z_1,r_1), \dots, B(z_k, r_k)$, we choose $B(z_{k+1}, r_{k+1})$ to be a largest ball $B(x, r_x)$ that is disjoint from every previous ball $B(z_j, r_j)$, $j=1, \dots, k$.  If no further selection is possible, we terminate the process.  Since $\widetilde{E}$ is a bounded set, if the algorithm does not terminate, the radii $r_j$ tend to $0$ as $j\rightarrow \infty$.  By construction, the radii of the balls $B_j$ are non-increasing.

%Let $r_1$ be the maximum of the radii $r_c$ for all $c\in \widetilde{E}$ (recall here each $r_c$ is a dyadic fraction).  Pick $c_1\in \widetilde{E}$ with $r_{c_1}=r_1$, and let $B(c_1, r_1)$ be the first ball in the cover.  Given balls $B(c_1, r_1),\dots, B(c_k, r_k)$, we define $r_{k+1}$ to be the maximum of the radii of the balls $B(c, r_c)$, $c\in \widetilde{E}$, that are disjoint from  the $k$ balls in the sequence. Pick $c_{k+1}$ so that $r_{c_{k+1}} = r_{k+1}$, and let $B(c_{k+1}, r_{k+1})$ be the next ball in the cover.
%Given $B(c_1, r_1), \dots, B(c_k, r_k)$, define $r_{\max}^{(k+1)} = \max\{r_c: c\in \widetilde{E}\,\backslash \cup_{j=1}^k B(c_j, 3r_j)\}$, and let $c_{k+1}$ be such that $r_{k+1} = r_{c_{k+1}} = r_{\max}^{(k+1)}$.  The ball $B(c_{k+1}, r_{k+1})$ forms the $(k+1)$-st ball in the cover.

The balls $B(z_j, 2r_j)$ cover the set $\widetilde{E}$.  In fact,
\begin{equation}\label{blah} \text{for any }x\in \!\widetilde{E},\text{ we have }x\in\! B(z_j, 2r_j)\text{ for some }j \text{ with }r_j\geq r_x.\end{equation}  Indeed, otherwise $B(x, r_x)$ is disjoint from all balls $B(z_j, r_j)$ with $r_j\geq r_x$, but has not been chosen in the Vitali cover. This contradicts the selection rule.

By compactness, there exists $J\in \mathbf{N}$ such that the finite sequence $B(z_1, 2r_1), \dots, B(z_J, 2r_J)$ covers $\widetilde{E}$.  For a point $x\in \widetilde{E}$, let $j(x)\in \{1,\dots, J\}$ be the index corresponding to a largest ball $B(z_{j(x)}, 2r_{j(x)})$ containing $x$.  %Since the radii $r_j$ are non-increasing,  the ball $B(z_{j(x)}, 2r_{j(x)})$ is the largest ball to contain $x$ from the concentric doubles of the entire (not necessarily finite) Vitali cover.  It follows from (\ref{blah})
Since the radii $r_j$ are non-increasing, from (\ref{blah}) we see that $r_{j(x)}\geq r_x$. %, and hence the property (\ref{blah}) continues to hold for the finite subcover $B(c_1, 2r_1), \dots, B(c_J, 2r_J)$.
%Applying the Vitali covering lemma to the collection of balls $(B(c,r_c))_{c\in \widetilde{E}}$ yields
%a finite sequence $(B(c_j, r_j))_{j=1,\dots,J}$ of balls, such that $\bigcup_{j=1}^J B(c_j, 3r_j)\supset \widetilde{E}$.\footnote{Applied to $(B(c,r_c))_{c\in \widetilde{E}}$, the Vitali covering lemma yields a countable sequence $(B(c_j, 3r_j))_j$, which we order so that $r_j$ is non-increasing.  Since $\widetilde{E}$ is a compact set, it is covered by a finite subcollection $(B(c_j, 3r_j))_{j=1, \dots J}$, for some $J>0$ .}
%countable subcollection $(B(c_j, r_j))_{j=1}^{\infty}$ of disjoint balls, such that $\bigcup_j B(c_j, 3r_j)\supset \bigcup _{c\in \widetilde{E}}B(c, r_c)$.  We order the balls so that $r_j\geq r_{j+1}$ for all $j\geq 1$.  Since $\widetilde{E}$ is a compact set, there is a finite subsequence $B(x_1, r_1), \dots , B(x_J, r_J) $, which covers $\widetilde{E}$. %with the property that each ball $B_j$ in the finite subsequence has radius greater than the maximum of the radii of all balls which has been removed\footnote{For any $t>0$, there are only a finite number of balls in the entire Vitali subcover with radius greater than $t$.}.  We shall henceforth work with this \emph{finite} subcollection of balls.

The finite collection of further enlarged balls $T_j = B(z_j, 4r_j)$, $j=1, \dots, J$,
 forms the \textit{top cover}.    We will need the following two key observations about the top cover.

 First, for each point $x\in \widetilde{E}$,  the associated top cover ball $T_{j(x)}=B(z_{j(x)}, 4r_{j(x)})$ satisfies $x\in\frac{1}{2} T_{j(x)}$ and $r_{x}\leq r_{j(x)}$.  Therefore, the number of good scales $2^{-k}$ at $x$  with $2^{-k}\leq r_{j(x)}$ is still at least $\widetilde{T}$.

The second key property is a measure estimate:
\begin{equation}\label{meastj}\sum_{j=1, \dots,\, J}\mu(T_j)\leq \frac{C_{19}}{\Delta}\mu(\Omega).\end{equation}
To see this, note that $\mu(B(z_j, r_j))\geq \Delta r_j^s/2^s$, and therefore (\ref{growth}) implies that
$$\mu(B(z_j, 4r_j))\leq C_1 4^s r_j^s\leq \frac{C_1 8^s}{\Delta}\mu(B(z_j, r_j)).$$  As the balls $B(z_j, r_j)$ are disjoint and contained in $\Omega$, we conclude that (\ref{meastj}) holds.

%Note that each $x\in \widetilde{E}$ is covered by a ball $B(c_j, 3r_j)$, for $j\in \{1,\dots, J\}$, with $r_j\geq r_x$. Indeed, by inspection of the proof of the Vitali covering lemma, such a ball exists within the (possibly) infinite Vitali cover of $(B(c, r_c))_{c\in \widetilde{E}}$.  Since the finite sub-cover has been extracted from the largest balls in the Vitali cover, the claim follows.

(ii)  \textbf{\textit{From the top cover to the bottom cover.}}  Suppose that $M>1$ and $\delta<\min\bigl(1, \tfrac{\Delta}{2^{s+1}}\bigl)$.  Fix $q\in \mathbf{N}$ such that $q$ is slightly greater than $\varepsilon^{-1}\exp[C_{16}2^{s\beta}M^{s\beta}/\delta^{\beta}].$  For each $j\in \{1, \dots, J\}$, we apply Proposition \ref{qualv} with $B_0 = T_j$, and $\delta$ replaced by $\delta/(2M)^s$.  With our choice of $q$, the set $E_{\delta/(2M)^s}^q(T_j)\subset \frac{1}{2}T_j$ has measure $\mu(E^q_{\delta/(2M)^s}(T_j))\leq \varepsilon \mu(T_j)$ for each $j$.  Define the exceptional set $F$ by\begin{equation}\label{exceptset}F=\bigcup_{j=1,\dots, \, J}E^q_{\delta/(2M)^s}(T_j).\end{equation}
Then $F$ is an open set, and (\ref{meastj}) implies that $\mu(F) \leq \frac{C_{19}\varepsilon}{\Delta}\mu(\Omega)$.

Let $x\in \widetilde{E}\,\backslash F$.   Since $x\in(\frac{1}{2}T_{j(x)})\backslash E_{\delta/(2M)^s}^q(T_{j(x)})$, and $r_{j(x)}\geq r_x$, there exists a ball $B(x,M\tilde{t}_x)$ such that \begin{equation}\label{txrad}r_{j(x)}\geq  M \tilde{t}_x \geq 2^{-q}4r_{j(x)}\geq 2^{-(q-2)}r_x,\end{equation} and $\mu(B(x, M\tilde{t}_x))\leq \frac{\delta}{(2M)^s}(M\tilde{t}_x)^s = \frac{\delta}{2^s} \tilde{t}_x^s$. Now let $t_x = 2^{-\ell}r_{j(x)}$ where $\ell$ is such that $2^{-\ell}r_{j(x)}\in (\frac{1}{2}\tilde{t}_x, \tilde{t}_x]$.  Then the ball $B(x, Mt_x)$ satisfies
\begin{equation}\label{smalldens}\mu(B(x, Mt_x))\leq \delta t_x^s.
\end{equation}
By construction, $B(x,M t_x)\subset T_{j(x)}$, and moreover, \begin{equation}\label{lowdensdeep}\operatorname{dist}(B(x, Mt_x), \partial T_{j(x)})\geq r_{j(x)}.\end{equation} From (\ref{txrad}), we see that $t_x\geq \frac{2^{-(q-1)}}{M}r_x$.  Therefore, if $\widetilde{T}> q+\log_2 M$, then each $x\in \widetilde{E}\,\backslash F$ has at least $\widetilde{T}- q-\log_2 M $ good scales $2^{-k}$ with $2^{-k}\leq t_x$.  %From the measure estimate (\ref{meastj}), we observe that the exceptional set of all points $x\in E\cap \widetilde{B}\subset \bigcup_j \frac{1}{2}T_j$ for which $t_x=0$ has measure at most $\varepsilon\sum_j \mu(T_j)\leq C_{19}\varepsilon \mu(\Omega)/\Delta$.

% Here and throughout the rest of the argument, we are tacitly assuming the number of good scales $T$ at any $x\in E$ is greater than any operation in the Cantor construction.  Once the construction is complete the required number of scales will be counted.

We will now shrink the balls $B(x,t_x)$ to eliminate the possibility that the mass of any ball in the collection is concentrated near its boundary.

To this end, fix $x\in \widetilde{E}\,\backslash F$.  Suppose $(1-3\varepsilon)^s>\tfrac{1}{2}$, and put $\lambda_j = (1-3\varepsilon)^j$. Consider the sequence of balls $\{B(x, \lambda_j t_x)\}_j$, and assume that
\begin{equation}\label{bignearbound}\mu(B(x, \lambda_j t_x) \backslash B(x, \lambda_{j+1}t_x)) \geq 3d\varepsilon \mu(B(x, \lambda_j t_x)),
\end{equation}
for all $j=0,\dots,k-1$.
%At each $x$, the stopping time procedure terminates before we encounter the next good scale smaller than $t_x$.\footnote{In particular, if $\widetilde{T}\geq q+\log_2 M$, the procedure terminates for each $x\in \widetilde{E}\backslash F$.}
%Indeed, %for $a \in (0,d]$, and
%for each $j<k$, we have
Then
$$\frac{\mu(B(x,\lambda_{j}t_x))}{(\lambda_{j}t_x)^{s}}\! \leq \!\frac{1-3d\varepsilon}{(1-3\varepsilon)^{s}}\!\cdot\!\frac{\mu(B(x, \lambda_{j-1}t_x))}{(\lambda_{j-1}t_x)^{s}}<\frac{\mu(B(x, \lambda_{j-1}t_x))}{(\lambda_{j-1}t_x)^{s}},
$$
for each $j=1,\dots, k$.  Since $\mu(B(x, t_x))\leq \delta t_x^s$, we see by induction that
\begin{equation}\label{stopdecay}\frac{\mu(B(x,\lambda_{j}t_x))}{(\lambda_{j}t_x)^{s}}\leq \delta \,\,\text{ for all }j=0,\dots, k.
\end{equation}

Suppose that $2^{-\ell}$ is a good scale at $x$ with $2^{-\ell}\in [\lambda_k t_x, t_x]$.  Then let $j\geq 0$ be the largest index with $\lambda_j t_x\geq 2^{-\ell}$.  Since $0\leq j\leq k$, we may apply (\ref{stopdecay}) to observe that $$\mu(B(x,2^{-\ell}))\leq \mu(B(x, \lambda_j t_x))\leq\delta \lambda_j^st_x^s\leq\frac{\delta}{(1-3\varepsilon)^s}2^{-\ell s}< \frac{\Delta}{2^{s}}2^{-\ell s},$$
which is a contradiction.  As long as $\widetilde{T}> q+\log_2 M$, there is a good scale $x$ no greater than $t_x$, and hence (\ref{bignearbound}) fails for a finite index.

Let $k$ be the least index such that
\begin{equation}\label{annulinot}
\mu(B(x, \lambda_k t_x) \backslash B(x, \lambda_{k+1}t_x)) \leq 3d\varepsilon \mu(B(x, \lambda_k t_x)).
\end{equation}
As we have seen,

\centerline{\emph{there is no good scale at $x$  between $\lambda_k t_x$  and $t_x$.}}  %The observation that these balls can be found at each $x\in \widetilde{E}\backslash F$ at a \textit{uniform cost in good scales} is crucial to the construction.

Now put $\rho(x) = \lambda_k(x)t_x$.  The introduction of $\lambda_k$ does not distort the density estimate (\ref{smalldens}) too much.

% The stopping time argument does not distort the density estimate (\ref{smalldens}) too much:
\begin{lem}\label{afterstoplem} The following estimate holds:
\begin{equation}\label{afterstop}\mu(B(x, M\rho(x))) \leq 2 M^s \delta \rho(x)^s.
\end{equation}
\end{lem}

\begin{proof}  If $M\rho(x)\geq t_x$, then  (\ref{afterstop}) follows from (\ref{smalldens}).  Otherwise, let $j$ be the largest index with $\lambda_jt_x\geq M\rho(x)$.  We have $0\leq j\leq k$  and (\ref{stopdecay}) yields
$$\mu(B(x, M\rho(x)))\leq \delta (\lambda_jt_x)^s \leq (1-3\varepsilon)^{-s}\delta M^s\rho(x)^s\leq 2M^s\delta \rho(x)^s,$$
as required.
\end{proof}

We shall now apply the Besicovitch covering construction to the family of balls $\{B(x, \rho(x))\}_{x\in \widetilde{E}\backslash F}$.    First note that all radii $\rho(x)$ are of the form $2^{-\ell_1}(1-3\varepsilon)^{\ell_2}$, for some nonnegative integers $\ell_1$ and $\ell_2$ (which depend on the point $x$).  %Since there are only finitely many numbers of this type that are greater than any given threshold,
Hence, given any non-empty subcollection of balls from $\{B(x, \rho(x))\}_{x\in \widetilde{E}\backslash F}$, there exists a ball of maximum radius in the sub-collection.

Let $B_1 = B(x_1, \rho_1)$ be a largest ball $B(x, \rho(x))$.  Given balls $B_1,\dots, B_k$, let $B_{k+1} = B(x_{k+1}, \rho_{k+1})$ be a largest ball $B(x, \rho(x))$ whose centre $x$ does not lie in $B_j$ for any $j\in \{1, \dots, k\}$.  If no further selection is possible, the process terminates.  It is clear by construction that the radii are non-increasing in $j$.  Since $\widetilde{E}\backslash F$ is bounded, if the algorithm does not terminate, then $\rho_j \rightarrow 0$ as $j\rightarrow \infty$ (note that the balls $B(x_j, \rho_j/2)$ are disjoint).

%To continue the cover, define $M_1 = \sup\{\rho(x): x\in \widetilde{E}\backslash (F \cup \bigcup_{j=1}^{n} B_j)$.  Choose a ball $B_{n+1}=B(x_{n+1}, \rho(x_{n+1}))$ with $x_{n+1}\in \widetilde{E}\backslash (F \cup \bigcup_{j=1}^{n} B_j)$ and $\rho(x_{n+1})\in (M_1/2, M_1]$.  We then follow the packing algorithm of the first step in the cover, until there are no balls with radius in  $(M_2/2, M_2]$, whose centre does not lie in any ball already chosen in the cover.  We continue in this manner, first defining $M_3$ to be the supremum of those balls in $(B(x, \rho(x)))_{x\in \widetilde{E}\backslash F}$ whose centre does not lie in any previous ball of the cover, and then carrying out an analagous packing argument.  We shall relabel $\rho(x_j) = \rho_j$.

A ball $B(x, \rho(x))$ would only remain unselected if $x\in B_j$ for a ball $B_j$ with $\rho_j \geq \rho(x)$.  Therefore, the balls $B_j$ form an open cover of the compact set $\widetilde{E}\backslash F$.  It follows from compactness that the selection algorithm terminates with a finite sequence $\{B_j\}_{j=1, \dots, K}$, which covers $\widetilde{E}\,\backslash F$.  The finite collection of balls $\{B_j\}_{j=1,\dots, K}$ forms the \emph{bottom cover}.

The selection rule guarantees that a centre $x_j$ does not lie in any ball $B_k$ for $k\neq j$.  This is immediate for $k<j$. If $k>j$ then $\rho_k\leq \rho_j$, and so $x_j\in B_k$ implies that $x_k\in B_j$, which contradicts the choice of $B_k$. Suppose now that $z$ lies in the intersection of two balls $B_{j}$ and $B_{k}$.   Since $|x_j-x_k|\geq \max(\rho_j, \rho_k)$, the line segment between $x_j$ and $x_k$ is the longest side of the triangle formed by the three points $z$, $x_j$, and $x_k$.  It follows that the angle between $x_j$ and $x_k$, measured at the point $z$, is at least $\pi/3$.  Since this holds for each pair of Besicovitch balls containing $z$, we see that any point can be contained in at most $C_{20}$ of the balls in the bottom cover.%An elementary calculation (see for example Lemma 2.2 of \cite{Mat95}) shows that for any two such balls $B_j$ and $B_k$, we have
Let $\widetilde{B}_j$ be the closure of $(1-3\varepsilon)B_j\backslash \cup_{i<j}B_i$ for each $j=1, \dots, K$, and define $\widetilde{E}_j = \widetilde{B}_j\cap \widetilde{E}\backslash F$.   If $x\in \widetilde{E}_j$ for some $j
\in \{1, \dots, K\}$, then $B_j$ is the largest of the Besicovitch balls to contain $x$.  By the selection rule, it follows that $\rho_j \geq \rho(x)$.  Recall that there are no good scales at $x$ between $\rho(x)$ and $t_x$.  As a result, if $\widetilde{T}$ satisfies \begin{equation}\label{tildetcond}\widetilde{T}>  q+\log_2M +\log_2 \frac{1}{\varepsilon} +3,\end{equation} then there are at least $\widetilde{T}-  q-\log_2M -\log_2 \frac{1}{\varepsilon} -3$ good scales $2^{-k}$ at each $x\in \widetilde{E}_j$, with $2^{-k}\leq \varepsilon \rho_j/4$.

The sets $\widetilde{E}_j$ cover $\widetilde{E}$ except for the intersection of $\widetilde{E}$ with $F\cup \bigcup_j [B_j\backslash (1-3\varepsilon)B_j]$.  This latter set has small measure.  Indeed, since the balls $B_j$ are contained in $\Omega$, and have a finite covering number of at most $C_{20}$, we may apply (\ref{annulinot}) to estimate $$\mu\bigl(\bigcup_j B_j\backslash (1-3\varepsilon)B_j\bigl)\leq 3d\varepsilon\sum_j \mu(B_j)\leq 3d\varepsilon C_{20}\mu(\Omega).$$ Combined with the measure estimate for $F$, we see that
\begin{equation}\label{levelloss}\mu\bigl(F\cup \bigcup_j B_j\backslash (1-3\varepsilon)B_j\bigl) \leq C_{21}\varepsilon \mu(\Omega)/\Delta,\end{equation}
since $\Delta\leq C_1$.

The sets $\widetilde{B}_j$ are nicely separated:  $\operatorname{dist}(\widetilde{B}_j, \widetilde{B}_k) \geq 3\varepsilon \max(\rho_j, \rho_k)$ for all $j\neq k$.  For those non-empty $\widetilde{B}_j$, define $\Omega_j$ to be the closed $\varepsilon \rho_j$ neighbourhood of $\widetilde{B}_j$. It is clear that $\operatorname{dist}(\Omega_j, \Omega_k) \geq \varepsilon \max(\rho_j, \rho_k)$ whenever $j\neq k$.

 Let us now summarize the key properties of the construction:

a)  \emph{Self-similarity.}  Given an admissible triple $(\Omega, \widetilde{E}, B)$, the algorithm yields a collection of admissible triples $(\Omega_j, \widetilde{E}_j, B_j)$, with $\Omega_j \subset \Omega$ for each $j$.  Indeed, for each $j$ we have $\widetilde{E}_j\subset \Omega_j\subset B_j$, and $\operatorname{dist}(\widetilde{E}_j, \partial \Omega_j)\geq\varepsilon \rho_j$.  We are therefore able to iterate the algorithm.

%There are two obstacles that could prevent successful iteration of the algorithm.  Firstly, we may not have enough good scales at each point of $\widetilde{E}$ to construct successive top covers. Secondly, the set $\widetilde{E}$ may be exhausted by exceptional sets.  Our next point concerns the first of these issues. The second obstacle to iterating the algorithm is exhaustion of $\widetilde{E}$.

b) \emph{Uniform cost in good scales.}  Suppose that $\widetilde{T}$ satisfies (\ref{tildetcond}). Then there are enough good scales at each point of $\widetilde{E}\backslash F$ to construct the cells $\Omega_j$ and $\widetilde{E}_j$.  Furthermore, for each $j$, and for any $x\in \widetilde{E}_j$, there are at least $\widetilde{T}-  q-\log_2M -\log_2  \frac{1}{\varepsilon} - 3$ good scales at $x$ smaller than $\varepsilon \rho_j/4$.

c) \emph{Small loss of measure.}  An immediate consequence of (\ref{levelloss}) is that
\begin{equation}\label{smalllosest}\mu\Bigl(\bigcup_j \widetilde{E}_j\Bigl)\geq \mu(\widetilde{E}) - C_{21}\varepsilon \mu(\Omega)/\Delta.\end{equation}

d) \emph{Separated cells.} Any two cells $\Omega_i$ and $\Omega_j$ are well separated: $\operatorname{dist}(\Omega_i, \Omega_j)\geq \varepsilon \max(\rho_i,\rho_j)$ for any $i\neq j$. %, and $\operatorname{dist}(\Omega_i, \partial \Omega)> \rho_i$ for each $i$.

e) \emph{Low density cells.}  For each $j=1,\dots,K$, the cell $\Omega_j \subset B_j$ and \begin{equation}\label{bottomlowdens}\mu(M B_j)\leq 2M^s\delta \rho_j^s.\end{equation}

f) \emph{Thick cells}.  By their definition, each cell $\Omega_j$ contains an open ball of radius $\varepsilon \rho_j$.

g) \emph{Associated top cover balls.}  Each cell $\Omega_j$ can be associated to a top cover ball $T_k = B(z_k, 4r_k)$, for some $k\in \{1,\dots, J\}$, so that $M\rho_j \leq r_k$, $\Omega_j\subset B_j \subset T_k$, and $\operatorname{dist}(B_j, \partial T_k)\geq r_k.$  To see this, note that the bottom cover ball $B_j$ is a subset of a low density ball $B(x,t_x)$, for some $x\in \widetilde{E}\backslash F$.  The top cover ball $T_{j(x)}$ satisfies the required properties (see (\ref{lowdensdeep})).

\subsection{Construction of the set.}\label{alllevels}
We will now carry out an $N$-fold iteration of the algorithm of Section \ref{topcover} to produce the Cantor set.

For each $k\geq 0$, define $\widetilde{T}^{(k)}$ by
$$\widetilde{T}^{(k)} = (N-k)(q+\log_2 M+\log_2\frac{1}{\varepsilon}+3).
$$
Assume that we are given a finite collection of admissible level $k$ triples $(\Omega_j^{(k)}, \widetilde{E}_j^{(k)}\!, \!B_j^{(k)})$, satisfying the following properties:
\begin{itemize}
\item For each $j$, every $x\in \widetilde{E}^{(k)}_j$ has at least $\widetilde{T}^{(k)}$ good scales smaller than $\rho_j^{(k)}/4$, where $\rho_j^{(k)}$ is the radius of $B_j^{(k)}$.
\item For any $i \neq j$, $\operatorname{dist}(\Omega_i^{(k)}, \Omega_j^{(k)})\geq \varepsilon \max(\rho^{(k)}_i, \rho^{(k)}_j)$.
\end{itemize}

With $j$ fixed, applying the algorithm to the triple $(\Omega_j^{(k)}, \widetilde{E}_j^{(k)}, B_j^{(k)})$ yields a finite collection of new admissible triples. The union (over $j$) of all these collections forms the collection of level $k+1$ triples  $(\Omega_{\ell}^{(k+1)}\!,\! \widetilde{E}_{\ell}^{(k+1)}\!,\! B_{\ell}^{(k+1)}).$

For a fixed $\ell$, every $x\in \widetilde{E}_{\ell}^{(k+1)}$ has at least $\widetilde{T}^{(k+1)}$ good scales less than or equal to $\varepsilon \rho_{\ell}^{(k+1)}/4$,
%\begin{equation}\label{goodscaleloss}\widetilde{T}^{(k+1)}> \widetilde{T}^{(k)} - q -\log_2M -\log_2 1/\varepsilon-3.\end{equation}
where $\rho_{\ell}^{(k+1)}$ is the radius of $B_{\ell}^{(k+1)}$.  This follows from property (b) of the construction.%Also, since the cells $\Omega_j^{(k)}$ are disjoint, it follows from (\ref{smalllosest}) that
%\begin{equation}\label{levelloss2}
%\mu\Bigl(\bigcup_{\ell} \widetilde{E}_{\ell}^{(k+1)}\Bigl)\geq \mu\Bigl(\bigcup_j \widetilde{E}_j^{(k)}\Bigl) -C_{21}\varepsilon \mu\Bigl(\bigcup_j \Omega_j^{(k)}\Bigl)/\Delta.
%\end{equation}

Note that if $\ell\neq n$, then $\operatorname{dist}(\Omega_{\ell}^{(k+1)}, \Omega_{n}^{(k+1)})\geq \varepsilon \max(\rho^{(k+1)}_{\ell}, \rho^{(k+1)}_{n})$.  To see this, note that each level $k+1$ cell $\Omega_{\ell}^{(k+1)}$ has a unique parent cell $\Omega_j^{(k)}$.  If two level $k+1$ cells originate from the same parent cell, then the required separation follows directly from the construction (see property (d) above).  If they have different parent cells, then the claim follows from the separation between those parent cells, since $\rho_j^{(k)}\geq \rho_{\ell}^{(k+1)}$ whenever $\Omega_j^{(k)}$ is the parent cell of $\Omega_{\ell}^{(k+1)}$.

To begin the iteration, assume that $T>\widetilde{T}^{(0)}$.  Let $\widetilde{E}^{(0)}_1 = E$, and put $\displaystyle\rho^{(0)}_1 = 2\text{diam}(E)+\frac{ \,4}{\varepsilon}.$  Define $B^{(0)}_1$ to be a ball of radius $\rho_1^{(0)}$, centred at a point of $E$.  Let $\Omega_1^{(0)}$ be the closed $\varepsilon \rho_1^{(0)}$-neighbourhood of $E$.  The initial triple $(\Omega_1^{(0)}, \widetilde{E}_1^{(0)}, B_1^{(0)})$ is admissible provided $\varepsilon \leq 1/2$.  Indeed, for such $\varepsilon$ we have $\varepsilon \rho^{(0)}_1 +\text{diam}(E)< \rho^{(0)}_1$, and hence $\Omega_1^{(0)}\subset B^{(0)}_1$.  Note that the maximal good scale at each point of $E$ is smaller than $\varepsilon \rho_1^{(0)}/4$ (this is merely the statement that $\varepsilon \rho_1^{(0)}/4>1$).  %Hence, as long as  $T> q+\log_2 M+\log_21/\varepsilon+3$, we may apply the algorithm to obtain the first level sets $(\Omega^{(1)}_j, \widetilde{E}_j^{(1)}, B_j^{(1)})_j$.

Iterating the construction $N$ times from this initial triple, we  obtain the levels $(\Omega_j^{(k)}, \widetilde{E}_j^{(k)}, B_j^{(k)})_j$, for $k=0,\dots, N$.   The sets $\Omega_j^{(k)}$ are the \emph{level $k$ \!Cantor cells}.  %It will be convenient to set $\Omega_{1}^{(0)} = \mathbf{R}^d$.

%In order for the $N$-fold iteration not to terminate prematurely,  we need to ensure that there are enough good scales at each point of $E$, and also that $E$ is not exhausted by exceptional sets.

%(i) \textit{Good scales.}  The maximal good scale at each $x\in E$ is admissible for the initial top cover in the construction.  In each subsequent level, it follows from (\ref{goodscaleloss}) that $q+\log_2M - \log_2 \varepsilon +3$ additional good scales are sufficient to execute the algorithm.  Since there are $N$ levels, we demand that $T$ is at least $ N( q +\log_2 M - \log_2 \varepsilon + 3 )$.  Since $q$ is the dominant term, it suffices to require
%\begin{equation}\label{Tcond}T \geq\Bigl\lfloor \frac{C_{22} N}{\varepsilon}\exp\Bigl(\frac{C_{16}{M}^{s\beta}}{\delta^{\beta}}\Bigl)\Bigl\rfloor +1,\end{equation}
%with $\beta$ as in Proposition \ref{qualv}.

%We then repeat the construction within each pair $\widetilde{B}_j^1$ and $\Omega_j^1$, which are contained in a ball $B_j$ of radius $\rho_j$.  The whole top to bottom construction is repeated $N$ times.
%We thereby obtain collections of cells $(\widetilde{B}_j^k)_j$ and $(\Omega_j^{(k)})_j$ for $1\leq k \leq N$, such that each $(k+1)$-st level cell $\Omega_{\ell}^{k+1}$ is contained inside a unique $k$-th level cell $\Omega_{j}^k$. %and measures $\mu_j^k$ are obtained for $k=1,\dots N$.  We define $m_j^k = \mu_j^k(\mathbf{R}^d) = \mu_j^k(\Omega_j^{(k)})$.
%It will be convenient to let $\Omega^0_1 = \mathbf{R}^d$.% and $\mu^0_1 = \mu$.

The condition that $T>\widetilde{T}^{(0)}=N ( q +\log_2 M + \log_2 \frac{1}{\varepsilon} + 3 )$ guarantees a sufficient number of good scales at any point in $E$ to construct the $N$ levels of the Cantor set. Since $q$ is the dominant term, it suffices to require that $T$ satisfies
\begin{equation}\label{Tcond}T \geq \frac{C_{22} N}{\varepsilon}\exp\Bigl(\frac{C_{16} 2^{s\beta}M^{s\beta}}{\delta^{\beta}}\Bigl),\end{equation}
with $\beta$ as in Proposition \ref{qualv}.

Let us now  place a restriction on $\varepsilon$ to ensure that the majority of the measure of $E$ is preserved after the $N$-fold iteration.  To this end, note that for each $k=1, \dots, N$, it follows from property (c) of the construction that
\begin{equation}
\mu\Bigl(\bigcup_{\ell} \widetilde{E}_{\ell}^{(k)}\Bigl)\geq \mu\Bigl(\bigcup_j \widetilde{E}_j^{(k-1)}\Bigl) -\frac{C_{21}\varepsilon}{\Delta} \sum_j\mu(\Omega_j^{(k-1)}).
\end{equation}
Since the cells $\Omega_j^{(k-1)}$ are disjoint, we have
$$\sum_j\mu(\Omega_j^{(k-1)})\leq \mu (\mathbf{R}^d)\leq \frac{\mu(E)}{\gamma},
$$
and recalling that $\widetilde{E}^{(0)}_1=E$, we inductively obtain
$$\mu\Bigl(\bigcup_{\ell} \widetilde{E}_{\ell}^{(k)}\Bigl)\geq \bigl(1-\frac{k C_{21}\varepsilon }{\Delta \gamma}\bigl)\mu(E),$$for any $k=0,\dots, N$.
%(ii) \textit{Loss of $E$.}  As a result of (\ref{levelloss2}), in each level of the construction a set of measure at most $C_{21}\varepsilon \mu(\Omega_1^{(0)})/\Delta$ is lost from $E$.  (For all $j$ and $k$, the cell $\Omega_j^{(k)}$ is contained in  $\Omega_1^{(0)}$).
Suppose that $\varepsilon$ satisfies
\begin{equation}\label{epscond}N\frac{C_{21} \varepsilon}{\Delta\gamma} <\frac{1}{2}.
\end{equation}
Then we see that $E$ will not be exhausted after constructing the $N$ levels.  Moreover, we have the estimate
\begin{equation}\label{Ethick}\mu\Bigl(\bigcup_j \widetilde{E}_j^{(N)}\Bigl)\geq \frac{1}{2}\mu(E).\end{equation}
%Then from an iterative use of (\ref{smalllosest}),

Let $F=\bigcup_j \Omega_j^{(N)}$, and define
$\mu' = \chi_F \mu$ to be the rarefied measure associated to the $N$-th Cantor level.
%Denote $m=\mu(\mathbf{R}^d)$ and $m' = \mu'(\mathbf{R}^d)$.
We will make regular use of the following properties of the measure $\mu'$.

(1) \emph{Domination}.  The measure $\mu'$ is dominated by $\mu$.

(2) \textit{Separation in the support}. Suppose $\Omega$ is a level $k$ Cantor cell, and $B= B(x, \rho)$ is the ball in the bottom cover of the $k$-th level that gave birth to $\Omega$.  Then we have
\begin{equation}\label{suppsep}\operatorname{dist}(\operatorname{supp}(\mu')\backslash \Omega, \Omega)\geq\varepsilon \rho.\end{equation}
This property is an immediate consequence of the separation between the Cantor cells $\Omega_j^{(k)}$ for each level $k=1,\dots, N$.

%To see this, note that if $\Omega'$ is another level $k$ cell, then $\operatorname{dist}(\Omega,\Omega')\geq \varepsilon \rho$. This is immediate from construction if the two cells were born out of the same level $k-1$ cell.  If $\Omega$ and $\Omega'$ originate from different level $k-1$ cells, then the claim follows from the separation property between those parent cells (the spacing will in fact be much greater in this instance).

(3) \textit{Significant mass}.  Since $\widetilde{E}_j^{(N)}\subset \Omega_j^{(N)}$, the inequality (\ref{Ethick}) implies that $$\mu'(\mathbf{R}^d) \geq \mu(E)/2\geq \gamma \mu(\mathbf{R}^d)/2.$$
%Since the measures are nested, and the sets $\Omega_j^{(k)}$ are disjoint in $j$ for a fixed $k$,
%$$m\geq\sum_j m_j^k \geq m', \text{ for any }k=0,\dots N.
%$$
%This concludes the fractal construction.

%We shall now work with the Riesz transform in $L^2(\mu')$.

\section{The $L^2(\mu')$ estimates}\label{l2ests}

In this section we will show that assumption (\ref{levelsetfalse}) implies that the norm of $R^{\#}(\mu')$ in $L^2(\mu')$ is large.  From this we will conclude the proof of Theorem \ref{thm1}.

\subsection{Reduction to $L^2(\mu')$ estimates} We first introduce the partial Riesz transforms.  For $x\in \bigcup_j \Omega_j^{(k)}$, define $\Omega^{(k)}(x)$ to be the unique level $k$ cell containing $x$.  The partial Riesz transform $R^{(k)}(\mu')$ is defined by
\begin{equation}\label{partRiesz}R^{(k)}(\mu')(x) = \int_{\Omega^{(k)}(x)\backslash \Omega^{(k+1)}(x)}\frac{y-x}{|y-x|^{1+s}} d\mu'(y),
\end{equation}
for $x\in \bigcup_j \Omega_j^{(k+1)}$.

We will see that Theorem \ref{thm1} follows from the subsequent three propositions.

The first proposition concerns the boundedness of the sum of partial Riesz transforms in $L^2(\mu')$.

\begin{prop}\label{prop1}  %Suppose $\varepsilon$, $M$ and $\delta$ are chosen to satisfy
%\begin{equation}\label{prop1cond}
%\delta<\frac{\varepsilon^s}{2M^s}.
%\end{equation}
The following inequality holds:
\begin{equation}\label{claim1}\int_{\mathbf{R}^d} \Bigl|\sum_{k=0}^{N-1}R^{(k)}(\mu')\Bigl|^2 d\mu' \leq 2\Bigl(C_3+\frac{4M^{2s}\delta^{2}}{\varepsilon^{2s}}\Bigl) \mu'(\mathbf{R}^d).
\end{equation}
\end{prop}

The second proposition states that the partial Riesz transforms are almost orthogonal to one another.

\begin{prop}\label{prop2}  There exists a constant $K_1=K_1(s,d)>1$ such that for each $k=0, \dots ,N-2$,
\begin{equation}\label{claim2}\begin{split}
\Bigl|&\int_{\mathbf{R}^d} \Bigl( R^{(k)}(\mu'), \sum_{j=k+1}^{N-1}R^{(j)}(\mu')\Bigl)d\mu'\Bigl| \\
&\leq K_1\sqrt{\mu'(\mathbf{R}^d)}\Bigl(\frac{M^s\delta}{\varepsilon}+\frac{1}{M}\Bigl)\sum_{j=k+1}^{N-1}||R^{(j)}(\mu')||_{L^2(\mu')}.
\end{split}\end{equation}
%With $K>0$ the constant in Proposition \ref{prop3}, suppose the parameters $\varepsilon$, $M$ and $\delta$ are chosen such that
%\begin{equation}\label{prop2cond}
%\frac{M^s\delta}{\varepsilon}+\frac{1}{M} \leq \frac{\Delta^2\gamma^2}{C_{24}NK}.
%\end{equation}
%Then for each $k=0, \dots ,N-2$,
%\begin{equation}\label{claim2}\begin{split}
%\Bigl|\int_{\mathbf{R}^d} \Bigl( R^{(k)}(d\mu'), &\sum_{j=k+1}^{N-1}R^{(j)}(d\mu')\Bigl)d\mu'\Bigl| \\
%&\leq \frac{\gamma^2\Delta^2\sqrt{ \mu'(\mathbf{R}^d)}}{2KN}\sum_{j=k+1}^{N-1}||R^{(j)}(d\mu')||_{L^2(d\mu')}.
%\end{split}\end{equation}
\end{prop}

The third proposition, which is the heart of the argument, concerns the size of each partial Riesz transform in $L^2(\mu')$.

\begin{prop}\label{prop3}  There exists a constant $K_2=K_2(s,d)>1$ such that if $\varepsilon$, $M$ and $\delta$ are chosen satisfying the inequalities
\begin{equation}\label{prop3cond}\frac{M^{2s}\delta}{\varepsilon^{d+s}}+\frac{1}{M} \leq\frac{\gamma^4\Delta^4}{K_2} \text{ and } \frac{M^s\delta}{\varepsilon^d}\leq1,
\end{equation}
then for each $k=0, \dots ,N-1$,
\begin{equation}\label{claim3}
\int_{\mathbf{R}^d}|R^{(k)}(\mu')|^2 d\mu' \geq \frac{1}{K_2}\cdot\gamma^4\Delta^4\mu'(\mathbf{R}^d).
\end{equation}
%\begin{equation}\label{prop3cond}\frac{M^s\delta}{\varepsilon^{d+s}}+\frac{1}{M} \leq\frac{\gamma^4\Delta^4}{C_{35}}.
%\end{equation}
%Then there is a constant $K>0$, depending only on $s$, such that for each $k=0, \dots ,N-1$
%\begin{equation}\label{claim3}
%\int_{\mathbf{R}^d}|R^{(k)}(d\mu')|^2 d\mu' \geq \frac{1}{K^2}\cdot\gamma^4\Delta^4\mu'(\mathbf{R}^d),
%\end{equation}
\end{prop}

Taking these three Propositions for granted for the time being, let us conclude the proof of Theorem \ref{thm1}.

\begin{proof}[Proof of Theorem \ref{thm1}]%Let us assume the conditions on the parameters $\varepsilon$, $\delta$ and $M$ from the previous three propositions are in force.  We claim that
%\begin{equation}\label{propcombine}
%\int_{\mathbf{R}^d} \Bigl|\sum_{k=0}^{N-1}R^{(k)}(d\mu')\Bigl|^2 d\mu' \geq \frac{N}{2K^2}\gamma^4\Delta^4 \mu'(\mathbf{R}^d).
%\end{equation}
%Indeed, expand the square in the left hand side and applying Proposition \ref{prop2}, we obtain
Suppose that $\varepsilon$, $M$ and $\delta$ are chosen to satisfy $$\frac{M^s\delta}{\varepsilon}+\frac{1}{M} \leq \frac{\Delta^2\gamma^2}{4N K_1 \sqrt{K_2}}.$$
Then it follows from Proposition \ref{prop2} that
\begin{equation}\begin{split}\nonumber\int_{\mathbf{R}^d} \Bigl|\sum_{k=0}^{N-1}&R^{(k)}(\mu')\Bigl|^2 d\mu' \\
&\geq \sum_{k=0}^{N-1}||R^{(k)}(\mu')||_{L^2(\mu')}\cdot\Bigl(||R^{(k)}(\mu')||_{L^2(\mu')} - \frac{\gamma^2\Delta^2\sqrt{ \mu'(\mathbf{R}^d)}}{2\sqrt{K_2}}\Bigl).
\end{split}\end{equation}
Assuming the conditions (\ref{prop3cond}) are in force, applying Proposition \ref{prop3} yields $||R^{(k)}(\mu')||_{L^2(\mu')} -\frac{\gamma^2\Delta^2\sqrt{ \mu'(\mathbf{R}^d)}}{2\sqrt{K_2}}\geq \tfrac{1}{2}||R^{(k)}(\mu')||_{L^2(\mu')}$, and therefore
\begin{equation}\label{propcombine}
\int_{\mathbf{R}^d} \Bigl|\sum_{k=0}^{N-1}R^{(k)}(\mu')\Bigl|^2 d\mu' \geq \frac{N}{2K_2}\gamma^4\Delta^4 \mu'(\mathbf{R}^d).
\end{equation}
Put $N=\lfloor (8C_3 K_2)/(\Delta^4\gamma^4)\rfloor +1.$ If $\varepsilon$, $M$ and $\delta$ are chosen to satisfy $2M^s\delta/\varepsilon^s\leq \sqrt{C_3}$, then (\ref{propcombine}) is in contradiction with Proposition \ref{prop1}.  As a result, the assumption (\ref{levelsetfalse}) is false.   It remains to make a consistent choice of $\varepsilon, M$ and $\delta$, and consequently determine an admissible size of $T$.

Recall that (\ref{epscond}) is the only restriction on $\varepsilon$ in terms of $N$ only.  A suitable choice of $\varepsilon$ is therefore $\varepsilon = c \gamma\Delta/N = c \Delta^5\gamma^5$.   We now determine $M$, and subsequently  $\delta$, according to the following four conditions:
\begin{equation}\nonumber\begin{split}
&\frac{2M^s\delta}{\varepsilon^s}\leq \sqrt{C_3}, \;\; \frac{M^s\delta}{\varepsilon}+\frac{1}{M} \leq \frac{\Delta^2\gamma^2}{N K_1 \sqrt{K_2}},\\
&\frac{M^s\delta}{\varepsilon^d}\leq 1, \,\text{ and }\frac{M^{2s}\delta}{\varepsilon^{d+s}}+\frac{1}{M} \leq\frac{\gamma^4\Delta^4}{K_2}.
\end{split}\end{equation}
First pick $M$ subject to
\begin{equation}\label{Mcond}M\geq 2\max\Bigl(\frac{N K_1 \sqrt{K_2}}{\Delta^2\gamma^2},\frac{K_2}{\gamma^4\Delta^4}\Bigl).
\end{equation}
Then choose $\delta$ satisfying
\begin{equation}\label{deltacond}\delta \leq \frac{1}{2}\min\Bigl(\frac{\varepsilon^s\sqrt{C_3}}{M^s}, \frac{\varepsilon^d}{M^s}, \frac{\Delta^2\gamma^2\varepsilon}{M^sN K_1 \sqrt{K_2}},\frac{\gamma^4\Delta^4\varepsilon^{d+s}}{M^{2s}K_2}\Bigl).
\end{equation}
%In terms of the order of magnitude as $\Delta$ and $\gamma$ tends to zero, the most restrictive condition on $M$ is the second one.  Hence we set $M = C\Delta^{-6}\gamma^{-6}$.   Subsequently, the most restrictive condition on $\delta$ is the third in the list.  We set $\delta = c\Delta^{4+5d+11s} \gamma^{4+5d+11s}$.

Since $N$ and $\varepsilon$ are power functions in $\Delta$ and $\gamma$, we can choose $M$ and $\delta$ to be power functions in $\Delta$ and $\gamma$ as well.  (A computation shows that we may choose $M = C\Delta^{-6}\gamma^{-6}$ and then $\delta = c\Delta^{4+5d+17s} \gamma^{4+5d+17s}$.)

As a result of (\ref{Tcond}),  we assert the existence of positive constants $\alpha = \alpha(s,d)$ and $C=C(s,d)$, such that (\ref{levelsetfalse}) must be false if
$T \geq \exp[(C\Delta^{-1}\gamma^{-1})^{1/\alpha}].$ Theorem \ref{thm1} follows.  %Such a choice of $T$ determines $\gamma \in (0,1]$ provided $T\geq \exp[(C/\Delta)^{1/\alpha}]$, in which case we can pick $\gamma = C/(\Delta \log^{\alpha}(T)).$
%The contradiction of (\ref{levelsetfalse}) with this choice of $\gamma$ proves Theorem \ref{thm1} in the case when $T\geq\exp[(C/\Delta)^{1/\alpha}]$.  The case when $T<\exp[(C/\Delta)^{1/\alpha}]$ is trivial.
\end{proof}

We turn now to proving the propositions. Propositions \ref{prop1} and \ref{prop2} are quite simple to prove, but Proposition \ref{prop3} requires some work.

\subsection{Proof of Proposition \ref{prop1}} The $T(1)$-theorem (quoted as Theorem \ref{T1thm} in this paper) states that the operator $R^{\#}(\,\cdot\,\mu)$ is bounded in $L^2(d\mu)$, with operator norm at most $\sqrt{C_3}$.  Since $\sum_j \chi_{\Omega_j^{(N)}}\in L^2(\mu)$ with $L^2(\mu)$ norm equal to %$||\sum_j \chi_{\Omega_j^{(N)}}||_{L^2(\mu')} =
$\sqrt{\mu'(\mathbf{R}^d)}$, we deduce that
\begin{equation}\label{l2restbd}
\int_{\mathbf{R}^d}|R^{\#}(\mu')|^2d\mu' \leq \int_{\mathbf{R}^d}|R^{\#}(\mu')|^2d\mu \leq C_3 \mu'(\mathbf{R}^d).
\end{equation}Proposition \ref{prop1} is a simple consequence of (\ref{l2restbd}) along with the following lemma.
\begin{lem} For any $x\in \operatorname{supp}(\mu')$, the following inequality holds:
\begin{equation}\label{partialsequalmax}
 \Bigl|\sum_{k=0}^{N-1}R^{(k)}(\mu')(x)\Bigl| \leq R^{\#}(\mu')(x) + \frac{2M^s\delta}{\varepsilon^s}
 \end{equation}
\end{lem}

\begin{proof}  Suppose $\Omega^{(N)}(x) = \Omega_j^{(N)}$ for some $j$.  Consider the ball $B_j = B(x_j, \rho_j)$ in the bottom cover of the $N$th level that gave birth to $\Omega_j^{(N)}$.  Since $x\in \Omega^{(N)}_j$, it immediately follows that
$$\Bigl|\sum_{k=0}^{N-1}R^{(k)}(\mu')(x)\Bigl| \leq R^{\#}(\mu')(x) + \Bigl|\int_{2B_j\backslash \Omega^{(N)}(x)}\frac{y-x}{|y-x|^{1+s}} d\mu'(y)\Bigl|.
$$
In order to estimate the second integral, observe from (\ref{suppsep}) that the integrand is pointwise at most $1/(\varepsilon\rho_j)^s$.  Therefore, assuming $M>2$,
$$\Bigl|\int_{2B_j\backslash \Omega^{(N)}(x)}\frac{y-x}{|y-x|^{1+s}} d\mu'(y) \Bigl|\leq \mu'(2B_j)\frac{1}{(\varepsilon\rho_j)^s}\leq\frac{\mu(MB_j)}{(\varepsilon\rho_j)^s}.
$$
Appealing to (\ref{bottomlowdens}), we obtain the required estimate.
\end{proof}

To prove Proposition \ref{prop1}, we apply (\ref{partialsequalmax}) to obtain
$$\int_{\mathbf{R}^d}\Bigl|\sum_{k=0}^{N-1}R^{(k)}(\mu')\Bigl|^2 d\mu'\leq \int_{\mathbf{R}^d}\Bigl(R^{\#}(\mu')+ \frac{2M^s\delta}{\varepsilon^s}\Bigl)^2 d\mu'.
$$
Since $(a+b)^2\leq 2(a^2+b^2)$ for $a,b\in \mathbf{R}$, the desired inequality follows from (\ref{l2restbd}).

\subsection{Proof of Proposition \ref{prop2}}  We begin with a simple oscillation estimate.

\begin{lem}\label{osclem}
Let $\nu$ be a signed measure, and let $\Omega\subset B=B(z,\rho)$ be such that $\operatorname{dist}(\Omega, \operatorname{supp}(\nu))\geq\varepsilon\rho$.  Then,
\begin{equation}\label{oscstate}
\text{osc}_{\,\Omega}\,R(\nu)\leq \frac{2}{(\varepsilon\rho)^s}|\nu|(B(z, M\rho/3)) + \frac{C_{23}}{M}\sup_{r>0}\frac{|\nu|(B(z,r))}{r^s}.
\end{equation}
Also, if $\sigma$ is a signed measure supported on $\Omega$ such that $\sigma(\Omega)=0$, then
\begin{equation}\label{dualosc}
\Bigl|\int_{\mathbf{R}^d}|R(\sigma)|\,d\nu \Bigl|\leq \Bigl[\frac{2}{(\varepsilon\rho)^s}|\nu|(B(z, M\rho/3)) \!+\! \frac{C_{23}}{M}\sup_{r>0}\frac{|\nu|(B(z,r))}{r^s}\Bigl]|\sigma|(\Omega).
\end{equation}
\end{lem}

\begin{proof}  For points $x, x'\in \Omega$, we wish to estimate the quantity $|R(\nu)(x) - R(\nu)(x')|$.  To this end, note that for each $y\in \operatorname{supp}(\nu)$, we have
\begin{equation}\label{Eosc1}
\Bigl|\frac{y-x}{|y-x|^{1+s}} - \frac{y-x'}{|y-x'|^{1+s}}\Bigl|\leq \frac{2}{(\varepsilon\rho)^s}.
\end{equation}
In addition, if $M\geq 6$, then
\begin{equation}\label{Eosc2}
\Bigl|\frac{y-x}{|y-x|^{1+s}} - \frac{y-x'}{|y-x'|^{1+s}}\Bigl|\leq C\frac{\rho}{|z-y|^{1+s}}
\end{equation}
for $y\in \mathbf{R}^d\backslash B(z, M\rho/3)$.

Integrating estimates (\ref{Eosc1}) and (\ref{Eosc2}) with respect to $|\nu|$ over the sets $B(z, M\rho/3)$ and $\mathbf{R}^d\backslash B(z, M\rho/3)$ respectively, we arrive at (\ref{oscstate}).

To prove (\ref{dualosc}), note that for any $x'\in \Omega$ and $y\in \mathbf{R}^d$, we have $$|R(\sigma)(y)| \leq \int_{\Omega}\Bigl|\frac{y-x}{|y-x|^{1+s}} - \frac{y-x'}{|y-x'|^{1+s}}\Bigl| \,d|\sigma|(x).$$  Using (\ref{Eosc1}) and (\ref{Eosc2}), we obtain
   $$|R(\sigma)(y)|\leq \frac{2|\sigma|(\Omega)}{(\varepsilon \rho)^s}\chi_{B(z, M\rho/3)}(y) +  \frac{C\rho|\sigma|(\Omega)}{|z-y|^{1+s}}\chi_{\mathbf{R}^d\backslash B(z, M\rho/3)}(y),$$  for any $y\in \operatorname{supp}(\nu)$.  Integrating this inequality over $|\nu|$, we arrive at (\ref{dualosc}).%estimate (\ref{Eosc1}) with respect to $\nu$ over the ball $B(z, M\rho/3)$, and integrating (\ref{Eosc2}) with respect to $\nu$ over the compliment of $B(z, M\rho/3)$ yields the desired statement (\ref{oscstate}).
%To see (\ref{dualosc}),  write $\int_{\mathbf{R}^d}|R(\sigma)|d\nu = \int_{\mathbf{R}^d}R(\sigma)\cdot \textbf{g}d\nu$, with $\mathbf{g} = \tfrac{R(\sigma)}{|R(\sigma)|}$ (if $R(\sigma)=0$, then define $\textbf{g}=0$).  Taking the adjoint, our goal is to estimate the absolute value of $\int_{\mathbf{R}^d}R^{*}(\textbf{g}d\nu)d\sigma$.  Notice that (\ref{oscstate}) holds for the adjoint Riesz transform of a signed vector measure $\nu$ satisfying the hypotheses of the lemma, with $|\nu| = (|\nu_1|^2 +\dots +|\nu_d|^2)^{1/2}$ ($\nu_j$ denoting the $j$-th component of $\nu$).  The estimate (\ref{dualosc}) now follows from the mean zero property of $\sigma$, since $|\textbf{g}|\leq 1$.
%$\int_{\mathbf{R}^d}|R_j(\sigma)|d\nu = \int_{\mathbf{R}^d}R_j(\sigma)\cdot d\nu_1 - \int_{\mathbf{R}^d}R_j(\sigma) d\nu_2$, where $\nu_1 = \nu\chi_{\{x: R_j(\sigma)\geq 0\}}$, and $\nu_2 = \nu-\nu_1$.  %Taking the adjoint, we wish to estimate the absolute value of $\int_{\mathbf{R}^d}R^{*}_j(\nu_1-\nu_2) d\sigma$.  From its proof above, we see that (\ref{Eosc1}) holds with $R$ replaced by $R_j$.  Since $\sigma$ has mean zero, the required bound follows from (\ref{Eosc1}), as $|\nu| = |\nu_1| + |\nu_2|$.
\end{proof}

By inspection of the proof of Lemma \ref{osclem}, we obtain the following: if $\nu$, $\Omega$ and $B$ satisfy the assumptions of Lemma \ref{osclem}, and if $g$ is a bounded vector field, then
\begin{equation}\label{adjoscest}
\text{osc}_{\,\Omega}\,R^{*}(g\nu)\leq ||g||_{L^{\infty}}\Bigl[\frac{2}{(\varepsilon\rho)^s}|\nu|(B(z, M\rho/3)) + \frac{C_{23}}{M}\sup_{r>0}\frac{|\nu|(B(z,r))}{r^s}\Bigl].
\end{equation}

We turn now to the proof of Proposition \ref{prop2}.

\begin{proof}[Proof of Proposition \ref{prop2}]  On account of the first part of Lemma \ref{osclem}, we claim that
\begin{equation}\label{osccellest}
\text{osc}_{\Omega^{(k+1)}_j} R^{(k)}(\mu') \leq C_{24}\Bigl(\frac{M^s\delta}{\varepsilon^s}+\frac{1}{M}\Bigl).
\end{equation}
To see this, note that $\Omega_j^{(k+1)}\subset \Omega_{\ell}^{(k)}$ for some choice of $\ell$.  Let $\nu =  \chi_{\Omega_{\ell}^{(k)}\backslash \Omega_{j}^{(k+1)}}\cdot\mu'$.  Then for any $x\in \Omega_j^{(k+1)}$, we have $R(\nu)(x) = R^{(k)}(\mu')(x)$. Let $B=B(x_j, \rho_j)$ be the ball in the level $k+1$ bottom cover that gave birth to $\Omega_j^{(k+1)}$.  Then $\operatorname{dist}(\operatorname{supp}(\nu), \Omega_j^{(k+1)})\geq \varepsilon \rho_j$.  Applying Lemma \ref{osclem}, and estimating the right hand side of (\ref{oscstate}) with inequalities (\ref{bottomlowdens}) and (\ref{growth}) respectively, we get (\ref{osccellest}).

Now, fix $x\in \Omega^{(k+1)}_j\cap F$, and observe that $$\sum_{\ell=k+1}^{N-1}R^{(\ell)}(\mu')(x) = \int_{\Omega^{(k+1)}_j\backslash \Omega^{(N)}(x)} \frac{y-x}{|y-x|^{1+s}}d\mu'(y).$$
As the support of $\mu'$ is contained in $F$, we may write $\Omega_j^{(k+1)}\backslash \Omega^{(N)}(x)$ as the set of $y\in \Omega^{(k+1)}_j\cap F$ such that $\Omega^{(N)}(y)\neq \Omega^{(N)}(x)$.  Integrating over $x\in \Omega^{(k+1)}_j\cap F$ with respect to $\mu'$, we thereby obtain
\begin{equation}\label{meanzerotail}
\int_{\Omega_j^{(k+1)}}\!\sum_{\ell = k+1}^{N-1}\! R^{(\ell)}(\mu') d\mu' = \!\!\!\!\!\!\!\!\!\!\!\iint\limits_{\substack{(x,y) \,\in \,\Omega^{(k+1)}_j\times \Omega^{(k+1)}_j\!,\\ \Omega^{(N)}(x)\neq \Omega^{(N)}(y)}}\!\!\!\!\frac{y-x}{|y-x|^{1+s}}d\mu'(y)d\mu'(x)\!=\!0,
\end{equation}
since we are integrating an anti-symmetric function over a symmetric set.
Combining the oscillation estimate (\ref{osccellest}) with the mean zero property (\ref{meanzerotail}), we estimate
\begin{equation}\begin{split}
\Bigl|\int_{\Omega_j^{(k+1)}} &\Bigl( R^{(k)}(\mu'),  \sum_{\ell = k+1}^{N-1} R^{(\ell)}(\mu')\Bigl)d\mu'\Bigl|\\
&\leq C_{24}\Bigl(\frac{M^s\delta}{\varepsilon^s}+\frac{1}{M}\Bigl)\sum_{\ell= k+1}^{N-1}\int_{\Omega_j^{(k+1)}}|R^{(\ell)}(\mu')|d\mu'.
\end{split}\end{equation}
Summing these inequalities over $j$, we see that the estimate holds with the integration on the left and right hand sides taken over $\mathbf{R}^d$.  Applying the Cauchy-Schwarz inequality, we obtain (\ref{claim2}) with $K_1=C_{24}$.\end{proof}

We now turn to the proof of Proposition \ref{prop3}.   The proof follows \cite{ENV11}, but there are a couple of additional considerations needed to make the argument quantitative.  For the benefit of the reader we repeat the details, and so devote a full chapter to the proof.

\section{The proof of Proposition \ref{prop3}}\label{prop3sec}  For $n\in \{0, \dots, N-1\}$, consider a fixed Cantor cell $\Omega$ at level $n$. We shall set $m = \mu(\Omega)$ and $m'=\mu'(\Omega)$.   Let $\{\Omega_j\}_j$ denote the collection of those level $n+1$ cells that are contained in $\Omega$.  Each Cantor cell $\Omega_j$ is born out of a bottom cover ball $B_j$ of radius $\rho_j$.  We will work primarily within the cell $\Omega$, and then sum over all the level $n$ Cantor cells to prove Proposition \ref{prop3}.

It will be convenient to introduce a globally Lipschitz function $V(x)$, which behaves like $|x|^2$ for small values of $|x|$.  To this end, let $v\in C^{\infty}([0,\infty))$ be such that $v(0)=0$, $v'(0)=0$, $v''(t)=2$ for $t\in [0,1]$, $v''(t)$ is non-increasing in $t$, and $v''(t)=0$ for $t\geq 2$.  The function $v$ is convex, increasing, and satisfies $\min(t,t^2)\leq v(t)\leq t^2$ for all $t\in [0,\infty)$.

We will need a couple of additional consequences of the assumptions on $v$; namely,  $v'(t)^2\leq 4 v(t)$ and $v(at)\leq a^2v(t)$ for any $t>0$ and $a>1$.  To see these two inequalities, note that $v'(t) \geq tv''(t)$, as $v''(t)$ is non-increasing and $v'(0)=0$.  Integration of this inequality yields $v(t)\geq \int_0^t \tau v''(\tau)d\tau = t v'(t)-v(t)$, and thus $2v(t)\geq tv'(t)$ (or alternatively $(\log v(t))' \leq 2/t$).  Hence $4v(t)^2\geq t^2 v'(t)^2\geq v(t)v'(t)^2$, and the first inequality is proved.  Integrating $(\log v(\tau))' \leq 2/\tau$ between $\tau = t$ and $\tau = at$, we obtain the second inequality.

Now define $V(x) = v(|x|)$ for $x\in \mathbf{R}^d$. Then V is convex, and $\min(|x|, |x|^2)\leq V(x)\leq |x|^2$ for all $x\in \mathbf{R}^d$.  We also have $|\nabla V|\leq  \min(4, 2\sqrt{V})$, and $V(a|x|)\leq a^2 V(|x|)$ for all $a>1$ and $x\in \mathbf{R}^d$.

Our aim is to derive a lower bound for $\int_{\Omega}V(R^{(n)}(\mu')) d\mu'$.

We begin by showing that it suffices to work with a smooth approximation of $\mu'$.

\subsection{A smooth approximation of $\mu'$}  Recall that inside each Cantor cell $\Omega_j$, there is an open ball $\widetilde{\Omega}_j$ of radius $\varepsilon\rho_j$.  Define $\varphi_j\in C^{\infty}_0(\widetilde{\Omega}_j)$ so that
\begin{equation}\label{phijprops}\varphi_j\geq 0, \, \int_{\mathbf{R}^d} \varphi_j dm_d = \mu'(\Omega_j), \text{ and }||\varphi_j||_{L^{\infty}}\leq \frac{C_{25}\mu'(\Omega_j)}{(\varepsilon\rho_j)^d}.
\end{equation}
Let $\tilde{\mu} = \sum_j \tilde{\mu}_j$ where $\tilde\mu_j = \varphi_j m_d$.  By construction, $\tilde\mu(\mathbf{R}^d) = \mu'(\Omega)= m'$ and $\operatorname{supp}(\tilde\mu)\subset \Omega$.  The key properties of $\tilde{\mu}$ are contained in the following lemma.
\begin{lem}\label{smoothgrowthprop}The following two properties hold:

\indent (i)  Suppose $M^s\delta/\varepsilon^d\leq1$.  Then
\begin{equation}\label{anyballsm}
\tilde{\mu}(B(z,t)) \leq C_{26} t^s, \text{ for any ball } B(z,t).
\end{equation}
\indent (ii)  For a bottom cover ball $B_j$, one has
\begin{equation}\label{balljsm}
\tilde{\mu}\bigl(\tfrac{M}{3}B_j\bigl) \leq \frac{C_{27}M^{2s}\delta}{\varepsilon^d}\rho_j^s.
\end{equation}
\end{lem}

\begin{proof}  Fix a ball $B=B(z,t)$, and write
$$\tilde{\mu}(B) = \sum_{\rho_j\leq t} \tilde{\mu}(B\cap \Omega_j) +\tilde{\mu}\Bigl(\bigcup_{\rho_j>t}(B\cap \Omega_j)\Bigl).
$$
For each $j$ with $B\cap \Omega_j \neq \varnothing$ and $\rho_j\leq t$, the inclusion $\Omega_j\subset 3B$ holds.  Since the cells $\Omega_j$ are pairwise disjoint, we have
$$\sum_{\rho_j\leq t} \tilde{\mu}(B\cap \Omega_j) \leq \mu'(3B).
$$
To estimate the second term, note that for any cell $\widetilde{\Omega}_j$ with $\rho_j>t$, the $L^{\infty}$ estimate for $\varphi_j$ and the measure estimate (\ref{bottomlowdens}) yield
\begin{equation}\nonumber\begin{split}\int_{B\cap \widetilde{\Omega}_j}\varphi_j dm_d&\leq C_{25} m_d(B\cap \widetilde{\Omega}_j)\frac{\mu'(\Omega_j)}{(\varepsilon\rho_j)^d} \\
&\leq Cm_d(B\cap \widetilde{\Omega}_j)\frac{M^s\delta}{\varepsilon^d \rho_j^{d-s}}\leq Cm_d(B\cap \widetilde{\Omega}_j)\frac{M^s\delta}{\varepsilon^d t^{d-s}}.
\end{split}\end{equation}
After summation, we obtain
 $$\tilde{\mu}\Bigl(\bigcup_{\rho_j>t}B\cap \Omega_j\Bigl)\leq C\sum_jm_d(B\cap \widetilde{\Omega}_j)\frac{M^s\delta}{\varepsilon^d t^{d-s}} \leq C\frac{M^s\delta}{\varepsilon^d}t^s.
$$
Bringing our estimates together, we conclude that
$$\tilde{\mu}(B) \leq \mu'(3B) + \frac{CM^s\delta}{\varepsilon^d}t^{s}, \text{ for any ball } B=B(z,t).
$$
Part (i) is now an immediate consequence of the growth condition (\ref{growth}) (recall that $\mu'$ is dominated by $\mu$).  To prove part (ii), note that $\tilde{\mu}(\tfrac{M}{3}B_j) \leq \mu'(MB_j) + \tfrac{CM^s\delta}{\varepsilon^d}(\tfrac{M}{3}\rho_j)^{s}$.  Applying (\ref{bottomlowdens}) yields the required estimate. \end{proof}

The primary advantage of the smoothed measure $\tilde\mu$ is that each $\tilde{\mu}_j$ satisfies
\begin{equation}\label{linfsmooth}
||R(\tilde{\mu}_j)||_{L^{\infty}} \leq \sup_{x\in \mathbf{R}^d}\int_{\widetilde\Omega_j}\frac{\varphi_j(y)}{|y-x|^s}dm_d(y) \leq \frac{C\mu'(\Omega_j)}{(\varepsilon \rho_j)^s} \leq C_{28} \frac{M^s\delta}{\varepsilon^s},
\end{equation}
where (\ref{bottomlowdens}) has been used in the last inequality.  By an analogous argument, we see that if $g$ is a bounded vector field, then
\begin{equation}\label{adjsmooth}
||R^{*}(g\tilde\mu_j)||_{L^{\infty}}\leq C_{28}||g||_{L^{\infty}}\frac{M^s\delta}{\varepsilon^s}.
\end{equation}

\subsection{A comparison lemma} We wish to show that a lower bound for the Riesz transform of $\tilde{\mu}$ transfers to a lower bound for the partial Riesz transform of $\mu'$, with only a small error term.  This is achieved with a comparison lemma.

\begin{lem}  The following estimate holds:
\begin{equation}\label{lowtransfer}
\int_{\Omega}V(R^{(n)}(\mu')) d\mu' \geq \int_{\Omega} V(R(\tilde{\mu})) d\tilde{\mu} - C_{29}\Bigl(\frac{M^{2s}\delta}{\varepsilon^{d+s}} +\frac{1}{M}\Bigl)m'.
\end{equation}
\end{lem}

\begin{proof}  The proof is split into three comparisons.  We first claim that
\begin{equation}\nonumber
\int_{\Omega}\Bigl| V(R^{(n)}(\mu')) - V(R^{(n)}(\tilde\mu))\Bigl|d\mu'\leq C\Bigl(\frac{M^{s}\delta}{\varepsilon^s} + \frac{1}{M}\Bigl)m'.
\end{equation}
To see this, note that since $V$ has Lipschitz constant of at most $4$, it suffices to prove that
\begin{equation}\label{cfclaim1}\int_{\Omega}| R^{(n)}(\mu') - R^{(n)}(\tilde\mu)|d\mu'\leq C\Bigl(\frac{M^{s}\delta}{\varepsilon^s} + \frac{1}{M}\Bigl)m'.
\end{equation}
For a fixed $j$, let $\sigma_j = \chi_{\Omega_j}\mu' - \tilde{\mu}_j$.  Then $|\sigma_j|(\Omega_j)\leq 2\mu'(\Omega_j)$.  It is clear that $R^{(n)}(\sigma_j) = 0$ on $\Omega_j$.  Applying the second estimate in Lemma \ref{osclem} with $\sigma = \sigma_j$ and $\nu = \chi_{\Omega\backslash \Omega_j}\mu'$ yields
$$\int_{\Omega}|R^{(n)}(\sigma_j)|d\mu' \leq C\Bigl(\frac{\mu'(MB_j)}{(\varepsilon\rho_j)^s}+\frac{C_1}{M}\Bigl)\mu'(\Omega_j)\leq C\Bigl(\frac{M^{s}\delta}{\varepsilon^s}+\frac{1}{M}\Bigl)\mu'(\Omega_j),
$$
%(This is analogous to the derivation of (\ref{osccellest}) from Lemma \ref{osclem}).
here we have used (\ref{bottomlowdens}) to estimate $\mu'(MB_j)$. Summing over $j$, we arrive at (\ref{cfclaim1}).

Next we claim that
\begin{equation}\label{cfclaim2}
\Bigl|\int_{\Omega} V(R^{(n)}(\tilde{\mu}))d\mu' - \int_{\Omega}V(R^{(n)}(\tilde\mu))d\tilde\mu\Bigl| \leq C\Bigl(\frac{M^{2s}\delta}{\varepsilon^{d+s}}+\frac{1}{M}\Bigl)m'.
\end{equation}
To this end, apply the first statement of Lemma \ref{osclem} with $\nu = \chi_{\Omega\backslash \Omega_j}\tilde\mu$.  From the growth properties of $\tilde\mu$ from Lemma \ref{smoothgrowthprop}, it follows that
$$\text{osc}_{\Omega_j}V(R^{(n)}(\tilde{\mu}))\leq 4\text{osc}_{\Omega_j}R^{(n)}(\tilde\mu)\leq C\Bigl(\frac{M^{2s}\delta}{\varepsilon^{d+s}}+\frac{1}{M}\Bigl).
$$
On the other hand, $\tilde\mu(\Omega_j) = \mu'(\Omega_j)$, and so we have
$$\Bigl|\int_{\Omega_j} (V(R^{(n)}(\tilde{\mu})) )d(\mu'-\tilde\mu)\Bigl| \leq 2\text{osc}_{\Omega_j}V(R^{(n)}(\tilde{\mu}))\mu'(\Omega_j)%\leq C\Bigl(\frac{M^{2s}\delta}{\varepsilon^{d+s}}+\frac{1}{M}\Bigl)\mu'(\Omega_j)
.$$
Applying the oscillation estimate to the right hand side, we arrive at (\ref{cfclaim2}) after summation in $j$.

Finally, noting that $|R^{(n)}(\tilde\mu) - R(\tilde\mu)| = |R(\tilde\mu_j)|$ on $\Omega_j$, we use the Lipschitz property of $V$, combined with the $L^{\infty}$ estimate (\ref{linfsmooth}), to see that
\begin{equation}\label{cfclaim3}\nonumber
\int_{\Omega}|V(R^{(n)}(\tilde{\mu})) - V(R(\tilde\mu))|d\tilde\mu \leq 4\sum_j \int_{\Omega_j}|R(\tilde\mu_j)|d\tilde\mu \leq 4C_{28}\frac{M^s\delta}{\varepsilon^s}m'.
\end{equation}
Bringing together these three comparisons, we obtain the lemma.
\end{proof}

\subsection{The $\Psi$-function}Consider now the level $n+1$ top cover balls $T_j = B(z_j, 4r_j)$ that are contained in $\Omega$.  Let $\mathcal{J}$ be the set of $j\in \mathcal{J}$ such that $T_j \not\subset T_i$ for any $i\neq j$.   For each $j\in \mathcal{J}$, let $$\widetilde{T}_j =\Bigl\{\bigcup_k \Omega_k : \Omega_k\subset T_j,\, \Omega_k\not\subset T_i \text{ for any }i\in \mathcal{J} \text{ with }i<j\Bigl\}.$$   The sets $\widetilde{T}_j\subset T_j$ are disjoint, and $\bigcup_{j\in \mathcal{J}} \widetilde{T}_{j} \supset \operatorname{supp}(\mu')\cap \Omega$.  (Recall here that each cell $\Omega_k$ is associated to a top cover ball $T_j$ contained in $\Omega$, by property (g) of the Cantor construction.)

%Let $\xi\in C^{\infty}_0(B(0,1))$ be a bump function satisfying
%$$\int_{\mathbf{R}^d}\xi dm_d = m_d(B(0,1)),\, 1\leq \xi\leq 2^d \text{ on } B(0, \tfrac{1}{2}), \text{ and }|\nabla \xi|\leq 4\cdot2^{d}.
%$$
Recall the bump function $\varphi$ from Section \ref{fourier}.
For each $j\in \mathcal{J}$, and $k\geq 2$, let $\varphi_{k,j}(\, \cdot\,) = \varphi\bigl(\frac{\cdot\, - z_j}{2^{k-1} 4r_j}\bigl)$.  Then $\operatorname{supp}(\varphi_{k,j})\subset 2^k T_j$, and $\int_{\mathbf{R}^d}\varphi_{k,j} dm_d = m_d(2^k T_j)$.  We define the $\Psi$ function by%such that $\varphi_{k,j} \in C^{\infty}_0(2^kT_j)$, with $1\leq \varphi_{k,j}\leq 2^d$ on $2^{k-1}T_j$, $\int_{\mathbf{R}^d}\varphi_{k,j}dm_d = m_d(2^kT_j)$, and $|\nabla \varphi_{k,j}|\leq 2^d/(2^kr_j)$.  Now denote
\begin{equation}\label{Psidef}\Psi(x) = \sum_{k\geq 2}2^{k(s-d)} \sum_{j\in \mathcal{J}} \frac{\mu'(\widetilde{T}_j)}{m_d(2^kT_j)}\varphi_{k,j}(x).
\end{equation}
Notice that
\begin{equation}\label{psiintbd}\int_{\mathbf{R}^d}\Psi dm_d = \sum_{k\geq2}2^{k(s-d)}\sum_{j\in \mathcal{J}} \mu'(\widetilde{T}_j) \leq C_{30} m'.
\end{equation}
The following two results contain the properties of $\Psi$ that we will need.  Recall that $m=\mu(\Omega)$ and $m' = \mu'(\Omega)$.
\begin{lem}\label{Rieszlowbdlem}  Let $\nu$ be a nonnegative Borel measure with smooth density such that $\nu(\mathbf{R}^d) \geq m'$.  Suppose in addition that $\nu$ is supported on $\cup_j \Omega_j$,  and $\nu(\widetilde{T}_j)\leq  2\mu'(\widetilde{T}_j)$ for each $j\in \mathcal{J}$.  Then the following estimate holds:
\begin{equation}\label{Vlowbdpsi}
\int_{\mathbf{R}^d}V(R(\nu))\Psi dm_d \geq c_{31}\frac{\Delta^2(m')^3}{m^2}.
\end{equation}
\end{lem}
\begin{proof}  We will first prove that
\begin{equation}\label{Rieszlowbd}
\int_{\mathbf{R}^d}|R(\nu)|\Psi dm_d \geq c \Delta\frac{(m')^2}{m}.
\end{equation}
Recall the definitions of $\varphi$ and $\psi$ from Section \ref{fourier}, and note the pointwise estimate
$\Psi(x) \geq c\sum_{j\in \mathcal{J}}\frac{\mu'(\widetilde{T}_j)}{r_j^d}\bigl|\psi\bigl(\frac{x-z_j}{4r_j}\bigl)\bigl|,$ along with the inequality
\begin{equation}\begin{split}\nonumber\int_{\mathbf{R}^d}|R(\nu)| \Bigl|\psi\Bigl(\frac{\cdot\,-z_j}{4r_j}\Bigl)\Bigl| dm_d& \geq \int_{\mathbf{R}^d}\Bigl(R(\nu), \psi\Bigl(\frac{\cdot\,-z_j}{4r_j}\Bigl)\Bigl)dm_d \\
&= \int_{\mathbf{R}^d}R^{*}\Bigl(\psi\Bigl(\frac{\cdot\,-z_j}{4r_j}\Bigl)m_d\Bigl)d\nu.\end{split}\end{equation}
Employing the equality $R^{*}(\psi(\tfrac{\cdot-z_j}{4r_j})m_d) = (4r_j)^{d-s}\varphi(\tfrac{\cdot-z_j}{4r_j})$ (see  (\ref{fourequal})), we deduce that
\begin{equation}\begin{split}\nonumber &\int_{\mathbf{R}^d}|R(\nu)|\Psi dm_d \geq c\sum_{j\in\mathcal{J}}\frac{\mu'(\widetilde T_j)}{r_j^d}\int_{\mathbf{R}^d}R^{*}\Bigl(\psi\Bigl(\frac{\cdot-z_j}{4r_j}\Bigl)m_d\Bigl) d\nu\\
&= c4^{d-s}\sum_{j\in \mathcal{J}} \frac{ \mu'(\widetilde{T}_j)}{r_j^s}\int_{\mathbf{R}^d}\varphi\Bigl(\frac{\cdot-z_j}{4r_j}\Bigl)d\nu \geq c\sum_{j\in \mathcal{J}} \frac{ \mu'(\widetilde{T}_j)\nu(\widetilde{T}_j)}{r_j^s}\geq c\sum_{j\in \mathcal{J}}\frac{\nu(\widetilde{T}_j)^2}{r_j^s}.
\end{split}\end{equation}
%The property $R(f(\lambda\cdot) m_d)(x) = \lambda^{s-d}R(f m_d)(\lambda x)$, for $\lambda>0$, accounts for the change in the power of $r_j$ in the penultimate inequality.
Since the pairwise disjoint balls $B(z_j, r_j)$ are contained in $\Omega$, and satisfy $\mu(B(z_j, r_j))\geq \frac{\Delta}{2^s} r_j^s$, we obtain $$\sum_j r_j^s \leq \frac{2^s}{\Delta}\sum_{j}\mu(B(z_j, r_j)) \leq \frac{2^s\mu(\Omega)}{\Delta} = \frac{2^sm}{\Delta}.$$
We therefore have
$$\int_{\mathbf{R}^d}|R(\nu)|\Psi dm_d\geq c\sum_j \frac{\nu(\widetilde{T_j})^2}{r_j^s} \geq  c\frac{\Delta(m')^2}{m},$$
where the Cauchy-Schwarz inequality has been used in the last step.  Hence (\ref{Rieszlowbd}) is proved.

Let $x\in \mathbf{R}^d$.  Then since $V(|x|)\geq \min(|x|, |x|^2)$, we see that $V(|x|)\geq \lambda |x| -\lambda^2$ for any $\lambda\in (0,1)$.\footnote{It is trivial that $\lambda|x|-\lambda^2\leq |x|$ for $\lambda \in (0,1)$.  Since $\lambda|x|\leq \tfrac{1}{2}|x|^2+\tfrac{1}{2}\lambda^2$, we also have $\lambda |x| -\lambda^2\leq |x|^2$.}  Hence, with $\lambda\in (0,1)$,
\begin{equation}\begin{split}\nonumber\int_{\mathbf{R}^d} V(R(\nu))\Psi dm_d&\geq \lambda \int_{\mathbf{R}^d}|R(\nu)|\Psi dm_d - \lambda^2\int_{\mathbf{R}^d}\Psi dm_d \\
&\geq c\lambda\frac{\Delta(m')^2}{m} - C_{30}\lambda^2m'.
\end{split}\end{equation}
 Since $\Delta\leq C_1$ and $m'\leq m$, we may pick $\lambda = \frac{c\Delta m'}{2C_{30}C_1m}$, and the result follows.
%Recalling (\ref{psiintbd}), an application of Jensen's inequality in (\ref{Rieszlowbd}) yields
%$$\int_{\mathbf{R}^d} V(R(\nu))\Psi dm_d\geq V\Bigl(\frac{c\Delta m'}{C_{30}m}\Bigl)\int_{\mathbf{R}^d}\Psi dm_d\geq  V\Bigl(\frac{c\Delta m'}{C_{30}m}\Bigl)2^{2(s-d)}m'.
%$$
%Note that $V\bigl(\frac{c\Delta m'}{C_{30}m}\bigl)\geq \min \bigl( \bigl(\frac{c\Delta m'}{C_{30}m}\bigl)^2, \frac{c\Delta m'}{C_{30}m}\bigl)$.
% Since $\Delta\leq C_1$ and $m'\leq m$, we have $V\bigl(\frac{c\Delta m'}{C_{30}m}\bigl) \geq \bigl(\frac{c}{C_{30}C_1}\bigl)^2\bigl(\frac{\Delta m'}{m}\bigl)^2$.
\end{proof}

The next result is an $L^2(\tilde\mu)$ bound for  $R(\Psi m_d)$.  %We again follow \cite{ENV11}, but there are a couple of extra considerations needed to make the argument quantitative.

\begin{prop}\label{rphil2lem}%Suppose that $M^s\delta/\varepsilon^d<1$.
There exists a constant $C_{34}$ such that
\begin{equation}\label{rphil2}
\int_{\Omega}|R(\Psi m_d)|^2 d\tilde\mu \leq C_{34}  m'.
\end{equation}
\end{prop}

We begin with an auxiliary lemma.  For a fixed $A\geq 2$, define the Marcinkiewicz $g$-function by
$$g_A = \sum_{j\in \mathcal{J}}\frac{\mu'(\widetilde{T}_j)}{(Ar_j)^s}\chi_{AT_j}.
$$
\begin{lem}\label{gfunklem}  There exists a constant $C_{32}$, such that for any $A\geq 2$, we have
$$\int_{\Omega}g_A^2d\mu'\leq C_{32} m'.
$$
\end{lem}

Note that the constant here is independent of $A$.

\begin{proof}  From the growth bound (\ref{growth}), $\mu'(3AT_j)\leq C_1 3^s(Ar_j)^s$.  Therefore, for any nonnegative $f\in L^2(\chi_{\Omega}\mu')$, we have
\begin{equation}\begin{split}\nonumber\int_{\Omega}g_A f d\mu' &=\sum_{j\in \mathcal{J}}\mu'(\widetilde{T}_j)\frac{1}{(Ar_j)^s}\int_{AT_j\cap \Omega}fd\mu'\\
&\leq C_1 3^s\sum_{j\in \mathcal{J}}\mu'(\widetilde{T}_j)\frac{1}{\mu'(3AT_j)}\int_{AT_j\cap \Omega}fd\mu'.
\end{split}\end{equation}
Note that $$\mu'(\widetilde{T}_j)\frac{1}{\mu'(3AT_j)}\int_{AT_j\cap \Omega}fd\mu'\leq \mu'(\widetilde{T}_j)\inf_{\widetilde{T}_j}\mathcal{M}(f\chi_{\Omega}),$$ where $\mathcal{M}(f) = \sup\limits_{B:\, x\in B}\displaystyle\frac{1}{\mu'(3B)}\int_{B}|f|d\mu'$.
Since the sets $\widetilde{T}_j$ are disjoint, we observe that
$$ \sum_{j\in \mathcal{J}}\mu'(\widetilde{T}_j)\inf_{\widetilde{T}_j}\mathcal{M}(f\chi_{\Omega})\leq \int_{\Omega}\mathcal{M}(f\chi_{\Omega})d\mu'.$$

By the usual weak type argument involving the Vitali covering lemma, the maximal operator $\mathcal{M}$ is bounded in $L^2(\mu')$, with an operator norm not exceeding $C=C(d)>0$ (see for example \cite{NTV97}).  Applying the Cauchy-Schwarz inequality, we obtain
$$\int_{\Omega} g_{A} f d\mu' \leq C_13^sC \sqrt{m'}||f||_{L^2(\chi_{\Omega}\mu')}.
$$
The lemma now follows by appealing to duality in $L^2(\chi_{\Omega}\mu')$.
\end{proof}

Our next lemma is a comparison argument.  For a fixed $k\geq 2$, define $\Psi_k = \sum_{j\in \mathcal{J}} \frac{\mu'(\widetilde{T}_j)}{m_d(2^kT_j)}\varphi_{k,j}.$

\begin{lem}\label{l2interchangelem}  There exists a constant $C_{33}$ such that
\begin{equation}\label{l2interchange}
\int_{\Omega}|R(\Psi_k m_d)|^2d\tilde\mu \leq 2\int_{\Omega}|R(\Psi_k m_d)|^2d\mu'+ C_{33}m'.
\end{equation}
\end{lem}

\begin{proof} Recall that each Cantor cell $\Omega_{\ell}$ is born out of a bottom cover ball $B_{\ell}$ of radius $\rho_{\ell}$, with $\Omega\supset B_{\ell}\supset \Omega_{\ell}$.  We shall estimate  $\sup_{B_{\ell}}|\nabla R(\Psi_k m_d)| \rho_{\ell}$.

For each bump function $\varphi_{k,j}$, observe the estimate $$|\nabla R(\varphi_{k,j}m_d)(x)|\leq \frac{C (2^kr_j)^d}{(2^k r_j + |x-z_j|)^{s+1}},  \;x\in \mathbf{R}^d.$$  For $x\notin 2^{k+1}T_j$, this estimate follows from differentiating the kernel in the Riesz transform.  If $x\in 2^{k+1}T_j$, we employ the convolution structure to differentiate the bump function $\varphi_{k,j}$, which has a gradient bound of $C/(2^k r_j)$.

We therefore obtain
$$|\nabla R(\Psi_k m_d)(x)| \rho_{\ell} \leq C\sum_{j\in \mathcal{J}}\frac{\mu'(\widetilde{T}_j)\rho_{\ell}}{(2^k r_j + |x-z_j|)^{s+1}}, \text{ for all }x\in \mathbf{R}^d.$$

Now fix $x\in B_{\ell}$, and split the index set $\mathcal{J}$ into two: $\mathcal{J}_1(x)  = \{j\in \mathcal{J}: |x-z_j|\leq 2^{k+1}r_j\}$, and $\mathcal{J}_2(x) =\mathcal{J}\backslash \mathcal{J}_1(x)$.

To bound the sum over $\mathcal{J}_1(x)$, we first claim that if $j\in \mathcal{J}_1(x)$, then $2^{k+1}r_j\geq M \rho_{\ell}/2$. To see this, recall that $B_{\ell}$ is associated to some top cover ball $T_{i}\supset B_{\ell}$, such that $\operatorname{dist}(B_{\ell}, \partial T_{i})\geq r_{i}\geq M\rho_{\ell}$ (see property (g) of the Cantor construction).  Since $T_j$ is not contained in $T_i$, we have $2 \cdot 2^{k+1}r_j\geq r_{i}\geq M\rho_{\ell}$, as required.  Employing this observation, we see that
$$\sum_{j\in \mathcal{J}_1(x)}\frac{\mu'(\widetilde{T}_j)\rho_{\ell}}{(2^k r_j+|x-c_j|)^{s+1}}\leq \frac{C}{M}\sum_{j\in \mathcal{J}_1(x)}\frac{\mu'(\widetilde{T}_j)}{(2^k r_j)^{s}}.
$$
Moreover, if $M\geq 4$, then $2^{k+2}T_j\supset B_{\ell}$ for any $j\in \mathcal{J}_1(x)$.  As a result, with $x\in B_{\ell}$ fixed, the function $\sum_{j\in \mathcal{J}_1(x)}\frac{\mu'(\widetilde{T}_j)}{(2^{k+2}r_j)^s}\chi_{2^{k+2}T_j}$ is constant on  $ B_{\ell}$, and is bounded by $\inf_{B_{\ell}}g_{2^{k+2}}$. We therefore conclude that
$$\sum_{j\in \mathcal{J}_1(x)}\frac{\mu'(\widetilde{T}_j)\rho_{\ell}}{(2^k r_j+|x-c_j|)^{s+1}}\leq C\inf_{ B_{\ell}}g_{2^{k+2}}.$$

Regarding the estimate for the sum over $\mathcal{J}_2(x)$ (with $x\in B_{\ell}$), we claim that for each $j\in \mathcal{J}_2(x)$, we have $|x-y|\geq \rho_{\ell}$ for all $y\in T_j$.  Indeed, if there exists $y\in T_j$ with $|x-y|<\rho_{\ell}$, then $$2\rho_{\ell}> 2(|x-z_j|-|y-z_j|)>2^{k+2}r_j -  8r_j\geq 2^{k+1}r_j\geq 8r_j.$$
Since $2B_{\ell}$ intersects $T_j$, and has radius greater than the diameter of $T_j$, we see that $4B_{\ell}\supset T_j$. Provided $M> 4$, property (g) of the Cantor construction ensures that the ball $T_j$ is a strict subset of the top cover ball associated to $\Omega_{\ell}$ (contradicting $j\in \mathcal{J}$).

Consequently, for any $x\in B_{\ell}$, we obtain
\begin{equation}\begin{split}\nonumber&\sum_{j\in \mathcal{J}_2(x)}\frac{\mu'(\widetilde{T}_j)\rho_{\ell}}{(2^k r_j + |x-z_j|)^{s+1}} \leq\int_{\mathbf{R}^d}\sum_{j\in \mathcal{J}_2(x)}\frac{\rho_{\ell}\chi_{\widetilde{T}_j}(y)}{|x-z_j|^{s+1}}d\mu'(y) \\
&\leq 2^{s+1}\int_{\mathbf{R}^d}\sum_{j\in \mathcal{J}_2(x)}\frac{\rho_{\ell}\chi_{\widetilde{T}_j}(y)}{|x-y|^{s+1}}d\mu'(y)\leq 2^{s+1}\int_{|x-y|> \rho_{\ell}}\frac{\rho_{\ell}}{|x-y|^{s+1}}d\mu'(y). %= (s+1)\rho_{\ell}\int_{\rho_{\ell}}^{\infty}\frac{\mu(B(x,t))}{t^{s+1}}\frac{dt}{t}.
\end{split}\end{equation}
Applying the growth condition on the measure $\mu'$ from (\ref{growth}), we see that this integral is bounded by an absolute constant $C$ depending on $s$ and $d$.  Bringing everything together yields
$$\sup_{B_{\ell}}|\nabla R(\Psi_k m_d)| \rho_{\ell}\leq C\inf_{z\in B_{\ell}}g_{2^{k+2}}(z) + C,$$
and hence we have $\text{osc}_{\Omega_{\ell}}R(\Psi_k m_d) \leq C\inf_{z\in \Omega_{\ell}}g_{2^{k+2}}(z) + C$.

To conclude the proof of the lemma, note that for a continuous function $f$, the following inequality holds
%$$\int_{\Omega_{\ell}}|f|^2 d(\tilde\mu\!-\! \mu') \!\leq\! 2\!\int_{\Omega_{\ell}}|f-\!\frac{1}{\mu'(\Omega_{\ell})}\!\int_{\Omega_{\ell}} fd\mu'|^2 d(\mu'\!+\!\tilde\mu) + 4\int_{\Omega_{\ell}}|f|^2 d\mu'.
%$$
$$\int_{\Omega_{\ell}}|f|^2 d \tilde \mu \leq 2|c_f|^2\tilde\mu(\Omega_{\ell}) + 2\int_{\Omega_{\ell}}|f-c_f|^2 d\tilde\mu,
$$
where $c_f = \frac{1}{\mu'(\Omega_{\ell})}\!\int_{\Omega_{\ell}} fd\mu'$.  Since $\tilde\mu(\Omega_{\ell}) = \mu'(\Omega_{\ell})$, we have $|c_f|^2\tilde\mu(\Omega_{\ell})\leq \int_{\Omega_{\ell}}|f|^2 d\mu'.$  Consequently,  $\int_{\Omega_{\ell}}|f|^2 d \tilde \mu\leq 2\int_{\Omega_{\ell}}|f|^2 d\mu'+ 2\mu'(\Omega_{\ell})\text{osc}_{\Omega_{\ell}}(f)^2$. Applying the oscillation estimate for $R(\Psi_k m_d)$, we see that
$$\int_{\Omega_{\ell}}|R(\Psi_k m_d)|^2d\tilde\mu \leq 2\!\int_{\Omega_{\ell}}|R(\Psi_k m_d)|^2d\mu'+ \!C\!\mu'(\Omega_{\ell})\inf_{\Omega_{\ell}}(g_{2^{k+2}})^2+C\mu'(\Omega_{\ell}).
$$
Since Lemma \ref{gfunklem} yields $ \sum_{\ell}\mu'(\Omega_{\ell})\inf_{\Omega_{\ell}}(g_{2^{k+2}})^2\leq \int_{\Omega} g_{2^{k+2}}^2 d\mu'\leq Cm',$ we arrive at (\ref{l2interchange}) after summation in $\ell$.
\end{proof}

We turn now to the proof of Proposition \ref{rphil2lem}:
\begin{proof}[Proof of Proposition \ref{rphil2lem}]   To obtain the $L^2(\tilde\mu)$ estimate for $\Psi$, it suffices to prove an analogous estimate with $\Psi$ replaced by $\Psi_k$, with a constant independent of $k$.  On account of Lemma \ref{l2interchangelem}, the proposition will follow once we assert that
\begin{equation}\label{onekest}
\int_{\Omega} |R(\Psi_k m_d )|^2 d\mu' \leq Cm',
\end{equation}
with the constant $C$ independent of $k$.  To prove (\ref{onekest}), we shall compare $R(\Psi_k m_d)$ with the vector field $\Theta = \sum_{j\in \mathcal{J}} \chi_{\mathbf{R}^d\backslash 2^{k+1}T_j}R(\chi_{\widetilde{T}_j}\mu').$
To this end, we first claim that
\begin{equation}\label{thetal2}\int_{\Omega} |\Theta|^2 d\mu' \leq Cm'.
\end{equation}
Indeed, for any vector field $g$ in $L^2(\chi_{\Omega}\mu')$, note that
$\bigl|\int_{\Omega} \Theta \cdot g d\mu'\bigl|$ is equal to $$\Bigl|\int_{\Omega}\sum_{j\in \mathcal{J}} \chi_{\widetilde{T}_j}(y)\!\Bigl[\int_{\mathbf{R}^d\backslash 2\cdot 2^{k}T_j}\!\frac{y-x}{|y-x|^{1+s}}\cdot g(x)\chi_{\Omega}(x)d\mu'(x)\Bigl]d\mu'(y)\Bigl|.
$$
Since the sets $\widetilde{T}_j$ are disjoint we see that  $$\Bigl|\int_{\Omega} \Theta \cdot g d\mu'\Bigl|\leq\int_{\Omega}(R^*)^{\#}(g\chi_{\Omega} \mu')d\mu'.$$  As $\mu'$ is dominated by $\mu$, the mapping $g\mapsto (R^*)^{\#}(g\mu')$ is bounded on $L^2(\mu')$, with operator norm at most $\sqrt{dC_3}$, see (\ref{T1adj}).  The Cauchy-Schwarz inequality now yields $\int_{\Omega}(R^*)^{\#}(g\chi_{\Omega} \mu')d\mu' \leq \sqrt{m'}\sqrt{d C_3}||g||_{L^2(\chi_{\Omega}\mu')}$.  Appealing to duality in vector valued $L^2(\chi_{\Omega}\mu')$, we obtain (\ref{thetal2}).

To estimate $|R(\Psi_k m_d) - \Theta|$ pointwise, examine the difference
\begin{equation}\label{1jdiff}\Bigl|R\Bigl(\frac{\mu'(\widetilde{T}_j)}{m_d(2^kT_j)}\phi_{k,j} m_d \Bigl)(x)-\chi_{\mathbf{R}^d\backslash 2^{k+1}T_j}R(\chi_{\widetilde{T}_j}\mu')(x)\Bigl|.
\end{equation}

If $x\in 2^{k+1}T_j$, then the second term does not contribute.  Crudely estimating the first term, we can bound the difference in this case by
$C\frac{\mu'(\widetilde{T}_j)}{(2^{k}r_j)^s}.$

In the case when $x\not\in 2^{k+1}T_j$, note that $\nu = \frac{\mu'(\widetilde{T}_j)}{m_d(2^kT_j)}\phi_{k,j}m_d - \chi_{\widetilde{T}_j}\mu'$ has mean zero.  Since the distance between $x$ and $\operatorname{supp}(\nu)$ is comparable to $|x-z_j|$, with $z_j$ the centre of $T_j$, we derive the estimate $|R(\nu)(x)| \leq C\frac{\mu'(\widetilde{T}_j)2^kr_j}{|x-z_j|^{s+1}}.$

Combining these two estimates, we see that the difference in (\ref{1jdiff}) is bounded by
$C\sum_{\ell\geq k}2^{k-\ell}\frac{\mu'(\widetilde{T}_j)}{(2^{\ell}r_j)^s}\chi_{2^{\ell}T_j}.$
After the summation in $j$, we have
$$|R(\Psi_k m_d) - \Theta|\leq C\sum_{\ell \geq k}2^{k-\ell}g_{2^{\ell}}.
$$
Since $||g_{2^{\ell}}||_{L^2(\chi_{\Omega}\mu')}\leq C\sqrt{m'}$,  with a constant $C$ independent of $\ell$, we have $||R(\Psi_k m_d)-\Theta||_{L^2(\chi_{\Omega}\mu')}\leq C\sqrt{m'}$, and (\ref{onekest}) follows.%This is obtained via duality: let $f\in L^2(d\tilde\mu)$, then
%\begin{equation}\label{maximaldualstep1}\int_{\mathbf{R}^d} g_{\ell} f d\tilde\mu = \int_{\mathbf{R}^d}\Bigl[ \sum_j \chi_{\widetilde{T}_j} \frac{1}{(2^{\ell}r_j)^s}\int_{2^{\ell}T_j}fd\tilde\mu\Bigl] d\tilde\mu.\end{equation}
%Observe that under the condition $M^s\delta/\varepsilon^d<1$, the growth estimate (\ref{anyballsm}) yields $\tilde\mu(3 \cdot 2^{\ell}T_j) \leq C(2^{\ell}r_j)^s$.  From this observation, along with the disjointness of $\widetilde{T}_j$, the right hand side of (\ref{maximaldualstep1}) is dominated by
%$C\int_{\mathbf{R}^d}\mathcal{M}(fd\tilde\mu) d\tilde\mu$.  Here $\mathcal{M}(fd\tilde\mu)(x) = \sup_{B: x\in B}\frac{1}{\tilde\mu(3B)}\int_B fd\tilde\mu.$
%The maximal operator $\mathcal{M}$ is bounded in $L^2(d\tilde\mu)$ with norm dependent only on dimension, as result of the Vitali covering lemma, see e.g. \cite{NTV97}.  A final application of the Cauchy-Schwarz inequality yields the required $L^2(d\tilde\mu)$ bound on $g_{\ell}$.  Hence, (\ref{onekest}) is proved, and with it Lemma \ref{rphil2lem}.
\end{proof}

%Note that Lemma \ref{Rieszlowbdlem} is a lower bound for the Riesz transform in the norm $L^1(\Psi dm_d)$.  That this will be useful indicates some flexibility in the norm with which the estimate (\ref{partialsequalmax}) is obtained.  We take advantage of this by i

\subsection{An extremal problem}  With a view to obtaining a contradiction, assume that
\begin{equation}\label{supposesmall}\int_{\mathbf{R}^d}V(R(\tilde{\mu}))d\tilde\mu \leq \lambda \mu'(\Omega) = \lambda m'.
\end{equation}
We will obtain a contradiction if $\lambda>0$ is chosen small enough.  To this end, we will replace $\tilde{\mu}$ by an energy minimizing measure.  This  idea is reminiscent of the idea of equilibrium measure in potential theory.

For a vector $\textbf{a} = \{a_j\}_j$ with $a_j\geq 0$ for all $j$, define the measure $\mu^{\bf{a}}$ by $\mu^{\bf{a}} = \sum_j a_j \tilde{\mu}_j$, with $\tilde \mu_j$ as in (\ref{phijprops}).  By construction, $\operatorname{supp}(\mu^{\bf{a}})\subset \bigcup_j \widetilde{\Omega}_j$ for any choice of $\mathbf{a}$.   Note that the vector $\mathbf{a}$ is of finite dimension, since there are a finite number of Cantor cells $\Omega_j$.  Consider now the functional $F(\mathbf{a})$, given by
$$F(\mathbf{a}) = \lambda m'\cdot \sup_j a_j + \int_{\mathbf{R}^d}V(R({\mu}^{\bf{a}}))d\mu^{\bf{a}}.
$$
The reasoning behind the definition of $F$ is the following.  The second term is precisely the energy that we wish to minimize.  The inclusion of the first term is to prevent the extremal measure from being much larger than $\tilde\mu$ on any cell $\Omega_j$.

Let $\bf{a}^{\star}$ be the minimizer for $F$ under the constraint $\mu^{\bf{a}}(\mathbf{R}^d) = m'$. That a minimizer should exist is easy to see; firstly, since $\tilde\mu(\mathbf{R}^d) = m'$, the vector $\bf{a}=\bf{1}$ is admissible; and secondly, the functional $F(\bf{a})$ is continuous in $\textbf{a}$ and grows to infinity as any component of  $\textbf{a}$ tends to infinity.  For notational ease we let $\mu^{\star} = \mu^{\bf{a}^{\star}}$.   Note that $F(\textbf{a}^{\star})\leq F(\textbf{1})\leq 2\lambda m'$, and hence $\mu^{\star}\leq 2\tilde\mu$.

In order to obtain information from the minimizer, one can examine the first variation of the functional $F$ under a distortion of $\mu^{\star}$.  This examination yields the following lemma.
\begin{lem}\label{firstvar}  For each $j$ with $a_j^{\star}>0$, there exists a point $w\in \widetilde{\Omega}_j$ such that
\begin{equation}\label{firstvarst}V(R(\mu^{\star}))(w) + R^*[\nabla V(R(\mu^{\star}))\mu^{\star}](w) \leq 6\lambda.
\end{equation}
\end{lem}

\begin{proof}Fix $j$ with $a_j^{\star}>0$.  We shall estimate the functional $F$ evaluated at the vector
$$ \textbf{b} = \frac{\tilde\mu(\mathbf{R}^d)}{\tilde{\mu}(\mathbf{R}^d) - t\tilde{\mu}_j(\mathbf{R}^d)}(\textbf{a}^{\star} - t\textbf{e}_j),$$
where $\textbf{e}_j$ is the vector whose $j$th component is $1$, and all other components are zero.
Note here that $\mathbf{b}$ is an admissible vector provided $0<t<a_j^{\star}$.  First observe that
$$F(\textbf{a}^{\star} - t\textbf{e}_j) \leq F(\textbf{a}^{\star}) -tI+O(t^2), \text{ as }t\rightarrow 0^{+},
$$
with $I$ denoting the quantity
$$I = \int_{\mathbf{R}^d} V(R(\mu^{\star})) d\tilde\mu_j+\int_{\mathbf{R}^d}(\nabla V(R(\mu^{\star})), R(\tilde\mu_j)) d\mu^{\star}.
$$
Since $V(a|x|)\leq a^2V(|x|)$ for all $a>1$, the normalization in the definition of $\textbf{b}$ can increase the value of the functional $F$ by a factor of at most $\tilde\mu(\mathbf{R}^d)^3/(\tilde{\mu}(\mathbf{R}^d) - t\tilde{\mu}_j(\mathbf{R}^d))^3$.  We therefore obtain
$$F(\textbf{a}^{\star}) \leq F(\textbf{b})\leq \frac{\tilde\mu(\mathbf{R}^d)^3}{(\tilde{\mu}(\mathbf{R}^d) - t\tilde{\mu}_j(\mathbf{R}^d))^3} (F(\textbf{a}^{\star}) -tI) +O(t^2).
$$
The first inequality here is just the minimization property of $\textbf{a}^{\star}$.
Comparing first order terms, and taking the limit as $t\rightarrow 0^+$, we arrive at
%$$I\leq\frac{F(\textbf{a}^{\star})}{t}\Bigl(1 - t\frac{\tilde\mu_j(\mathbf{R}^d)}{\tilde\mu(\mathbf{R}^d)}\Bigl)^3 + O(t),
%$$
\begin{equation}\label{Ismall}I \leq 3F(\textbf{a}^{\star})\frac{\tilde\mu_j(\mathbf{R}^d)}{\tilde\mu(\mathbf{R}^d)}\leq 6\lambda \tilde{\mu}_j(\mathbf{R}^d).\end{equation}
To deduce (\ref{firstvar}), we re-write $I$ as an integral over $\tilde\mu_j$:
$$I = \int_{\mathbf{R}^d} \Bigl(V(R(\mu^{\star})) + R^*[\nabla V(R(\mu^{\star})) \mu^{\star}]\Bigl) d\tilde\mu_j.
$$
Due to (\ref{Ismall}), we conclude that $V(R(\mu^{\star})) + R^*[\nabla V(R(\mu^{\star}))\mu^{\star}] \leq 6\lambda$ on average, with respect to $\tilde\mu_j$.  Since $\operatorname{supp}(\tilde \mu_j)\subset \widetilde{\Omega}_j$, there must exist $w\in \widetilde{\Omega}_j$ satisfying (\ref{firstvarst}). \end{proof}

Next we shall strengthen (\ref{firstvarst}) to a uniform estimate on $\widetilde{\Omega}_j$. To do this, we will obtain some oscillation estimates.  Since $\mu^{\star}\leq 2\tilde\mu$, Lemma \ref{smoothgrowthprop} provides us with growth estimates for the measure $\mu^{\star}$.   Using these growth properties, an application of the first part of Lemma \ref{osclem} with $\nu = \chi_{\mathbf{R}^d\backslash \Omega_j}\mu^{\star}$ yields
$$\text{osc}_{\Omega_j}R(\chi_{\mathbf{R}^d\backslash \Omega_j}\mu^{\star})\leq C\Bigl(\frac{M^{2s}\delta}{\varepsilon^{d+s}}+\frac{1}{M}\Bigl).$$
Since $|\nabla V(R(\mu^{\star}))|\leq 4$, the adjoint oscillation estimate (\ref{adjoscest}), applied with $g= \nabla V(R(\mu^{\star}))$, yields
$$\text{osc}_{\Omega_j}R^*[\nabla V(R(\mu^{\star}))\chi_{\mathbf{R}^d\backslash \Omega_j}\mu^{\star}]\leq C\Bigl(\frac{M^{2s}\delta}{\varepsilon^{d+s}}+\frac{1}{M}\Bigl).%\Bigl).
$$
%$$\text{osc}_{\Omega_j}\bigl[R(\chi_{\mathbf{R}^d\backslash \Omega_j}\mu^{\star})+  R^*[(\nabla V(R(\mu^{\star})))\chi_{\mathbf{R}^d\backslash \Omega_j}\mu^{\star}] \bigl]\leq C\Bigl(\frac{M^{2s}\delta}{\varepsilon^{d+s}}+\frac{1}{M}%\Bigl).
%$$
%Since $|\nabla V|\leq 4$ and $\mu^{\star} \leq 2\tilde\mu$, these inequalities follow in an analogous fashion to the derivation of (\ref{osccellest}) from Lemma \ref{osclem}, using the properties of $\tilde\mu$ from Lemma \ref{smoothgrowthprop}.

On the other hand, recalling the $L^{\infty}(m_d)$ estimate for $R(\tilde\mu_j)$ from (\ref{linfsmooth}),  we deduce that
$$\text{osc}_{\Omega_j}R(\chi_{ \Omega_j}\mu^{\star})\leq 2||R(\chi_{ \Omega_j}\mu^{\star})||_{L^{\infty}} \leq C\Bigl(\frac{M^{s}\delta}{\varepsilon^{s}}\Bigl).
$$
Similarly, applying the adjoint $L^{\infty}$ estimate (\ref{adjsmooth}) with $g=\nabla V(R(\mu^{\star}))$, we see that
$$\text{osc}_{\Omega_j}R^*[\nabla V(R(\mu^{\star})) \chi_{\Omega_j}\mu^{\star}] \leq C\Bigl(\frac{M^{s}\delta}{\varepsilon^{s}}\Bigl).
$$

Now note that
$$\text{osc}_{\Omega_j}V(R(\mu^{\star}))\leq 4\text{osc}_{\Omega_j}R(\mu^{\star}) \leq 4\text{osc}_{\Omega_j}R(\chi_{\mathbf{R}^d\backslash \Omega_j}\mu^{\star})+ 4\text{osc}_{\Omega_j}R(\chi_{ \Omega_j}\mu^{\star}),$$
and hence when combined with Lemma \ref{firstvar}, these four oscillation estimates yield
\begin{equation}\label{smallonsupport}
V(R(\mu^{\star}))(w) + R^*[\nabla V(R(\mu^{\star}))\mu^{\star}](w) \!\leq\! 6\lambda  + C_{35}\Bigl(\frac{M^{2s}\delta}{\varepsilon^{d+s}}+\frac{1}{M}\Bigl),
\end{equation}
for all $w\in \widetilde{\Omega}_j$.  Since $j$ was arbitrary, (\ref{smallonsupport}) holds for all $w\in \operatorname{supp}(\mu^{\star})$.

To extend the estimate (\ref{smallonsupport}) to the whole space $\mathbf{R}^d$, we appeal to the following maximum principle, which can be found in Section 17 of \cite{ENV11}.
\begin{prop}\cite{ENV11}. \label{maxp} Let $s\in (d-1,d)$.  Suppose $\omega$ is a measure with a smooth compactly supported density with respect to $m_2$, and suppose $g$ is a smooth vector field.  Then
\begin{equation}\begin{split}\nonumber%\label{maxpcst}
\max_{\mathbf{R}^d}\bigl[V(R(\omega))+R^{*}(&g \omega)\bigl] \\&= \max_{\operatorname{supp}(\omega)}\bigl[V(R(\omega))+R^{*}(g \omega)\bigl] ,
\end{split}\end{equation}
provided the left hand side is positive.
\end{prop}

\begin{proof} We shall give a moderately detailed proof.  For a more careful exposition of this argument see Section 17 of \cite{ENV11}.  The key observation is that if $\nu$ is a vector valued measure with $C^{\infty}_0(\mathbf{R}^d)$ density with respect to $m_d$, then
\begin{equation}\label{adjoinmax}
\max_{\mathbf{R}^d} R^{*}(\nu) = \max_{\text{supp}(\nu)}R^{*}(\nu),
\end{equation}
provided the left hand side is positive.

To see this, we set $u=R^{*}(\nu)$. Then we can write $u$ as the $(s-1)$-dimensional Riesz potential $u(x) = \frac{-1}{s-1}\int_{\mathbf{R}^d}\frac{1}{|x-y|^{s-1}}p(y)dm_d(y),$ where $p$ is the divergence of the density of $\nu$. It is immediate that $\text{supp}(p)\subset \text{supp}(\nu)$.  Since $u$ decays suitably at infinity, the density $p$ can be recovered from $u$ by the integral operator
\begin{equation}\label{fraclap}p(x) = \varkappa \,P.V.\int_{\mathbf{R}^d} \frac{u(y)-u(x)}{|y-x|^{2d+1-s}} d m_d(y),
\end{equation}
where $\varkappa$ is a non-zero constant depending on $s$ and $d$, see for example \cite{Lan70, ENV11}.  For $s<d-1$, the analogue of this inversion formula involves the Laplacian of $u$, and appears difficult to work with.  This is the main reason for our restriction to $s\in (d-1,d)$.

The decay of $u$ at infinity ensures that should $u$ have a positive maximum,  the maximum is attained and $u$ is not constant.  Now suppose that $u$ attains a positive maximum at $x$. Then we observe that the integral appearing in (\ref{fraclap}) is non-zero.  Hence $x\in \text{supp}(p)\subset \text{supp}(\nu)$, and (\ref{adjoinmax}) is proved.

To prove the proposition, write $V(x) = \max_{t \geq 0, |e|=1}[t(e,x) - v^{*}(t)]$, where $v^*(t)$ is the Legendre transform of $v(t)$.  Fix $x\in \mathbf{R}^d$ with $V(R(\omega))(x) +R^{*}(g\omega)(x)>0$.  For some $t\geq 0$ and unit vector $e$, we have $$V(R(\omega))(x)+R^{*}(g\omega)(x) = R^*(g\omega - te \omega)(x)-v^{*}(t).$$
Since $v^*(t)\geq 0$, we see that $R^*([g - te] \omega)(x)>0$.  Hence (\ref{adjoinmax}) guarantees that $R^{*}([g - te] \omega)$ attains its maximum on the support of $\omega$.  We conclude that
\begin{equation}\begin{split}\nonumber V(R(\omega))(x) +R^{*}(g\omega)(x)&\leq \max_{\text{supp}(\omega)} R^{*}([g - te] \omega) -v^*(t) \\
&\leq \max_{\text{supp}(\omega)}V(R(\omega)) +R^{*}(g\omega),
\end{split}\end{equation}
as required.
\end{proof}

 Letting $\omega = \mu^{\star}$ and $g=\nabla V(R(\mu^{\star}))$ in Proposition \ref{maxp}, we conclude that (\ref{smallonsupport}) holds for all $w\in \mathbf{R}^d$, provided $V(R(\mu^{\star})) + R^*[\nabla V(R(\mu^{\star}))\mu^{\star}]$ has a positive maximum.  However, if this is not the case then (\ref{smallonsupport}) holds trivially for all $w\in \mathbf{R}^d$.

\subsection{The conclusion of the proof of Proposition \ref{prop3}}We are now in a position to bring our estimates together.

\begin{proof}[Proof of Proposition \ref{prop3}] We begin by integrating the bound (\ref{smallonsupport}), valid for all $w\in \mathbf{R}^d$, against the function $\Psi$ defined in (\ref{Psidef}).  The result is the estimate
\begin{equation}\begin{split}\label{intvspsi}\int_{\mathbf{R}^d}V(R(\mu^{\star}))\Psi dm_d & + \int_{\mathbf{R}^d}R^*[\nabla V(R(\mu^{\star}))\mu^{\star}]\Psi dm_d \\
&\leq C_{30}\mu'(\mathbf{R}^d)\Bigl[ 6\lambda  + C_{35}\Bigl(\frac{M^{2s}\delta}{\varepsilon^{d+s}}+\frac{1}{M}\Bigl)\Bigl].
\end{split}\end{equation}
The first integral on the left hand side of (\ref{intvspsi}) is estimated from below using Lemma \ref{Rieszlowbdlem}, since $\mu^{\star}$ satisfies the assumptions on the measure $\nu$.  To estimate the second integral on the left hand side of (\ref{intvspsi}), we write
\begin{equation}\label{beforecs}\int_{\mathbf{R}^d}\!R^*[\nabla V(R(\mu^{\star}))\mu^{\star}] \Psi dm_d = \int_{\mathbf{R}^d} \!\bigl(R(\Psi m_d), \nabla V(R(\mu^{\star}))\bigl)d\mu^{\star}.
\end{equation}
Applying the Cauchy-Schwarz inequality, we bound this expression in absolute value by
\begin{equation}\label{aftercs}
\Bigl[\int_{\mathbf{R}^d} \bigl|R(\Psi m_d)\bigl|^2 d\mu^{\star}\Bigl]^{1/2}\cdot\Bigl[\int_{\mathbf{R}^d}|\nabla V(R(\mu^{\star}))|^2d\mu^{\star}\Bigl]^{1/2},
\end{equation}
which we claim is no greater than $4\sqrt{\lambda C_{34}}m'.$  To see this, note that by Proposition \ref{rphil2lem},
$$\int_{\mathbf{R}^d} \bigl|R(\Psi m_d)\bigl|^2 d\mu^{\star}\leq 2\int_{\mathbf{R}^d} \bigl|R(\Psi m_d)\bigl|^2 d\tilde\mu\leq 2C_{34}m'.$$
On the other hand, since $|\nabla V|^2\leq 4V$, it follows that
$$\int_{\mathbf{R}^d}|\nabla V(R(\mu^{\star}))|^2d\mu^{\star}\leq 4\int_{\mathbf{R}^d}V(R(\mu^{\star}))d\mu^{\star}\leq 4F(\mathbf{a}^{\star})\leq 8\lambda m',
$$
and the claimed estimate follows.

Bringing everything together, we get the following inequality:
\begin{equation}\label{lamrel} c_{31}\frac{\Delta^2(m')^2}{m^2}-4\sqrt{C_{34}\lambda}\leq C_{30}\Bigl(6\lambda+ C_{35}\Bigl[\frac{M^{2s}\delta}{\varepsilon^{d+s}}+\frac{1}{M}\Bigl]\Bigl).
\end{equation}
Let $\lambda = c_{36}\Delta^4 (m'/m)^4$, for a suitable small constant $c_{36}$.  Then if $\varepsilon$, $M$, and $\delta$ satisfy
%\begin{equation}\label{parameterscond}C_{29}\Bigl[\frac{M^s\delta}{\varepsilon^{d+s}}+\frac{1}{M}\Bigl]\leq \frac{\lambda}{2},
%\end{equation}
%where $C_{29}$ is the constant from the comparison lemma (Lemma \ref{lowtransfer}).
\begin{equation}\label{anothercond}\Bigl[\frac{M^{2s}\delta}{\varepsilon^{d+s}}+\frac{1}{M}\Bigl]\leq \lambda,\end{equation}
we arrive at a contradiction with (\ref{lamrel}) provided $c_{36}$ was chosen small enough (recall here that $\Delta\leq C_1$).  As a result, either  (\ref{supposesmall}) or (\ref{anothercond}) is false.  Either way, we obtain
$$\int_{\Omega}V(R(\tilde{\mu}))d\tilde\mu \geq c_{36}\Bigl(\frac{\Delta m'}{m}\Bigl)^4m'- \Bigl[\frac{M^{2s}\delta}{\varepsilon^{d+s}}+\frac{1}{M}\Bigl]m'.
$$
Appealing to the comparison estimate (\ref{lowtransfer}), we conclude that
\begin{equation}\label{zeroorderlowbd}\int_{\Omega}V(R^{(n)}(\mu'))d\mu' \geq c_{36}\Bigl(\frac{\Delta m'}{m}\Bigl)^4m'- (C_{29}+1)\Bigl[\frac{M^{2s}\delta}{\varepsilon^{d+s}}+\frac{1}{M}\Bigl]m'.
\end{equation}
Now note that an application of H\"{o}lder's inequality yields
$$\sum_j \frac{\mu'(\Omega^{(n)}_j)^4}{\mu(\Omega^{(n)}_j)^4} \mu'(\Omega^{(n)}_j) \geq \frac{\mu'(\mathbf{R}^d)^5}{\mu(\mathbf{R}^d)^4} \geq \Bigl(\frac{\gamma}{2}\Bigl)^4\mu'(\mathbf{R}^d).
$$
Hence, summing (\ref{zeroorderlowbd}) over the level $n$ Cantor cells, and recalling that $V(x)\leq |x|^2$, we deduce that
$$\int_{\mathbf{R}^d}|R^{(n)}(\mu')|^2d\mu'\geq  \frac{c_{36}\Delta^4\gamma^4}{2^4} \mu'(\mathbf{R}^d)- (C_{29}+1)\Bigl[\frac{M^{2s}\delta}{\varepsilon^{d+s}}+\frac{1}{M}\Bigl]\mu'(\mathbf{R}^d).
$$
It remains to choose $K_2 =  \frac{2^5(C_{29}+1)}{c_{36}}$.  This completes the proof.
\end{proof}

\section{The exponential potential and capacity}\label{exppotsec}

To conclude the paper, we make a brief digression into capacity.  We shall set up a general form of non-linear capacity using Wolff's potentials.  Suppose that $\Phi:[0,\infty)\rightarrow [0,\infty)$ satisfies the following conditions:
\begin{enumerate}
\item $\Phi(0)=0$,
\item $\Phi$ is continuous, and strictly increasing,
\item there exist positive constants $\sigma$ and $\varkappa$, such that $\Phi(t)/t^{\sigma}$ is nondecreasing on $(0,\varkappa]$.
\end{enumerate}
We define the $s$-dimensional Wolff potential associated to the gauge $\Phi$ by
\begin{equation}\label{wolffphi}\mathcal{W}_{\Phi,s}(\mu)(x) = \int_0^{\infty}\Phi\Bigl(\frac{\mu(B(x,r))}{r^s}\Bigl) \frac{dr}{r}.
\end{equation}
The $s$-dimensional nonlinear capacity associated to $\Phi$ is defined for a compact set $E\subset \mathbf{R}^d$ by
\begin{equation}\begin{split}\label{capphi}
\text{cap}_{\Phi,s}(E) = & \sup \bigl\{\mu(E): \operatorname{supp}(\mu)\subset E, \text{ and }\\
& \mathcal{W}_{\Phi,s}(\mu)(x)\leq 1 \text{ for all }x\in \mathbf{R}^d\bigl\}.
\end{split}\end{equation}

First note that the capacity is indeed $s$-dimensional: for $\lambda>0$ and a compact set $E\subset \mathbf{R}^d$, define $\lambda E+z = \{\lambda e +z: e\in E\}$.  Then we have $\text{cap}_{\Phi,s}(\lambda E+z) = \lambda^s\text{cap}_{\Phi,s}(E)$, for any $z\in \mathbf{R}^d$.

To see this, observe that if $\mu$ is an admissible measure for the capacity of $\lambda E + z$, then the measure $\nu(A) = \lambda^{-s}\mu(\lambda A+z)$ is admissible for the capacity of $E$, and vice versa.

We now examine how the capacity changes with the size condition on the Wolff potential in (\ref{capphi}).  For $A>0$, and a compact set $E\subset \mathbf{R}^d$, we define
\begin{equation}\begin{split}\nonumber \operatorname{cap}_{\Phi, s}^{(A)}(E) =  \sup \bigl\{\mu(E): &\,\operatorname{ supp}(\mu)\subset E, \text{ and } \\
&\mathcal{W}_{\Phi,s}(\mu)(x)\leq A \text{ for all }x\in\mathbf{R}^d\bigl\}.
\end{split}\end{equation}
\begin{lem}\label{cutoffsame} Suppose $0<A'<A$. There exists a constant $C=C(A',A, \sigma, \varkappa,s)>0$, such that for all compact sets $E\subset \mathbf{R}^d$,
$$\operatorname{cap}^{(A')}_{\Phi, s}(E) \leq \operatorname{cap}_{\Phi, s}^{(A)}(E) \leq C \operatorname{cap}^{(A')}_{\Phi, s}(E).
$$ \end{lem}

\begin{proof}  The first inequality is trivial.  To see the second inequality, suppose that $\text{cap}_{\Phi,s}^{(A)}(E)>0$.  Let $\varepsilon>0$, and choose $\mu$ to be an admissible measure for $\text{cap}_{\Phi, s}^{(A)}(E)$, with $\mu(E)\geq (1-\varepsilon)\text{cap}_{\Phi, s}^{(A)}(E)$. Fix $x\in \mathbf{R}^d$, and note that
$$\log M  \cdot \Phi\Bigl(\frac{\mu(B(x,r))}{M^sr^s}\Bigl) \leq \int_r^{Mr}\Phi\Bigl(\frac{\mu(B(x,t))}{t^s}\Bigl) \frac{dt}{t}\leq \mathcal{W}_{\Phi,s}(\mu)(x)\leq A.
$$
Setting $M=e^{A/\varkappa}$, we conclude that $\Phi\bigl(\frac{\mu(B(x,r))}{M^sr^s}\bigl)\leq \varkappa$ for all $r>0$. % the condition $\mathcal{W}_{\Phi,s}(\mu)(x)\leq A$, implies that $\mu(B(x,r))\leq 2^s t(\tfrac{A}{\log 2}) r^s$ for all $r>0$.  Indeed, otherwise there exists $r>0$ such that $\Phi(\tfrac{1}{2^s}\tfrac{\mu(B(x,r))}{r^s})\!>\!\tfrac{A}{\log 2}$, and thus $\Phi(\tfrac{\mu(B(x,t))}{t^s})>\tfrac{A}{\log 2}$ for all $t\in (r,2r)$.  Integrating this last estimate from $r$ to $2r$ with respect to $\tfrac{dr}{r}$, we obtain $\mathcal{W}_{\Phi,s}(\mu)(x)>A$ (which is a contradiction).
%It follows that either $\tfrac{\mu(B(x,r))}{r^s}\leq \varkappa$, or $\varkappa < 2^s t(\tfrac{A}{\log 2})$.
Using conditions (2) and (3) in the definition of $\Phi$, we see that
\begin{equation}\begin{split}\nonumber A'\geq \frac{A'}{A} \int_{0}^{\infty}\Phi\Bigl(\frac{\mu(B(x,r))}{M^s r^s}\Bigl)\frac{dr}{r}\geq \int_0^{\infty}\Phi\Bigl(\Bigl(\frac{A'}{A}\Bigl)^{\frac{1}{\sigma}}\frac{\mu(B(x,r))}{M^s r^s}\Bigl)\frac{dr}{r}.
\end{split}\end{equation}

%\begin{equation}\begin{split}\nonumber \frac{1}{A} \Phi\Bigl(\frac{\mu(B(x,r))}{r^s}\Bigl)&\geq\frac{1}{A}\Phi\Bigl(\min\Bigl(1,\frac{\varkappa}{2^s t(\tfrac{A}{\log 2})}\Bigl)\cdot\frac{\mu(B(x,r))}{r^s}\Bigl)\\
%&\geq \Phi\Bigl(\frac{1}{A^{1/\sigma}}\min\Bigl(1,\frac{\varkappa}{2^s t(\tfrac{A}{\log 2})}\Bigl)\frac{\mu(B(x,r))}{r^s}\Bigl),
%\end{split}\end{equation}
%Here it is used that $\frac{\varkappa}{2^s t(\tfrac{A}{\log 2})}\frac{\mu(B(x,r))}{r^s}\leq \varkappa$.

Hence, $\bigl(\tfrac{A'}{A}\bigl)^{\frac{1}{\sigma}}M^{-s}\mu$ is an admissible measure for $\text{cap}^{(A')}_{\Phi,s}(E)$, and therefore $(1-\varepsilon)\text{cap}^{(A)}_{\Phi,s}(E) \leq \bigl(\tfrac{A}{A'}\bigl)^{\frac{1}{\sigma}} e^{sA/\varkappa} \text{cap}^{(A')}(E)$.
\end{proof}

 Let $\mathcal{H}^s(E)$ be the $s$-dimensional Hausdorff measure of a set $E$.  The next result states that the capacity is a finer set function than the Hausdorff measure, regardless of $\Phi$.

 %Let $\mathcal{D} = \bigl\{ 2^nk + 2^n[0,1)^d: k\in \mathbf{Z}^d, \, n\in \mathbf{Z}\bigl\}$ be the lattice of half open dyadic cubes.

\begin{lem}\label{Hauszero} Suppose that $E\subset \mathbf{R}^d$ is a compact set with $\mathcal{H}^s(E)<\infty$.  Then $\text{cap}_{\Phi,s}(E)=0$.
\end{lem}

\begin{proof} Suppose that $\mathcal{H}^s(E)<\infty$, but $\text{cap}_{\Phi,s}(E)>0$.  Then there exists a measure $\mu$ with $\mu(E)>0$ and $\mathcal{W}_{\Phi,s}(\mu)(x)\leq 1$ for all $x\in \mathbf{R}^d$.

 Let $\varepsilon>0$ be small enough so that $\Phi^{-1}(\varepsilon)$ exists ($\Phi^{-1}$ here denotes the inverse function to $\Phi$).  Note that $\Phi^{-1}(t)\rightarrow 0$ as $t\rightarrow 0$.   Let $\rho>0$, and consider the set $E_{\rho}=\{x\in E: \tfrac{\mu(B(x,r))}{r^s} \leq 2^s \Phi^{-1}(\varepsilon), \text{ for all }r\leq\rho\}$.

 Consider a cover of $E$ by balls $B_j$ with radii $r_j\leq  \rho/2$, satisfying $\sum_j r_j^s \leq \mathcal{H}^s(E)+1$.  For each ball $B_j$ intersecting $E_{\rho}$, let $x_j\in B_j \cap E_{\rho}$.  Then $E_{\rho}\subset \cup_j B(x_j, 2r_j)$, and we have $$\mu(E_{\rho}) \leq \sum_j\mu(B(x_j, 2 r_j)) \leq 4^s \Phi^{-1}(\varepsilon)\sum_j r_j^s \leq 4^s \Phi^{-1}(\varepsilon)(\mathcal{H}^s(E)+1).$$

 To obtain a contradiction, we claim that the sets $E_{\rho}$ increase to exhaust $E$ as $\rho \rightarrow 0^+$.  Assuming this, we have $\mu(E) \leq 4^{s} \Phi^{-1}(\varepsilon)(\mathcal{H}^s(E)+1)$, and the right hand side of this inequality can be chosen to be less than $\mu(E)$ for small enough $\varepsilon$, which is absurd.

To prove the claim, let $x\in E$.  Since $\mathcal{W}_{\Phi,s}(\mu)(x)\leq 1$, there exists $\rho>0$ small enough so that $\int_0^{2\rho}\Phi\bigl(\frac{\mu(B(x,r))}{r^s}\bigl)\frac{dr}{r}\leq \varepsilon \log 2$.  Then we have $\log2 \cdot\Phi(\tfrac{\mu(B(x,r))}{2^sr^s})\leq \varepsilon\log 2 $ for any $r<\rho$.  Hence $\tfrac{\mu(B(x,r))}{r^s}\leq 2^s\Phi^{-1}(\varepsilon)$ for any $r<\rho$, and therefore $x\in E_{\rho}$.
\end{proof}

The next result we shall require is an elementary maximum principle for general potentials.

\begin{lem}[Maximum Principle]  For a nonnegative measure $\mu$, denote $\tilde{\mu} =2^{-s}\mu$.  For $A>0$, suppose that $\mathcal{W}_{\Phi,s}(\mu)(x)\leq A$ for all $x\in\operatorname{supp}(\mu)$.  Then
$\mathcal{W}_{\Phi,s}(\tilde{\mu})(x)\leq A$ for all $x\in \mathbf{R}^d$.
\end{lem}

\begin{proof}  Let $\delta>0$. Suppose $x\not\in \operatorname{supp}(\mu)$.  Put $d= \operatorname{dist}(x, \operatorname{supp}(\mu))$.  Then we have $\mathcal{W}_{\Phi,s}(\mu)(x) = \int_d^{\infty}\Phi\bigl(\frac{\mu(B(x,r))}{r^s}\bigl)\frac{dr}{r}.$  Let $z\in \text{supp}(\mu)$ be such that $|x-z|< d+\delta$. Note that $B(x,r)\subset B(z,2r)$, for any $r>d+\delta$.  Hence, we see that
$$\int_{d+\delta}^{\infty}\Phi\Big(\frac{\mu(B(x,r))}{2^sr^s}\Bigl)\frac{dr}{r} \leq \int_{0}^{\infty}\Phi\Bigl(\frac{\mu(B(z,2r))}{(2r)^s}\Bigl)\frac{dr}{r}\leq A.
$$
Since $\delta>0$ was arbitrary, it follows that $\mathcal{W}_{\Phi,s}(\tilde\mu)(x)\leq A$.
\end{proof}

We shall now work with a specific $\Phi$-capacity.  Fix $\beta = \beta(s,d)$ satisfying $\beta >1/\alpha$, with $\alpha>0$ the constant of Theorem \ref{thm1}.  Now define $\Phi(t)=e^{-1/t^{\beta}}$.  A simple consequence of Theorem \ref{thm1} is that $\mathcal{W}_{\Phi,s}(\mu)$ is finite $\mu$ almost everywhere.

%Recall that $\mathcal{L}$ is the measure on $(0,\infty)$ defined by $\mathcal{L}(E) = \int_E \frac{dr}{r}$, for $E\subset (0,\infty)$.

\begin{prop}\label{aepot}Suppose $||R(\mu)||_{L^{\infty}}\leq 1$.  Then for each $\varepsilon>0$, there exists $A_{\varepsilon}>0$ depending on $\varepsilon$, $s$, and $d$, such that
\begin{equation}\label{aepotst}
\mu\bigl(\bigl\{x\in \mathbf{R}^d:\mathcal{W}_{\Phi,s}(\mu)(x)>A_{\varepsilon}\bigl\}\bigl)\leq \varepsilon \mu(\mathbf{R}^d).
\end{equation}
\end{prop}

\begin{proof}  Consider the exceptional set $F$ defined by
$$F=\!\bigcup_{k\in \mathbf{Z}_+}\bigl\{x\in \mathbf{R}^d: \,\mathcal{L}\bigl(\bigl\{r\in (0,\infty) :\frac{\mu(B(x,r))}{r^s}>2^{-k}\bigl\}\bigl)>T_k \bigl\},
$$
with $T_k>0$ to be chosen momentarily.  For each $k$, we apply Theorem \ref{thm1} with $\Delta=2^{-k}$ and $T=T_k$.  This yields
$$\mu(F)\leq C\sum_{k\in \mathbf{Z}_+}\frac{2^k}{\log^{\alpha}T_k}\mu(\mathbf{R}^d).$$
Now let $\varepsilon>0$, and suppose $\beta'=\beta'(s,d)$ satisfies $\beta>\beta'>1/\alpha$.  For a large constant $\widetilde{C}>0$, we put $T_k = \exp(\widetilde{C}\varepsilon^{-1/\alpha} 2^{k\beta'}).$

 If $\widetilde{C}>0$ is large enough (in terms of $s$, $\beta$ and $d$), we see that $\mu(F)\leq \varepsilon \mu(\mathbf{R}^d)$.

For all $x\in \mathbf{R}^d\backslash F$, we have $$\mathcal{L}\bigl(\{r\in (0,\infty) :\frac{\mu(B(x,r))}{r^s}>2^{-k}\bigl\}\bigl)\leq T_{k} \text{ for all }k\in \mathbf{Z}_+.$$

Now note that $\mathcal{W}_{\Phi,s}(\mu)(x)$ can be estimated from above by
\begin{equation}\begin{split}\nonumber\sum_{k\in \mathbf{Z}_+}\max_{2^{-(k+1)}\leq t \leq 2^{-k}}&\Phi(t) \cdot \mathcal{L}\bigl(\{r\in (0,\infty) :2^{-k}\geq \frac{\mu(B(x,r))}{r^s}>2^{-k-1}\bigl\}\bigl)\\
&+\sup_{t\geq 1}\Phi(t)\cdot \mathcal{L}\bigl(\{r\in (0,\infty) : \frac{\mu(B(x,r))}{r^s}>1\bigl\}\bigl).
\end{split}\end{equation}
Since $\Phi(t)$ is increasing, $\max_{2^{-(k+1)}\leq t \leq 2^{-k}}\Phi(t) =\exp(-2^{k\beta}),$
and hence for $x\in \mathbf{R}^d\backslash F$ we have
$$\mathcal{W}_{\Phi,s}(\mu)(x)\leq T_0 + \sum_{k\in \mathbf{Z}_+} \exp(-2^{k\beta})\cdot T_k.
$$

Since $\beta>\beta'$, this sum on the right hand side converges, and therefore
%$$\sum_{k\in \mathbf{Z}}\beta 2^{-k(\beta+1)}e^{-1/2^{\beta k}}\mathcal{L}\Bigl(\Bigl\{r\in (0,\infty) :\frac{\mu(B(x,r))}{r^s}>2^{k-1}\Bigl\}\Bigl)\leq A_{\varepsilon},
%$$
%for some constant $A_{\varepsilon}>0$.  On the other hand, a standard application of Fubini's theorem yields
%$$\mathcal{W}_{\Phi,s}(\mu)(x) \!= \!\!\int_0^{\infty}\!\!\!\beta \Delta^{-\beta-1}e^{-1/\Delta^{\beta}}\mathcal{L}\Bigl(\Bigl\{r\in (0,\infty) :\frac{\mu(B(x,r))}{r^s}>\Delta\Bigl\}\Bigl)d\Delta.$$
 $\mathcal{W}_{\Phi,s}(\mu)(x)\leq A_{\varepsilon}$ for all $x\in \mathbf{R}^d\backslash F$.
\end{proof}

The $s$-dimensional Cald\'{e}ron-Zygmund capacity of a compact set $E$ is defined by
\begin{equation}\label{czcapdef}\nonumber
\gamma_s(E) = \sup\{\mu(E):\mu\in \mathcal{M}^+(\mathbf{R}^d),\, \operatorname{supp}(\mu)\subset E,\, ||R(\mu)||_{L^{\infty}}\leq 1\bigl\}.\end{equation}
 %where  $\displaystyle \mathcal{E} = \bigl\{\mu\in \mathcal{M}^+(\mathbf{R}^d): \operatorname{supp}(\mu)\subset E,\, ||R(\mu)||_{L^{\infty}}\leq 1\bigl\}.$
A well-known conjecture is the following (see \cite{ENV10,Tol11}):
\begin{thmb}Suppose that $d\geq 2$ and $0<s<d$, $s\not\in \mathbf{N}$.  There exist positive constants $A_1$ and $A_2$, depending on $s$ and $d$, such that for every compact set $E\subset\mathbf{R}^d$,
\begin{equation}\label{capconj}A_1\operatorname{cap}_{\Phi,s}(E)\leq \gamma_s(E) \leq A_2\operatorname{cap}_{\Phi,s}(E),
\end{equation}
with $\Phi(t)=t^2$.\end{thmb}

In the literature, the capacity $\text{cap}_{\Phi,s}(E)$, with $\Phi(t)=t^2$ is often denoted by $\text{cap}_{\frac{2}{3}(d-s), \frac{3}{2}}(E)$, see for example \cite{AH96}.  %While this notation is very useful is some contexts  (see for example \cite{AH96}), it is not particularly instructive here.  This is due to the fact that when dealing with $\gamma$, one is looking at the size of the singular set of \textit{the gradient} of a fractional subharmonic function.  The standard potential theoretic notation is most useful in studying the singular set of fractional subharmonic functions themselves (rather than their gradients).

The conjecture above has been proven for $s\in (0,1)$ by Mateu, Prat and Verdera \cite{MPV05}.  Recently, an analogue has been proven for $s=0$ by Adams and Eiderman \cite{AE12}.  Both of these papers use curvature methods in order to prove their results, a technique which appears absent when $s>1$, see \cite{Far99}.  Any such estimate is false for integer $s$, which can be seen by considering a smooth $s$-dimensional submanifold $E\subset \mathbf{R}^d$, with $\mathcal{H}^s(E)<\infty$. (Here $\gamma_s(E)>0$, but $\text{cap}_{\Phi,s}(E)=0$.)

In \cite{ENV10},  a symmetrization of the kernel in the Riesz transform is used to obtain the lower bound in (\ref{capconj}) for all $0<s<d$.  It is therefore the upper bound which remains open.

Now suppose $s\in (d-1,d)$.  Using Proposition \ref{aepot}, we will see that the upper bound in (\ref{capconj}) holds if one replaces $\Phi(t)=t^2$ with the potential function $\Phi(t) = e^{-1/t^{\beta}}$.  Although a long way from the optimal result, it appears to be the first such bound outside the curvature range.

\begin{prop}\label{nonlinczineq}  Suppose $s\in (d-1,d)$, and $\Phi(t) = e^{-1/t^{\beta}}$.  There is a constant $C>0$ such that
\begin{equation}
\gamma_s(E)\leq C\operatorname{cap}_{\Phi, s}(E) \text{ for all compact sets }E\subset \mathbf{R}^d.
\end{equation}
\end{prop}

\begin{proof} Suppose $\gamma_s(E)=t>0$, since otherwise the inequality is trivial.  There exists a measure $\mu$ supported on $E$ such that $\mu(E)\geq 3t/4$ and $||R(\mu)||_{L^{\infty}}\leq1$.  By Proposition \ref{aepot}, there exists $A>1$, depending on $s$ and $d$, such that
$\mu\bigl(\bigl\{x\in \mathbf{R}^d:\mathcal{W}_{\Phi,s}(\mu)(x)>A\bigl\}\bigl)\leq \frac{\mu(\mathbf{R}^d)}{4}.$

Define now $\widetilde{E} = \bigl\{x\in E:\mathcal{W}_{\Phi,s}(\mu)(x)\leq A\bigl\},$
and let $\omega = \chi_{\widetilde{E}}d\mu$.  Then $\omega(E)\geq t/2$, and $\mathcal{W}_{\Phi,s}(\omega)(x)\leq A$ for all $x\in \operatorname{supp}(\omega)$.  If $\tilde\omega = 2^{-s}\omega$, then the maximum principle implies that $\mathcal{W}_{\Phi,s}(\tilde\omega)(x)\leq A$ for all $x\in \mathbf{R}^d$.  Hence $\text{cap}^{(A)}_{\Phi,s}(E) \geq \tfrac{t}{3^s2}$, and applying Lemma \ref{cutoffsame} completes the proof.
%Since $\Phi$ is convex in the interval $(0,((\beta+1)/\beta)^{1/\beta})$,  the measure $\omega/C_1T$ is admissible for the set $\mathcal{E}$, and therefore $\text{cap}_{\Phi, s}(E) \geq \frac{t}{CT}.$
%Since $T$ depends solely on $s$ and $d$, the result follows.
\end{proof}

Let $s\in (d-1,d)$, and let $E\subset \mathbf{R}^d$ be a compact set with $\mathcal{H}^s(E)<\infty$.  An immediate consequence of Lemma \ref{Hauszero} and Proposition \ref{nonlinczineq} is that $\gamma_s(E)=0$. This result is essentially equivalent to the main theorem of \cite{ENV11}.

%The corollary follows immediately from the fact that if $\mathcal{H}^s(E)<\infty$, then $\text{cap}_{\Phi, s}(E)=0$.  To see this, suppose $\mathcal{H}^s(E)<\infty$, but $\text{cap}_{\Phi, s}(E)>0$.  Then there exists a measure $\mu$ with $%\mu(E)>0$ and $\mathcal{W}_{\Phi,s}(\mu)(x)\leq 1$ for all $x\in E$.  Therefore, for each $\varepsilon>0$, there is a set $F\subset E$, with $\mu(F)\geq \mu(E)/2$, and  $k_0\in \mathbf{N}$, such that $$\sum_{\substack{Q\in \mathcal{D},%\,x\in Q,\\\ell(Q)\leq 2^{-k_0}}}\Phi\Bigl(\frac{\mu(Q)}{\ell(Q)^s}\Bigl) \leq \varepsilon, \text{ for all }x\in F.$$
%Here $\mathcal{D}$ is the standard lattice of half-open dyadic cubes in $\mathbf{R}^d$.  In particular, $\mu(Q)\leq (1/\ln(1/\varepsilon))^{1/\beta} \ell(Q)^s$ whenever $Q\in \mathcal{D}$ contains a point of $F$, and $\ell(Q)\leq 2^{-%k_0}$.

%On the other hand, since $\mathcal{H}^s(E)<\infty$, there exists a pairwise disjoint collection of dyadic cubes $Q_j\in \mathcal{D}$, satisfying the following properties: $\ell(Q_j)\leq 2^{-k_0}$ for each $j$, $E\subset\cup_j Q_j$, and $%\sum_j \ell(Q_j)^s \leq C(d)(\mathcal{H}^s(E) +1)$. As a result, we obtain the following inequalities:
%\begin{equation}\begin{split}\nonumber\mu(F)\leq \sum_{j : Q_j \cap F\neq \varnothing} \mu(Q_j) &\leq \Bigl(\frac{1}{\ln(1/\varepsilon)}\Bigl)^{\frac{1}{\beta}}\sum_j \ell(Q_j)^s \\
%&\leq C(d)\Bigl(\frac{1}{\ln(1/\varepsilon)}\Bigl)^{\frac{1}{\beta}}(\mathcal{H}^s(E)+1).
%\end{split}\end{equation}
%Hence we obtain a contradiction with $\mathcal{H}^s(E)<\infty$ provided $\varepsilon$ is chosen small enough.

 \end{document}